\documentclass[leqno, a4paper]{amsart}
\usepackage{amsmath, amssymb, amsthm, a4wide,  mathrsfs, mathabx}
\usepackage[usenames]{color}
\usepackage{eepic,epic}

\newtheorem{theorem}{Theorem}[section]

\newtheorem{corollary}[theorem]{Corollary}
\newtheorem{lemma}[theorem]{Lemma}
\newtheorem{proposition}[theorem]{Proposition}

\theoremstyle{definition}
\newtheorem{definition}[theorem]{Definition}

\theoremstyle{remark}
\newtheorem{remark}[theorem]{Remark}
\newtheorem{example}[theorem]{Example}

\newtheorem*{claim}{Claim}
\newtheorem*{remark*}{Remark}

\numberwithin{equation}{section}

\newcommand{\op}{\operatorname}

\newcommand{\brak}[1]{\langle #1 \rangle}

\newcommand{\bH}{\mathbb{H}}
\newcommand{\N}{\mathbb{N}}

\newcommand{\R}{\mathbb{R}}

\newcommand{\Z}{\mathbb{Z}}

\newcommand{\cG}{\mathcal{G}}

\newcommand{\cL}{\mathcal{L}}
\newcommand{\cU}{\mathcal{U}}
\newcommand{\cV}{\mathcal{V}}

\newcommand{\cX}{\mathcal{X}}

\newcommand{\fg}{\mathfrak{g}}

\newcommand{\rk}{\op{rk}}
\newcommand{\ran}{\op{ran}}
\newcommand{\dom}{\op{dom}}

\newcommand{\ad}{\op{ad}}

\newcommand{\Ow}{\op{O}_{w}}

\begin{document}

\title{Tangent Maps and Tangent Groupoid for Carnot Manifolds}

\subjclass[2000]{Primary}

\author{Woocheol Choi}
\address{Department of Mathematics Education, Incheon National University, Incheon, South Korea}
\email{choiwc@inu.ac.kr}

 \author{Rapha\"el Ponge}
\address{Department of Mathematical Sciences, Seoul National University, Seoul, South Korea}
 \email{ponge.snu@gmail.com}

 \thanks{WC was partially supported by POSCO TJ Park Foundation. RP\ was partially supported by Research Resettlement Fund and Foreign Faculty Research Fund of Seoul National University, and  Basic Research grants 2013R1A1A2008802 and 2016R1D1A1B01015971 of National Research Foundation of Korea.}

\begin{abstract}
This paper studies the infinitesimal structure of Carnot manifolds. By a Carnot manifold we mean a manifold together with a subbundle filtration of its tangent bundle which is compatible with the Lie bracket of vector fields. We introduce a notion of differential, called Carnot differential, for Carnot manifolds maps (i.e., maps that are compatible with the Carnot manifold structure). This differential is obtained as  a group map between the corresponding tangent groups. We prove that, at every point, a Carnot manifold map is osculated in a very precise way by its Carnot differential at the point. We also show that, in the case of maps between nilpotent graded groups, the Carnot differential is given by the Pansu derivative. Therefore, the Carnot differential is the natural generalization of the Pansu derivative to maps between general Carnot manifolds. Another main result is a construction of an analogue for Carnot manifolds of Connes' tangent groupoid. Given any Carnot manifold $(M,H)$ we get a smooth groupoid that encodes the smooth deformation of the pair $M\times M$ to the tangent group bundle $GM$. This shows that, at every point, the tangent group  is the tangent space in a true differential-geometric fashion. Moreover, the very fact that we have a groupoid accounts for the group structure of the tangent group. Incidentally, this answers a well-known question of Bella\"iche~\cite{Be:Tangent}. 
\end{abstract}

\maketitle

 \section{Introduction}
 The focus of this paper is on the infinitesimal structure of Carnot manifolds. By a Carnot manifold we mean a manifold $M$ together with a filtration of subbundles,
\begin{equation}
   H_{1}\subset H_{2}\subset \cdots \subset H_{r}=TM,
   \label{eq:Intro.Carnot-filtration}
\end{equation}which is compatible with the Lie bracket of vector fields. 
We refer to Section~\ref{sec:Carnot-Manifolds}, and the references therein, for various examples of Carnot manifolds. 
Many of those examples are equiregular Carnot-Carath\'eodory manifolds, in which case the filtration~(\ref{eq:Intro.Carnot-filtration}) arises from the iterated  Lie bracket of sections of $H_1$. 
However,  even for studying equiregular (and even non-regular) Carnot-Carath\'eodory structures we may be naturally led to consider non bracket-generated Carnot filtrations (see Section~\ref{sec:Carnot-Manifolds} on this point).
 
There is a general understanding that (graded) nilpotent Lie groups are models for Carnot manifolds. 
From an algebraic perspective, any filtration~(\ref{eq:Intro.Carnot-filtration}) gives rise to a graded vector bundle $\fg M:=\fg_1M\oplus \cdots \oplus \fg_r M$, where $\fg_wM=H_{w}/H_{w-1}$. As a vector bundle $\fg M$ is (locally) isomorphic to the tangent bundle $TM$. Moreover, as observed by Tanaka~\cite{Ta:JMKU70}, the Lie bracket of vector fields induces on each fiber $\fg M(a)$, $a\in M$, a Lie algebra bracket. This turns $\fg M(a)$ into a graded nilpotent Lie algebra. Equipping it with its Dynkin product we obtain a graded nilpotent group $GM(a)$, which is called the tangent group at the point $a\in M$.  
 
There is an alternative construction of a graded nilpotent group that is meant to approximate the Carnot manifold $(M,H)$ at the point $a$. This construction originated from the work of Folland-Stein~\cite{FS:CMPAM74}, Rothschild-Stein~\cite{RS:ActaMath76}, and others on hypoelliptic PDEs. In this approach  we obtain a graded nilpotent Lie group $G^{(a)}$ out of ``anisotropic asymptotic expansions" of vector fields in privileged coordinates at the point $a\in M$. The graded nilpotent Lie group $G^{(a)}$ equipped with its left-invariant Carnot manifold structure is often called the nilpotent approximation of $(M,H)$ at the point $a\in M$. Note also that this construction is very sensitive to the choice of the privileged coordinates. 
I 
In~\cite{CP:Carnot} the two approaches are somehow reconciled by singling out a special class of privileged coordinates, in which the nilpotent approximation is given by the tangent group $GM(a)$. These coordinates are called Carnot coordinates. They are an essential ingredient of this paper.  Examples of Carnot coordinates include Darboux coordinates on a contact manifold, and the canonical coordinates of the first kind of Goodman~\cite{Go:LNM76} and Rothschild-Stein~\cite{RS:ActaMath76} on general Carnot manifolds. The polynomial privileged coordinates of Bella\"iche~\cite{Be:Tangent} are not Carnot coordinates in general. However, they can be converted into Carnot coordinates by means of a unique homogeneous polynomial change of coordinates. We call the resulting coordinates $\varepsilon$-Carnot coordinates. An important property of these coordinates is the fact that they are osculated in very precise manner by the group law of the tangent group (see~\cite{CP:Carnot}).  

As graded nilpotent Lie groups provide us with the natural notion of tangent space for Carnot manifolds, it is natural to seek for a notion of tangent map for maps between Carnot manifolds in terms of group maps between the corresponding tangent groups. Pansu~\cite{Pa:AM89} introduced a notion of derivative (called Pansu derivative) for maps between graded nilpotent groups. The Pansu derivative is a group map  between the corresponding source and target groups. Pansu~\cite{Pa:AM89} also proved that its derivative exists almost everywhere for Lipschitz continuous maps and quasi-conformal maps between Carnot groups equipped with their left-invariant Carath\'eodory metrics. This result was extended to quasi-conformal maps between equiregular Carnot-Carath\'eodory manifolds by Margulis-Mostow~\cite{MM:GAFA95}  in terms of the corresponding Carath\'eodory metrics. In~\cite{MM:GAFA95} differentiability is defined in terms of the corresponding Carath\'eodory metrics. Moreover, the differential is defined as a group map between the corresponding tangent cones which are defined in terms of metric equivalences of rectifiable paths. 

In this paper we take a different point of view, since Carath\'eodory metrics are not available for general Carnot manifolds. We consider Carnot manifold maps between step~$r$ Carnot manifolds $(M,H)$ to $(M',H')$.  By this we mean a smooth map $\phi:M\rightarrow M'$ whose differential is compatible with the Carnot filtrations. This condition implies that, at every point $a\in M$,  the differential $\phi'(a)$ induces a graded linear map $\hat{\phi}'(a):\fg M(a)\rightarrow \fg M'(a')$. We show that this is a Lie algebra map, and hence we get a Lie group map $\hat{\phi}'(a):GM(a) \rightarrow GM'(\phi(a))$ (Proposition~\ref{prop:Carnot-map.group-map}). We call it the Carnot differential of $\phi$ at the point $a\in M$.  Carnot differentials satisfy the chain rule (Proposition~\ref{prop:Carnot.product-tangent-map}). Incidentally, the assignment $(M,H)\rightarrow GM$ is a functor from Carnot manifolds to graded nilpotent Lie groups (see Proposition~\ref{prop:Carnot-diff.functor}).   

Our main result on Carnot differentials asserts that, in Carnot coordinates, a Carnot manifold map is osculated at every point in a very precise manner by its Carnot differential (Theorem~\ref{thm:Carnot-prop.tangent-map-approx}). Note it is crucial to work in Carnot coordinates to obtain an approximation by a group map (see Remark~\ref{rmk:Tangent-map.need-Carnot-coord} on this point). One consequence of this result is an explicit and simple action of Carnot diffeomorphisms on Carnot coordinates (see Proposition~\ref{prop:Carnot-approx.action-Carnot-coordinates} for the precise statement).  Furthermore, this osculation result provides us with a direct link with the Pansu derivative. More precisely, in the case of Carnot manifold maps between graded nilpotent Lie groups, the Carnot differential is naturally identified with the Pansu derivative (Theorem~\ref{thm:Pansu.Carnot-Pansu}). Therefore, the Carnot differential is the natural generalization of the Pansu derivative for maps between general Carnot manifolds.

In~\cite{Co:NCG} Connes constructed the tangent groupoid of a manifold $M$ as the smooth groupoid that encodes the smooth deformation of $M\times M$ to the tangent bundle $TM$. Bella\"iche~\cite{Be:Tangent} conjectured the existence of an analogue of the tangent groupoid for (equiregular) Carnot-Carath\'eodory manifolds and this could explain why the tangent space should be a group.  We know that for (equiregular) Carnot-Carath\'eodory manifolds $GM(x_0)$ is a tangent space in the metric sense introduced by Gromov~\cite{GLP:Cedic-Nathan81} (see~\cite{Be:Tangent, Mi:JDG85}). However, in this point of view the tangent space is only determined up to homeomorphisms, and so this does not account for the group structure of $GM(x_0)$. Furthermore, the approach is not applicable to general Carnot manifolds, for which we don't have Carath\'eodory metrics. 

As a manifold, Connes' tangent  groupoid $\cG M$ is obtained by glueing together the manifolds $M\times M\times \R^*$ and $TM$. The differentiable structure near $TM$ is such that, for all $x_0 \in M$ and $t\in \R^*$, the submanifold $\{x_0\}\times M \times \{t\}$ is near $(x_0,x_0,t)$  a zoomed in version of a neighborhood of $x_0$ in $M$, where the zooming is performed by rescaling by $t^{-1}$ along directions out of $x_0$ in local coordinates. Thus, this construction recasts in the language of groupoids the well known idea that, as we zoom in more and more around $x_0$, the manifold $M$ looks more and more like the tangent space $TM(x_0)$. In addition, the very fact that we obtain a groupoid accounts for the additive group structure of $TM(x_0)$. We refer to Section~\ref{sec:Lie-groupoids} for more details on groupoids and a review of the tangent groupoid's construction. 

Thanks to our results on Carnot differentials and the properties of $\varepsilon$-Carnot coordinates, we show how to associate with any Carnot manifold $(M,H)$ a smooth groupoid $\cG_H M$ that encodes the smooth deformation of $M\times M$ to the group bundle $GM$ (Theorem~\ref{thm:Groupoid.main}). We also show that this construction is functorial (Corollary~\ref{cor:Groupoid.functoriality}). 
As a manifold, $\cG_H M$ is obtained by glueing the manifolds $M\times M\times \R^*$ and $GM$. As with Connes' tangent groupoid, the differentiable structure near $GM$ is such that, for all $x_0 \in M$ and $t\in \R^*$, the submanifold $\{x_0\}\times M \times \{t\}$ is near $(x_0,x_0,t)$ a zoomed in version of a neighborhood of $x_0$. 
However, in this case the zooming is performed by an anistropic rescaling in $\varepsilon$-coordinates at $x_0$, where each direction out of $x_0$ is weighted according to the minimal space of the filtration $(H_1(x_0), \ldots, H_r(x_0))$ it starts out.  
This shows that $GM (x_0)$ is tangent to $(M,H)$ at $x_0$ in a differential-geometric fashion. Moreover, as we obtain a groupoid this further accounts for the occurrence of the group structure of $GM(x_0)$. We thus obtain a positive answer to Bella\"iche's conjecture. 

The main motivation of Connes~\cite{Co:NCG} for the construction of the tangent groupoid of a manifold was a new $K$-theoretic proof of the full index theorem of Atiyah-Singer~\cite{AS:AM68}. This approach to the index theorem was extended to hypoelliptic operators on contact manifolds by van Erp~\cite{vE:AM10.Part1, vE:AM10.Part2} (see also~\cite{BvE:ActaM14}) by using an analogue for contact manifolds of Connes' tangent groupoid. Therefore, it is expected that the Carnot groupoid for Carnot manifolds should play an important role in the reformulation of the index theorem for hypoelliptic pseudodifferential operators on Carnot manifolds, at least from a $K$-theoretic perspective. 

In addition, there are closed ties between Lie groupoids and pseudodifferential calculi (see, e.g., \cite{DS:AdvM14, LV:AIM17, vEY:Crelle17}). Therefore, the tangent groupoid for Carnot manifolds is an important tool for understanding the hypoelliptic pseudodifferential calculus on Carnot manifolds (see~\cite{vEY:Crelle17} on this point) on this point. Furthermore, the properties of Carnot coordinates, $\varepsilon$-Carnot coordinates, and Carnot differentials that are involved in the construction of the tangent groupoid for Carnot manifolds are important ingredients in the construction of a full symbolic calculus for the hypoelliptic pseudodifferential calculus on Carnot manifolds in~\cite{CP:PsiDOs}. 

This paper is organized as follows. In Section~\ref{sec:Carnot-Manifolds}, we review the main definitions and examples regarding Carnot manifolds and their tangent group bundles. In Section~\ref{sec:anisotropic}, we review the results of~\cite{CP:Privileged} on anisotropic asymptotic expansions of maps and vector fields. In Section~\ref{sec:Carnot-coordinates}, we recall the main definitions and properties of the Carnot coordinates. In Section~\ref{sec:Carnot-Differential}, we define the Carnot differential of a Carnot manifold map and show this is a group map. 
 In Section~\ref{sec:Tangent-approximation}, we show that, in Carnot coordinates, a Carnot manifold map is osculated in a very precise manner by its Carnot differential. In Section~\ref{sec:Pansu-derivative}, we show that in the setting of graded nilpotent Lie groups the Carnot differential agrees with the Pansu derivative. In Section~\ref{sec:Lie-groupoids}, we review some definitions and properties of Lie groupoids with a special attention to Connes' tangent groupoid of a manifold. Finally, in Section~\ref{sec:tangent-groupoid} we construct the tangent groupoid of a Carnot manifold. 
 
 \begin{remark*}
 The first version of this article was posted as a preprint on the arXiv server in October 2015. Since then, several alternative approaches to the tangent groupoid of a Carnot manifold have come out (see~\cite{HS:DM18, Mo:arXiv18, vEY:BLMS17}).  
\end{remark*}

\subsection*{Acknowledgements}
The authors wish to thank Andrei Agrachev, Davide Barilari, Enrico Le Donne, and Fr\'ed\'eric Jean for useful discussions related to
the subject matter of this paper. They also thank anonymous referees whose insightful comments greatly help improving the presentation of the paper. In addition, they would like to thank Henri Poincar\'e Institute, McGill University, National University of Singapore,  and University of California at Berkeley for their hospitality during the 
preparation of this paper.  

\section{Carnot Manifolds}\label{sec:Carnot-Manifolds}
In this section, we briefly review the main facts about Carnot manifolds. The exposition follows~\cite{CP:Privileged} closely. We refer to~\cite{CP:Privileged}  and the references therein for a more complete account. 

\subsection{Carnot Groups and Graded nilpotent Lie groups} 
In what follows, by the Lie algebra of a Lie group we shall mean the tangent space at the unit element equipped with Lie bracket induced by the Lie bracket of left-invariant vector fields. 

\begin{definition} \label{def:Carnot.Carnot-algebra}
A \emph{step $r$ nilpotent graded Lie algebra} is the data of a real Lie algebra $(\fg,[\cdot, \cdot])$ and a grading $  \fg=\fg_{1}\oplus \fg_{2}\oplus \cdots \oplus \fg_{r}$, which is compatible with the Lie bracket, i.e., 
\begin{equation}
        [\fg_{w},\fg_{w'}]\subset \fg_{w+w'} \ 
        \text{for $w+w'\leq r$} \quad \text{and} \quad [\fg_{w},\fg_{w'}]=\{0\}\ 
        \text{for $w+w'> r$}.
         \label{eq:Carnot.grading-bracket}
\end{equation}
We further say that $\fg$ is a \emph{Carnot algebra} when $\fg_{w+1}=[\fg_{1},\fg_{w}]$ for $w=1,\ldots, r-1$. 
A \emph{graded nilpotent Lie group} (resp., \emph{Carnot group}) is a connected  and simply connected nilpotent real Lie group whose Lie algebra is a  graded nilpotent  Lie algebra (resp., Carnot algebra). 
\end{definition}

%
%
%
%
%


Let $G$ be a step $r$ graded nilpotent Lie group with unit $e$. Its Lie algebra $\fg=T_eG$ comes equipped with a grading $  \fg=\fg_{1}\oplus \fg_{2}\oplus \cdots \oplus \fg_{r}$, which is compatible with its Lie algebra bracket. This grading  gives rise to a family of anisotropic dilations $\xi \rightarrow t\cdot \xi$, $t \in \R\setminus 0$, which are automorphisms of $\fg$ given by
\begin{equation}
 t\cdot (\xi_1+\xi_2 +\cdots + \xi_r) = t\xi_1+ t^2\xi_2 +\cdots + t^r \xi_r, \qquad \xi_j \in \fg_j.  
 \label{eq:Carnot.dilations}
\end{equation}
In addition, $\fg$ is canonically isomorphic to the Lie algebra of left-invariant vector fields on $G$: to any $\xi \in \fg$ corresponds the unique left-invariant vector field $X_\xi$ on $G$ such that $X_\xi(e)=\xi$. The flow $\exp(tX_\xi)$ exists for all times $t\in \R$, 
and so we have a globally defined exponential map $\exp:\fg \rightarrow G$ which is a diffeomorphism given by
\begin{equation*}
\exp(\xi)=\exp(X_\xi), \qquad \text{where $\exp(X_\xi):=\exp(tX_\xi)(e)_{|t=1}$}. 
\end{equation*}

For $\xi \in \fg$, let $\ad_\xi: \fg\rightarrow \fg$ be the adjoint endomorphism of $\xi$, i.e., $\ad_\xi \eta =[\xi,\eta]$ for all $\eta \in \fg$. This is a nilpotent endomorphism.  
By  the Baker-Campbell-Hausdorff formula we have 
\begin{equation}
\exp(\xi) \exp(\eta)= \exp(\xi \cdot \eta) \qquad \text{for all $\xi,\eta \in \fg$},
\label{eq:Carnot.BCH-Formula}
\end{equation}where $\xi \cdot \eta$ is given by the Dynkin product, 
\begin{align}
 \xi \cdot \eta & = \sum_{n\geq 1} \frac{(-1)^{n+1}}{n} \sum_{\substack{\alpha, \beta \in \N_0^n\\ \alpha_j+\beta_j\geq 1}} 
 \frac{(|\alpha|+|\beta|)^{-1}}{\alpha!\beta!} (\ad_\xi)^{\alpha_1} (\ad_\eta)^{\beta_1} \cdots (\ad_\xi)^{\alpha_n} (\ad_\eta)^{\beta_n-1}\eta \nonumber \\
 & = \xi +\eta + \frac{1}{2}[\xi,\eta] + \frac{1}{12} \left( [\xi,[\xi,\eta]]+  [\eta,[\eta,\xi]]\right) - \frac{1}{24}[\eta,[\xi,[\xi,\eta]]] + \cdots .
 \label{eq:Carnot.Dynkin-product}
\end{align}
The above summations are finite, since all the iterated brackets of length~$\geq r+1$ are zero. Any Lie algebra automorphism of $\fg$ then lifts to a Lie group isomorphism of $G$. In particular, the dilations~(\ref{eq:Carnot.dilations}) give rise to Lie group isomorphisms $\delta_t:G\rightarrow G$, $t\in \R\setminus 0$.

Conversely, if $\fg$ is a graded nilpotent Lie algebra, then~(\ref{eq:Carnot.Dynkin-product}) defines a product on $\fg$. This turns $\fg$ into a Lie group with unit $0$. Under the identification $\fg\simeq T_0 \fg$, the corresponding Lie algebra is naturally identified with $\fg$, and so we obtain a graded nilpotent Lie group. Under this identification 
the exponential map becomes the identity  map. Moreover, inversion with respect to the group law~(\ref{eq:Carnot.Dynkin-product}) is given by
\begin{equation}
\xi^{-1}=-\xi \qquad \text{for all $\xi \in \fg$}. 
\label{eq:GM.inverse}
\end{equation}

\subsection{Carnot manifolds} 
In what follows, given distributions $H_{j}\subset TM$, $j=1,2$,  on a manifold $M$, we denote by $[H_{1},H_{2}]$ 
the distribution generated by the Lie brackets of their sections, i.e., 
\begin{equation*}
  [H_{1},H_{2}]=\bigsqcup_{x\in M}\biggl\{ [X_{1},X_{2}](x); \ X_{j}\in C^{\infty}(M,H_{j}), j=1,2\biggr\}.
\end{equation*}

\begin{definition}
    A \emph{Carnot manifold} is a pair $(M,H)$, where $M$ is a manifold and $H=(H_{1},\ldots,H_{r})$ is a finite filtration of 
    subbundles, 
    \begin{equation}
       H_{1}\subset H_{2}\subset \cdots \subset H_{r}=TM,  
        \label{eq:Carnot.Carnot-filtration}
    \end{equation}which is compatible with the Lie bracket of vector fields, i.e., 
    \begin{equation}
        [H_{w},H_{w'}]\subset H_{w+w'}\qquad 
        \text{for $w+w'\leq r$}.
        \label{eq:Carnot-mflds.bracket-condition}
    \end{equation}
    The number $r$ is called the \emph{step} of the Carnot manifold $(M,H)$. The sequence $(\rk H_{1},\ldots, \rk H_{r})$ is called its \emph{type}.  
\end{definition}

\begin{remark}
    Carnot manifolds are also called filtered manifolds (see, e.g.,~\cite{CS:AMS09, Mo:HMJ93}). 
\end{remark}

\begin{definition}\label{def:Carnot.Carnot-mfld-map}
    Let $(M,H)$ and $(M',H')$ be Carnot manifolds of step $r$, where $H=(H_{1},\ldots,H_{r})$ and 
    $H'=(H_{1}',\ldots,H_{r}')$. Then
    \begin{enumerate}
        \item  A \emph{Carnot manifold map} $\phi:M\rightarrow M'$ is smooth map such that, for $j=1,\ldots,r$, 
        \begin{equation}
            \phi'(x)X\in H_{j}'(\phi(x)) \qquad \text{for all $(x,X)\in H_{j}$}.
            \label{eq:Carnot.Carnot-map}
        \end{equation}
    
        \item  A \emph{Carnot diffeomorphism} $\phi:M\rightarrow M'$ is a diffeomorphism such that $\phi$ and $\phi^{-1}$ are both Carnot manifold maps. 
    \end{enumerate}
\end{definition}

\begin{remark}\label{rmk:Carnot-mfld.Carnot-diffeo}
   If $\phi:M\rightarrow M'$ is diffeomorphism, then $\phi$ is a Carnot diffeomorphism if and only if $\phi_{*}H_{j}=H_{j}'$ for $j=1,\ldots,r$. In particular, $(M,H)$ and $(M',H')$ must have same type. Conversely, if $(M,H)$ and $(M',H')$ same type, then a diffeomorphism $\phi:M\rightarrow M'$ is a Carnot diffeomorphism as soon as $\phi$ is a Carnot manifold map. 
\end{remark}

Let $(M^{n},H)$ be an $n$-dimensional Carnot manifold of step $r$, so that $H=(H_{1},\ldots,H_{r})$, 
where the subbundles $H_{j}$ satisfy~(\ref{eq:Carnot.Carnot-filtration}).  

\begin{definition}\label{def:Carnot.weight-sequence}
    The \emph{weight sequence} of a Carnot manifold $(M,H)$ is the sequence $w=(w_{1},\ldots,w_{n})$ defined by
    \begin{equation}
    w_{j}=\min\{w\in \{1,\ldots,r\}; j\leq \rk H_{w}\}.
    \label{eq:Carnot.weight}
\end{equation}
\end{definition}

\begin{remark}
    Two Carnot manifolds have same type if and only if they have same weight sequence. 
\end{remark}

Throughout this paper we will make use of the following type of tangent frames.  

\begin{definition}  
An \emph{$H$-frame} over an open $U\subset M$ is a tangent frame $(X_{1},\ldots,X_{n})$ over $U$ which is compatible with the 
filtration $(H_{1},\ldots.,H_{r})$ in the sense that, for $w=1,\ldots,r$, the vector fields $X_{j}$, $w_{j}= w$, are sections of $H_{w}$.
\end{definition}


\begin{remark}\label{rmk:Carnot-mfld.brackets-H-frame}
    Let $(X_{1},\ldots,X_{n})$ be an $H$-frame near a point $a\in M$. As explained in~\cite{CP:Privileged},  the condition~(\ref{eq:Carnot-mflds.bracket-condition}) implies that, near $x=a$, there are smooth functions $L_{ij}^{k}(x)$, $w_k\leq w_i+w_j$,  such that, for  $i,j=1,\ldots, n$, we have 
    \begin{equation}
        [X_{i},X_{j}](x) =\sum_{w_{k}\leq w_{i}+w_{j}}L_{ij}^{k}(x)X_{k}(x) \qquad \textup{near $x=a$}. 
        \label{eq:Carnot-mfld.brackets-H-frame}
    \end{equation}   
\end{remark}

\subsection{Examples of Carnot Manifolds}\label{sec:Examples}
\renewcommand{\thesubsubsection}{\Alph{subsubsection}}
We refer to~\cite{CP:Privileged} and the references therein for a detailed description of various examples of Carnot manifolds. Following is a summary of the main examples described in~\cite{CP:Privileged}. 

\subsubsection{Graded nilpotent Lie groups} 
Let $G$ be a step $r$ graded nilpotent Lie group with unit $e$. Then its Lie algebra $\fg=T_eG$ has a grading $ \fg=\fg_{1}\oplus \fg_{2}\oplus \cdots \oplus \fg_{r}$
satisfying~(\ref{eq:Carnot.grading-bracket}). For $w=1,\ldots, r$, let $E_w$ be the $G$-subbundle of $TG$ obtained by left-translation of $\fg_w$ over $G$. We then obtain a vector bundle grading $TG=E_1\oplus \cdots \oplus E_r$. This grading gives rise to the filtration $H_1\subset \cdots \subset H_r=TG$, where $H_w:=H_1 \oplus \cdots \oplus H_w$. It can be shown that $[H_w,H_{w'}]\subset H_{w+w'}$ whenever $w+w'\leq r$ (see, e.g., \cite[Section~3.2]{CP:Carnot}). Therefore, this filtration defines a left-invariant Carnot manifold structure on $G$ which is uniquely determined by the grading of $\fg$. 

\subsubsection{Heisenberg manifolds} They are step 2 Carnot manifolds where $H_1$ is a hyperplane bundle.  Examples include the Heisenberg group $\bH^{2n+1}$, Cauchy-Riemann manifolds of hypersurface type, contact manifolds, even contact manifolds~\cite{Mo:AMS02}, and the confoliations of Elyashberg-Thurston~\cite{ET:C}.
  
\subsubsection{Foliations} They are step 2 Carnot manifolds where $H_1$ is integrable, i.e., $[H_1,H_1]=H_1$. Examples include Kronecker foliation, Reeb foliations, and foliations associated with locally free actions of Lie groups on manifolds (see, e.g., \cite{MM:Cambridge03}). 
  
 \subsubsection{Polycontact manifolds~\emph{(}\cite{Mo:AMS02, vE:Polycontact}\emph{)}} They are step 2 Carnot manifolds that are opposite to foliations inasmuch as $H_1$ is required to be totally non-integrable. This means that, for all $x\in M$, every direction in $(TM/H_1)^*(x)\setminus 0$ is contact, i.e, it defines a symplectic structure on $H_1(x)$ (see~\cite{Mo:AMS02, vE:Polycontact}). 
Beside contact manifolds, the main examples are the groups of Kaplan~\cite{Ka:TAMS80} and M\'etivier~\cite{Me:Duke80}, the principal bundles equipped with the horizontal distribution defined by a fat connection in the sense of Weinstein~\cite{We:Adv80}, the quaternionic contact manifolds of Biquard~\cite{Bi:Roma99, Bi:Asterisque}, and the unit sphere $\mathbb{S}^{4n-1}$ in quaternionic space~\cite{vE:Polycontact}. 

\subsubsection{Pluri-contact manifolds \emph{(}\cite{ACGL:Contact17, ACGL:CR17}\emph{)}} These manifolds  were introduced in the recent preprint~\cite{ACGL:Contact17}. They generalize polycontact manifolds in the sense that they are step 2 Carnot manifold where, for every $x\in M$, it is only required to have at least one contact direction in  $(TM/H_1)^*(x)\setminus 0$. Beside polycontact manifolds, examples  of pluri-contact manifolds that are not polycontact include products of contact manifolds~\cite{ACGL:Contact17} and nondegenerate CR manifolds of non-hypersurface-type~\cite{ACGL:CR17}.

\subsubsection{Equiregular Carnot-Carath\'eodory manifolds} In this case,
\begin{equation*}
H_{j+1}=H_j+[H_1,H_j] \qquad \text{for $j=1,\ldots,r-1$}. 
\end{equation*}
Equiregular Carnot-Carath\'eodory manifolds (a.k.a.~ECC manifolds) naturally appear in sub-Riemannian geometry and control theory associated with 
non-holonomic systems of vector fields. In particular, they occur in a number of  applied and real-life settings 
(see, e.g.,~\cite{ABB:SRG, Be:Tangent, CC:SRG, Gr:CC, Je:ESAIM96, Je:Brief14, Mo:AMS02, Mo:HMJ93, Ri:Brief14}). In addition, any non-equiregular Carnot-Carath\'eodory structure can be desingularized into an equiregular Carnot-Carath\'eodory structure (see, e.g., \cite{Je:Brief14}).

Contact and polycontact manifolds are examples of step 2 ECC manifolds with $r=2$. An example of step 3 ECC manifolds is provided by Engel manifolds, which are  4-dimensional manifolds equipped with a plane-bundle $H\subset TM$ such that $H_2:=H+[H,H]$ has constant rank 3 and $H_2+[H,H_2]=TM$ (see~\cite{Mo:AMS02}). 
More generally, ECC structures naturally appear in the context of parabolic geometry~\cite{CS:AMS09, Mo:AMS02}. Nondegerate Cauchy-Riemann structures and contact quaternionic structures are parabolic geometric structures. Further examples of parabolic geometric structures include the following: 
\begin{itemize}
    \item Contact path geometric structures, including contact projective structures (\emph{cf}.~\cite{Fo:IUMJ05, 
    Fo:arXiv05}). 

    \item  Generic $\left(m,\frac{1}{2}m(m+1)\right)$-distributions, including generic $(3,6)$-distributions studied by 
    Bryant~\cite{Br:RIMS06} (see also~\cite{CS:AMS09}). 
    
    \item Generic $(2,3,5)$-distributions of Cartan~\cite{Ca:AENS10}. 
\end{itemize}

\subsubsection{The heat equation on ECC manifolds.} Let $(M,H)$ be an ECC manifold, so that the filtration 
$H=(H_{1},\ldots,H_{r})$ is generated by the iterative Lie brackets of sections of $H_{1}$. Given any spanning frame $\{X_{1},\ldots,X_{m}\}$ 
of $H_{1}$, the sub-Laplacian $\Delta=-(X_{1}^{2}+\cdots X_{m}^{2})$ and the heat 
operator $\Delta+\partial_{t}$ are hypoelliptic (see~\cite{Ho:ActaMath67}). Following Rothschild-Stein~\cite{RS:ActaMath76} (and other authors) to study the heat operator $\Delta+\partial_{t}$ we lift the filtration $(H_1,\ldots, H_r)$ of $TM$ to the the filtration $\tilde{H}=(\tilde{H}_{1},\ldots,\tilde{H}_{r})$ of $T(M\times \R)$, where $\tilde{H}_{1}=\pi_{1}^{*}H$ and $\tilde{H}_{j}=\pi_{1}^{*}H^{[j]}+\pi_{2}^{*}T\R$ for $j\geq 2$. Here $\pi_{1}$ (resp., $\pi_{2}$) is the projection of $M\times \R$ onto $M$ (resp., $\R$). This allows time-differentiation to have weight~$2$. We obtain a Carnot filtration, but we do not get an ECC structure since $\tilde{H}_{1}$ is not bracket-generating (its Lie brackets only span $\pi_1^*TM$).

\subsubsection{Group actions on contact manifolds~\emph{(}\cite{CP:Privileged}\emph{)}} 
Let $(M,H)$ be an orientable contact manifold. As mentioned in~\cite{CP:Privileged}, for studying the action on $M$ of the group of contactomorphisms of $(M,H)$ we are lead to pass to the full space of the metric contact bundle $P\stackrel{\pi}{\rightarrow}M$. Picking a (positive) contact form on $M$, we can realize $P$ as the $\R^*_+$-bundle $P:=P(H)\oplus \R_+^*\theta$, where $P(H)$ is the metric bundle of $H$ (so that the fiber $P(H)(x)$ at $x\in M$ consists of positive-definite quadratic forms on $H(x)$). Under this passage the Carnot flitration $(H,TM)$ of $TM$ lifts to the filtration $(V, V+\pi^*H, TP)$ of $TP$, where $V=\ker d\pi$ is the vertical bundle of the fibration $\pi:P\rightarrow M$. This provides us with Carnot filtration, but we don't get an ECC structure since $V$ is integrable. 

\begin{remark}
 The last two examples above show that, even if our main focus is on ECC manifolds, we may be naturally lead to 
    consider non-ECC Carnot structures. These examples are one of the main motivation for considering Carnot manifolds instead of
    sticking to the setup of ECC manifolds. 
\end{remark}

\subsection{The Tangent Group Bundle of a Carnot Manifold}\label{sec:tangent-group}
The constructions of the tangent Lie algebra bundle and tangent group bundle of a Carnot manifold go back to Tanaka~\cite{Ta:JMKU70} (see also~\cite{AM:ERA03, AM:JDCS05, CS:AMS09, GV:JGP88, MM:JAM00, Me:Preprint82, Mo:AMS02, Mo:HMJ93, VG:JSM92}).  We refer to~\cite{ABB:SRG, FJ:JAM03, MM:JAM00, Ro:INA} for alternative intrinsic constructions of the tangent Lie algebra bundle and tangent group bundle. The exposition here follows~\cite{CP:Privileged} closely. 

The Carnot filtration $H=(H_{1},\ldots,H_{r})$ has a natural grading defined as follows. For $w=1,\ldots,r$, set 
$\fg_{w}M=H_{w}/H_{w-1}$ (with the convention that $H_{0}=\{0\}$), and define
\begin{equation}\label{eq-grading}
    \fg M := \fg_{1}M\oplus \cdots \oplus \fg_{r}M.
\end{equation}
In what follows, we denote by $\times_M$ the fiber products of vector bundles over $M$. 

\begin{lemma}[\cite{Ta:JMKU70}]\label{lem-dep}
 The Lie bracket of vector fields induces smooth bilinear bundle maps, 
\begin{equation*}
    \cL_{w,w'}:\fg_{w}M\times_M \fg_{w'}M\longrightarrow \fg_{w+w'}M, \qquad w+w'\leq r. 
\end{equation*}More precisely, given any $a\in M$ and any 
 section $X$ (resp., $Y$) of $H_{w}$ (resp., $H_{w'}$) near $a$, if we let $\xi(a)$ (resp., $\eta(a)$)  be the class of $X(a)$ (resp., $Y(a)$) in $\fg_{w}M(a)$ 
 (resp., $\fg_{w'}M(a)$), then we have
\begin{equation}
 \cL_{w,w'}\left(\xi(a),\eta(a)\right)= \textup{class of  $[X,Y](a)$ in $\fg_{w+w'}M(a)$}. 
  \label{eq:Tangent.Lie-backet}
\end{equation}
 \end{lemma}

 \begin{remark}\label{rmk:Tangent-group.frame-fgM}
    Let $(X_{1},\ldots,X_{n})$ be an $H$-frame over an open set $U\subset M$. For $j=1,\ldots,n$ and $x\in U$, let  
     $\xi_{j}(x)$ be the class of $X_{j}(x)$ in $\fg_{w_{j}}M(x)$. Then, as mentioned in~\cite{CP:Privileged}, for $w=1,\ldots,r$, 
     the family $\{\xi_{j}; w_{j}=w\}$ is a smooth frame of $\fg_{w}M$ over $U$, and so $(\xi_{1},\ldots,\xi_{n})$ is a smooth
     frame of $\fg M$ over $U$ which is compatible with the grading~(\ref{eq-grading}). We shall call such a frame a \emph{graded frame}. 
     In particular, for every $x\in U$, we obtain a \emph{graded basis} of $\fg M(x)$. 
       In addition, 
     using~(\ref{eq:Carnot-mfld.brackets-H-frame}) and~(\ref{eq:Tangent.Lie-backet}) 
     we get
     \begin{equation}
         \cL_{w_{i},w_{j}}(\xi_{i}(x),\xi_{j}(x))=\sum_{w_{k}=w_{i}+w_{j}}L_{ij}^{k}(x)\xi_{k}(x) \qquad 
         \text{for $w_{i}+w_{j}\leq r$}, 
         \label{eq:Tangent-Group.Levi-H-frame}
     \end{equation}
    where the coefficients $L_{ij}^k(x)$, $w_i+w_j=w_k$, are defined by~(\ref{eq:Carnot-mfld.brackets-H-frame}). As these coefficients depend smoothly on $x$, this shows that the bilinear bundle maps $\cL_{w_i,w_j}$ are smooth. 
 \end{remark}

\begin{definition}
The bilinear bundle map $[\cdot,\cdot]:\fg M\times_M  \fg M\rightarrow \fg M$ is defined as follows. For $a\in M$ and $(\xi,\eta) \in \fg_wM(a)\times \fg_{w'}M(a)$ we set
    \begin{equation}\label{eq-brak-a}
        [\xi,\eta]=\left\{
        \begin{array}{ll}
            \cL_{w,w'}(a)(\xi,\eta)& \text{if $w+w'\leq r$},  \smallskip \\
           0 &  \text{if $w+w'> r.$ }
        \end{array}\right.
    \end{equation}
 We then extend $[\cdot,\cdot]$ to all $\fg M\times_M  \fg M$ by bilinearity.   
\end{definition}

Lemma~\ref{lem-dep} ensures us that $[\cdot, \cdot]$ is a smooth bilinear bundle map. Furthermore, on each fiber  $\fg M(a)$, $a\in M$, it defines 
a Lie algebra bracket such that
       \begin{align}
                   \left[ \fg_{w}M, \fg_{w'}M \right]  \subset \fg_{w+w'}M   & \quad \text{if $w+w'\leq r$}, \\
                \left[\fg_{w}M,\fg_{w'}M\right] =\{0\}  &  \quad \text{if $w+w'>r $}.
          \label{eq:Carnot.Carnot-grading}
           \end{align}
Therefore, the Lie bracket is compatible with the grading~(\ref{eq-grading}), and so it turns $\fg M(a)$ into a (graded) nilpotent  Lie algebra of step $r$.  
 
\begin{proposition}[\cite{Ta:JMKU70}]
   $(\fg M, [\cdot, \cdot])$ is a smooth bundle of step $r$ graded nilpotent Lie algebras.
\end{proposition}

\begin{definition}
The Lie algebra bundle $(\fg M, [\cdot, \cdot])$  is called the \emph{tangent Lie algebra bundle} of $(M,H)$.
\end{definition}

\begin{remark}
 In~\cite{Mo:HMJ93, Ta:JMKU70} the tangent Lie algebra bundle $\fg M$ is called the symbol algebra of $(M,H)$. It is called the nilpotenization of 
    $(M,H)$ in~\cite{GV:JGP88, Mo:AMS02, VG:JSM92}. 
\end{remark}

\begin{remark}\label{rem-xixj}
    Let $(X_{1},\ldots,X_{n})$ be an $H$-frame near a point $a\in M$. As mentioned above, this gives rise to a 
    frame $(\xi_{1},\ldots,\xi_{n})$ of $\fg M$ near $x=a$, where $\xi_{j}$ is the class of $X_{j}$ in 
    $\fg_{w_{j}}M$. Furthermore, it follows from~(\ref{eq:Carnot-mfld.brackets-H-frame}) and~(\ref{eq-brak-a}) that, near $x=a$, we have
    \begin{equation}\label{eq-jk}
        [\xi_{i}(x),\xi_{j}(x)]=\left\{
        \begin{array}{cl}
          {\displaystyle \sum_{w_{i}+w_{j}=w_k}L_{ij}^{k}(x)\xi_{k}(x)}    & \text{if $w_{i}+w_{j}\leq r$},  \\
            0  & \text{if $w_{i}+w_{j}>r$},
        \end{array}\right.
    \end{equation}where the functions $L_{ij}^{k}(x)$ are given by~(\ref{eq:Carnot-mfld.brackets-H-frame}). Specializing this to $x=a$ provides us with the structure constants 
    of $\fg M(a)$ with respect to the basis $(\xi_{1}(a),\ldots,\xi_{n}(a))$.  
\end{remark}

\begin{remark}\label{rem:GM.ECC-Carnot-algebra}
 When $(M,H)$ is an ECC manifold, each fiber $\fg M(a)$, $a\in M$, is  a Carnot algebra in the sense of Definition~\ref{def:Carnot.Carnot-algebra} (see, e.g.,~\cite{CP:Privileged}). 
\end{remark}
  
The Lie algebra bundle $\fg M$ gives rise to a Lie group bundle $GM$ as follows. As a manifold we take $GM$ to be $\fg M$ and we equip the fibers $GM(a)=\fg M(a)$ with the Dynkin product~(\ref{eq:Carnot.Dynkin-product}).This turns $GM(a)$ into a step~$r$ graded nilpotent Lie group with unit $0$ whose Lie algebra is naturally isomorphic to $\fg M(a)$. The fiberwise product on $GM$ is smooth, and so we obtain the following result.

\begin{proposition}
    $GM$ is a smooth bundle of step $r$ graded nilpotent Lie groups. 
\end{proposition}

\begin{definition}\label{def:Tangent.Tangent-group}
    $GM$ is called the \emph{tangent group bundle} of $(M,H)$. Each fiber $GM(a)$, $a\in M$ is called the \emph{tangent group} of $(M,H)$ at $a$. 
\end{definition}

\begin{example}
 Suppose that $r=1$ so that $H_1=H_r=TM$. In this case, each Lie algebra $\fg M(a)$, $a\in M$, is Abelian and agrees with $TM(a)$. Therefore, as a Lie algebra bundle, $\fg M=TM$, and, as a Lie group bundle, $GM=TM$.  
\end{example}

\begin{example}\label{ex:Tangent-group.group}
Let $G$ be a graded nilpotent Lie group with Lie algebra $\fg=\fg_1\oplus \cdots \oplus \fg_r$. We equip $G$ with the left-invariant Carnot manifold structure defined by the grading of $\fg$. Then the left-regular actions of $G$ on itself and on $\fg$ give rise to canonical identifications $\fg G\simeq G\times \fg$ and $GG\simeq G\times G$ (see~\cite{CP:Carnot}). 
\end{example}

\begin{remark}
Some authors call $GM(a)$, $a\in M$, the osculating group of $(M,H)$ at $a$ (see, e.g., \cite{Jv:Preprint10, vE:AM10.Part1,vEY:BLMS17}).  
\end{remark}

\begin{remark}
 We refer to~\cite{CP:Carnot, Po:PJM06} for an explicit description of the tangent group bundle of a Heisenberg manifold. 
\end{remark}

\begin{remark}
When $(M,H)$ is an ECC manifold, it follows from Remark~\ref{rem:GM.ECC-Carnot-algebra} that every tangent group $GM(a)$ is a 
Carnot group in the sense of Definition~\ref{def:Carnot.Carnot-algebra}.
\end{remark}

\begin{remark}\label{rmk:tangent.dilations-group-automorphisms}
 In the same way as above, the grading~(\ref{eq-grading}) gives rise to a smooth family of dilations $\delta_t$, $t\in \R\setminus 0$. 
 This gives rise to Lie algebra automorphisms on the fibers of $\fg M$ and group automorphisms on the fibers of $GM$. In addition, the inversion on the fibers of $GM$ is just  the symmetry $\xi \rightarrow -\xi$. 
 \end{remark}

\section{Anisotropic Asymptotic Analysis}\label{sec:anisotropic}
In this section, we gather various results on anisotropic asymptotic expansions of maps and vector fields. This type of asymptotic expansions have been considered in various levels of generality by a number of authors~\cite{ABB:SRG, BG:CHM, Be:Tangent, Go:LNM76, Gr:CC, He:SIAMR, Je:Brief14, MM:JAM00, Me:CPDE76, Mi:JDG85, Ro:INA, Mo:AMS02, RS:ActaMath76}. We shall follow closely the exposition of~\cite{CP:Privileged}. 

In what follows, we let $\delta_t:\R^n\rightarrow \R^n$, $t\in \R$, be the one-parameter group of anisotropic dilations given by 
\begin{equation}
    \delta_{t}(x)=t\cdot x:=(t^{w_{1}}x_{1},\ldots,t^{w_n}x_{n}), \qquad t\in \R, \ x\in \R^n.
    \label{eq:Nilpotent.dilations2}
\end{equation}
We shall say that a function $f(x)$ on $\R^{n}$ is  \emph{homogeneous} of degree $w$,
$w\in \Z$, with respect to the dilations~(\ref{eq:Nilpotent.dilations2}) when 
\begin{equation}
f(t\cdot x)=t^{w}f(x) \qquad \text{for all $x\in \R^n$ and $t\in \R\setminus 0$}. 
\label{eq:priv.homog-function}
\end{equation}
For instance, given any $\alpha \in \N_{0}^{n}$, the monomial $x^{\alpha}$ is homogeneous of degree $\brak \alpha$, where 
$\brak \alpha:=w_1\alpha_1+ \cdots+w_n\alpha_n$.

In what follows we let $U$ be an open neighborhood of the origin $0\in \R^{n}$. 

\begin{definition}
Let $f\in C^{\infty}(U)$ and $w\in \N_0$. We shall say that
\begin{enumerate}
    \item  $f$ has weight $\geq w$ when $\partial^{\alpha}_{x}f(0)=0$ for all multi-orders
    $\alpha \in \N_{0}^{n}$ such that $\brak\alpha<w$.
    \item  $f$ has weight $w$ when $f(x)$ has weight~$\geq w$ and there is a multi-order
    $\alpha\in \N_{0}^{n}$ with $\brak\alpha=w$ such that $\partial^{\alpha}_{x}f(0)\neq 0$.
\end{enumerate}
\end{definition}

In what follows we assume we are given a pseudo-norm $\|\cdot\|$ on $\R^n$, i.e., a non-negative continuous function which is~$>0$ on $\R^n\setminus 0$ and satisfies 
\begin{equation}
 \|t\cdot x\|= |t| \|x\| \qquad \text{for all $x\in \R^n $ and $t\in \R$}.
 \label{eq:anisotropic.homogeneity-pseudo-norm}
\end{equation}
For instance, we may take $\|x\|_1=  |x_1|^{\frac{1}{w_1}} + \cdots + |x_n|^{\frac{1}{w_n}}$. As pointed out in~\cite{CP:Privileged} all pseudo-norms are equivalent, and so the specific choice of a pseudo-norm is irrelevant to our purpose. 

 In addition, given $n'\in \N$, we let $(w_1', \ldots, w_{n'}')$ be another weight sequence with values in $\{1, \ldots, r\}$, i.e., a non-decreasing sequence such that $w_1'=1$.  This allows us to endow $\R^{n'}$ with the family of anisotropic dilations~(\ref{eq:Nilpotent.dilations2}) associated with this weight sequence. We shall use the same notation for the dilations on $\R^n$ and $\R^{n'}$. When $n=n'$ we shall tacitly  assume that $w_j'=w_j$ for $j=1,\ldots,n$.

We shall say that a map $\Theta:\R^n \rightarrow \R^{n'}$ is \emph{$w$-homogeneous} when 
\begin{equation*}
\Theta(t\cdot x)=t\cdot \Theta(x) \qquad \text{for all $x\in \R^n$ and $t\in \R$}. 
\end{equation*}
Equivalently, if we set $\Theta(x)=(\Theta_1(x), \ldots, \Theta_{n'}(x))$, then the component $\Theta_k(x)$ is homogeneous of degree~$w_k'$ for $k=1,\ldots,n'$. 

\begin{definition}\label{def:anisotropic.Thetaw}
Let $\Theta(x)=(\Theta_{1}(x),\ldots,\Theta_{n'}(x))$ be a  map from $U$  to $\R^{n'}$. Given $m\in \Z$, $m\geq -w'_{n'}$, we say that $\Theta(x)$ is $\Ow(\|x\|^{w+m})$, and write $\Theta(x)=\Ow(\|x\|^{w+m})$, when, for $k=1,\ldots,n'$, we have
\begin{equation*}
 \Theta_{k}(x)=\op{O}(\|x\|^{w_{k}'+m}) \qquad \text{near $x=0$}.  
\end{equation*}
\end{definition}

In what follows we equip $C^\infty(U,\R^{n'})$ with its standard Fr\'echet-space topology. 

\begin{lemma}[see~\cite{CP:Privileged}]\label{lem-eq-we}
Let $\Theta(x)=(\Theta_{1}(x),\ldots,\Theta_{n'}(x))$ be a smooth map from $U$ to $\R^{n'}$, and set $ \cU=\{(x,t)\in U\times \R; \ t\cdot x \in U\}$. Given any $m\in \N_0$, the following are equivalent:
\begin{enumerate}
 \item[(i)]    The map $\Theta(x)$ is $\Ow(\|x\|^{w+m})$ near $x=0$.

 \item[(ii)]   For $k=1,\ldots,n'$, the component $\Theta_{k}(x)$ has weight~$\geq w_{k}'+m$. 

 \item[(iii)] For all $x\in \R^{n}$ and as $t\rightarrow 0$, we have $t^{-1}\cdot \Theta(t\cdot x)=\op{O}(t^{m})$.  

 \item[(iv)]  As $t\rightarrow 0$, we have $t^{-1}\cdot \Theta(t\cdot x)=\op{O}(t^{m})$ in $C^{\infty}(U,\R^{n'})$. 
 
 \item[(v)]  There is $\tilde{\Theta}(x,t)\in C^\infty(\cU, \R^{n'})$ such that 
 $t^{-1}\cdot \Theta(t\cdot x)=t^m \tilde{\Theta}(x,t)$ for all $(x,t)\in \cU$, $t\neq 0$.
\end{enumerate}
\end{lemma}

\begin{proposition}[see~\cite{CP:Privileged}]\label{lem:multi.Thetat}
Let $\Theta(x)=(\Theta_1(x), \ldots, \Theta_{n'}(x))$ be a smooth map from $U$ to $\R^{n'}$. Then, as $t\rightarrow 0$, we have
\begin{equation}
t^{-1}\cdot  \Theta(t\cdot x) \simeq \sum_{\ell \geq -w_{n'}'} t^\ell \Theta^{[\ell]}(x) \qquad \text{in $C^\infty(U,\R^{n'})$},
\label{eq:multi.Thetat}
\end{equation}
where $\Theta^{[\ell]}(x)$ is a polynomial map such that $t^{-1}\cdot \Theta^{[\ell]}(t\cdot x)=t^\ell \Theta^{[\ell]}(x)$ for all $t\in \R^*$.   
\end{proposition}

\begin{remark}
More explicitly, we have  $\Theta^{[\ell]}(x)=(\Theta^{[\ell]}_1(x), \ldots, \Theta^{[\ell]}_{n'}(x))$, where 
\begin{equation*}
\Theta^{[\ell]}_{k}(x)= \sum_{\brak\alpha=\ell +w_k'}\frac{1}{\alpha!} \partial^\alpha\Theta_k(0)x^\alpha, \qquad k=1,\ldots,n'.  
\end{equation*}
\end{remark}

\begin{remark}\label{rmk:anisotropic.weighted-asymptotic} 
 The asymptotic expansion~(\ref{eq:multi.Thetat}) exactly means that, for every integer $m> -w_{n'}'$, we have 
\begin{equation*}
 \Theta(x) - \sum_{\ell <m}  \Theta^{[\ell]}(x)=\Ow\left( \|x\|^{w+m}\right) \qquad \text{near $x=0$}.
\end{equation*}
In particular, we see that $ \Theta(x)=\Ow( \|x\|^{w+m})$ if and only if $\Theta^{[\ell]}(x)=0$ for $\ell<m$. 
\end{remark}

The previous considerations also extend to vector fields. The dilations~(\ref{eq:Nilpotent.dilations2}) act on vector fields by pullback. If $X=\sum_{j=1}^n a_j(x)\partial_{x_j}$ is a vector field on $U$, then we have 
\begin{equation}
 \delta_t^*X = \sum_{j=1}^n t^{-w_j}a_j(t\cdot x)\partial_{x_j} \qquad \text{for all $t\in \R^*$}. 
 \label{eq:anisotropic-dtX}
\end{equation}
We then shall say that a vector field $X$ on $\R^{n}$  is \emph{homogeneous} of degree $w$, $w\in \Z$, when
\begin{equation}
    \delta_{t}^{*}X=t^{\omega} X \qquad \text{for all $t\in \R^*$}. 
    \label{eq:priv.homog-vecto-field}
\end{equation}
For instance, the vector field $\partial_{x_j}$, $j=1,\ldots, n$, is homogeneous of degree $-w_j$. 

\begin{definition} 
Let $X=\sum_j a_j(x)\partial_j$ be a smooth vector field on $U$. We say that $X$ has weight $w$, $w\in \Z$, when, for every $j=1, \ldots, n$, the coefficient $a_j(x)$ has weight~$\geq w+w_j$ and we have equality at least once. 
 \end{definition}

Let us denote by $\cX(U)$ the space of (smooth) vector fields on $U$. Any vector field $X\in \cX(U)$ has a unique expression of the form, 
\begin{equation*}
    X=\sum_{1\leq j \leq n}a_{j}(X)(x)\partial_{j}, \qquad \text{with $a_{j}(X)\in C^{\infty}(U)$}.
\end{equation*}
If we set $X[x]=(a_1(X)(x),\ldots, a_n(X)(x))$, then $X\rightarrow X[x]$ is a linear isomorphism from $\cX(U)$ onto $C^\infty(U,\R^n)$. The standard Fr\'echet-space topology of $\cX(U)$ is precisely the weakest locally convex space topology with respect to which this linear map is continuous. 

\begin{proposition}[see~\cite{CP:Privileged}]\label{prop:nilp-approx.vector-fields}
    Let $X$ be a smooth vector field on $U$. Then, as $t\rightarrow 0$, we have 
                \begin{equation}\label{eq-delta-X}
            \delta_{t}^{*}X\simeq \sum_{\ell\geq -r} t^{\ell}X^{[\ell]} \qquad \text{in $\cX(U)$},
        \end{equation}
        where $X^{[\ell]}$ is a homogeneous polynomial vector field of degree $\ell$. In particular, 
$X$ has weight $w$ if and only if, as $t\rightarrow 0$, we have 
       \begin{equation*}
       \delta_{t}^{*}X= t^{w}X^{[w]} +\op{O}(t^{w+1})\qquad \text{in $\cX(U)$}.
    \end{equation*}
\end{proposition}

Finally, we will need the following extension of Lemma~\ref{lem-eq-we}. 

\begin{lemma}\label{lem:anisotropic.Theta-parameter}
 Let $W$ be an open subset of $\R^p\times \R^n$ such that $W \cap (\R^p\times \{0\})=U\times \{0\}$, where $U$ is a non-empty (open) subset of $\R^p$. In addition, set 
 $\cU=\{(x,y,t)\in U\times \R^n\times \R; (x,t\cdot y)\in W\}$. Let $\Theta(x,y)\in C^\infty(W,\R^n)$. Then the following are equivalent:
 \begin{enumerate}
 \item[(i)]    For every $x\in U$, we have $\Theta(x,y)=\Ow(\|y\|^{w+m})$ near $y=0$.

 \item[(ii)]   For every $(x,y)\in U\times \R^n$ and as $t\rightarrow 0$, we have $t^{-1}\cdot \Theta(x,t\cdot y)=\op{O}(t^{m})$. 
 
\item[(iii)]  There is a map $\tilde{\Theta}(x,y,t)\in C^\infty(\cU, \R^{n'})$ such that 
 $t^{-1}\cdot \Theta(x,t\cdot y)=t^m \tilde{\Theta}(x,y,t)$ for all $(x,y,t)\in \cU$ with $t\neq 0$.
\end{enumerate}
\end{lemma}
\begin{proof}
It follows from Lemma~\ref{lem-eq-we} that (i) and (ii) are equivalent and are both implied by (iii). Therefore, we only have to show that (i) implies (iii). Thus, suppose that, for every $x\in U$, we have $\Theta(x,y)=\Ow(\|y\|^{w+m})$ near $y=0$. If we set $\Theta(x,y)=(\Theta_1(x,y),\ldots, \Theta_{n'}(x,y))$, then this means that, for every $x\in U$ and $k=1,\ldots, n'$, we have $\Theta_k(x,y)=\op{O}(\|y\|^{w_k'+m})$ near $y=0$.  

Given any function $F(x,y)\in C^\infty(W)$ an elaboration of the arguments of the proof of the anisotropic Taylor formula in~\cite{CP:Privileged} shows that, for every $N\in \N_0$, there are functions $R_{N\alpha}(x,y)\in C^\infty(W)$, $|\alpha|\leq N \leq \brak\alpha$, such that
\begin{equation*}
 F(x,y)= \sum_{\brak\alpha <N}\frac{1}{\alpha!} y^\alpha\partial_y^\alpha F(x,0) + \sum_{|\alpha|\leq N \leq \brak\alpha} y^\alpha R_{N\alpha}(x,y) \qquad \text{for all $(x,y)\in W$}.  
\end{equation*}
Moreover, if we further assume that, for every $x\in U$, we have $F(x,y)=\op{O}(\|y\|^N)$ near $y=0$, then the proof of Lemma~\ref{lem-eq-we} for $n'=1$ in~\cite{CP:Privileged} shows that there is a map $G(x,y,t)\in C^\infty(\cU, \R^{n'})$ such that
\begin{equation*}
 F(x,t\cdot y) = t^N G(x,y,t) \qquad \text{for all $(x,y,t)\in \cU$}, 
\end{equation*}
More precisely, the map $G(x,y,t)$ is given by
\begin{equation*}
 G(x,y,t) = \sum_{|\alpha|\leq N \leq \brak\alpha}t^{\brak\alpha -N} y^\alpha R_{N\alpha}(x,t\cdot y), \qquad (x,y,t)\in \cU.  
\end{equation*}

Applying the above considerations to $F(x,y)=\Theta_k(x,y)$, $k=1,\ldots, n'$, and $N=\max\{w_k'+m,0\}$ we see that there is a function $\tilde{\Theta}_k(x, y,t)\in C^\infty(\cU)$ such that $\Theta_k(x,t\cdot y)=t^{w_k'+m}\tilde{\Theta}_k(x,y,t)$ for all $(x,y,t)\in \cU$. Thus, if we set $\tilde{\Theta}(x,y,t)= ( \tilde{\Theta}_1(x,y,t),\ldots, \tilde{\Theta}_{n'}(x,y,t))$, $(x,y,t)\in \cU$, then we obtain a smooth map from $\cU$ to $\R^{n'}$ such that 
\begin{equation*}
 t^{-1}\cdot \Theta(x,t\cdot y)=t^m \tilde{\Theta}(x,y,t)\qquad \text{for all $(x,y,t)\in \cU$, $t\neq 0$}.
\end{equation*}
This shows that (i) implies (iii). The proof is complete. 
\end{proof}

\section{Carnot Coordinates}\label{sec:Carnot-coordinates}
In this section, we survey the main definitions and facts regarding privileged coordinates, nilpotent approximation, and the Carnot coordinates of~\cite{CP:Carnot}. 

Throughout this section we let $(X_{1},\ldots,X_{n})$ be an $H$-frame over an open neighborhood of a given point $a\in M$. 

\subsection{Privileged coordinates and nilpotent approximation} Recall that local coordinates $(x_{1},\ldots,x_{n})$ centered at  $a$ are said to be \emph{linearly adapted} to
   the $H$-frame $(X_{1},\ldots,X_{n})$ when $X_{j}({0})=\partial_{j}$ for $j=1,\ldots,n$. Any system of local coordinates can be converted into a system of linearly adapted coordinates by means of a unique affine transformation (see, e.g., \cite{Be:Tangent, CP:Privileged}). 

\begin{definition}[\cite{Go:LNM76, RS:ActaMath76}]
 We say that local coordinates $(x_{1},\ldots,x_{n})$ are \emph{privileged coordinates at $a$ adapted to the $H$-frame} $(X_1,\ldots, X_n)$ when
 \begin{enumerate}
\item[(i)] They are linearly adapted at $a$ to $(X_1,\ldots, X_n)$.

\item[(ii)] For $j=1, \ldots,n$, the vector field $X_j$ has weight $-w_j$ in these coordinates. 
\end{enumerate}
\end{definition}

\begin{remark}
 In~\cite{Be:Tangent, BS:SIAMJCO90, CP:Privileged} privileged coordinates are defined by replacing the condition (ii) by the requirement that $X^\alpha(x_j)(0)=0$ whenever $\brak\alpha <w_j$. In the terminology of~\cite{Be:Tangent, BS:SIAMJCO90} this means that the coordinate function $x_j$ has order $w_j$ for $j=1,\ldots, n$. It is well known that if $(x_1, \ldots, x_n)$ are linearly adapted coordinates, then this condition implies (ii) (see~\cite{Be:Tangent, BS:SIAMJCO90, Je:Brief14}). The converse is proved in~\cite{CP:Privileged}, and so the two definitions are equivalent. 
\end{remark}

\begin{remark}
 We refer to~\cite{AS:DANSSSR87, Be:Tangent,BS:SIAMJCO90,CP:Privileged, Go:LNM76, He:SIAMR, RS:ActaMath76, St:1988} for various constructions of privileged coordinates. We also refer to~\cite{CP:Privileged} for the precise description of all the systems of privileged coordinates at a given point. 
\end{remark}

Suppose that $(x_1,\ldots, x_n)$ are privileged coordinates at $a$ adapted to $(X_1,\ldots, X_n)$. We let $U\subset \R^n$ be their range. For $j=1, \ldots, n$, the vector field $X_j$ has weight $w_j$ in these coordinates. Therefore, by Proposition~\ref{prop:nilp-approx.vector-fields} there is a polynomial vector field $X^{(a)}_j$ which is homogeneous of degree~$-w_j$ and such that, as $t\rightarrow 0$, we have  
\begin{equation}
 t^{w_j}\delta_t^* X_j=X_j^{(a)}+\op{O}(t) \qquad \text{in $\cX(U)$}. 
 \label{eq:Carnot-coord.approximationXj(a)}
\end{equation}
We call $X_j^{(a)}$ the \emph{model vector field} of $X_j$ at $a$ in the coordinates $(x_1, \ldots, x_n)$. 

Let $\fg^{(a)}$ be the subspace of vector fields on $\R^n$ spanned by $X_1^{(a)}, \ldots, X_n^{(a)}$. There is a natural grading 
$\fg^{(a)}=\fg_{1}^{(a)}\oplus \cdots \oplus \fg_{r}^{(a)}$, where $\fg_{w}^{(a)}$ is spanned by the vector fields $X_j^{(a)}$ with $w_j=w$. Moreover, 
using~(\ref{eq:Carnot-mfld.brackets-H-frame}) and~(\ref{eq:Carnot-coord.approximationXj(a)}) we see that 
 \begin{equation}
     [X_{i}^{(a)}, X_{j}^{(a)}]=\left\{
     \begin{array}{cl}
   {\displaystyle \sum_{w_{k}=w_{i}+w_{j}}L_{ij}^{k}(a)X_{k}^{(a)}} & \text{if $w_{i}+w_{j}\leq  r$},   \\
         0 & \text{otherwise}.  \\
     \end{array}\right.
     \label{eq:Nilpotent.structure-constants-Xj(a)}
 \end{equation}
It then follows that with respect to the bracket of vector fields $\fg^{(a)}$ is a graded nilpotent Lie algebra which is isomorphic to the tangent Lie algebra $\fg M(a)$. 

There is a unique group law on $\R^n$ such that if $G^{(a)}$ is $\R^n$ equipped with this group law, then $G^{(a)}$ is a graded nilpotent Lie group whose Lie algebra of left-invariant vector fields is precisely $\fg^{(a)}$. 

\begin{definition}
 The nilpotent graded Lie group $G^{(a)}$ equipped with its left-invariant Carnot manifold structure is called the \emph{nilpotent approximation} of $(M,H)$ at $a$ in the privileged coordinates $(x_1,\ldots, x_n)$. 
\end{definition}

The dilations~(\ref{eq:Nilpotent.dilations2}) are group automorphisms of the Lie group $G^{(a)}$. In addition, let $\fg(a)$ be graded nilpotent Lie algebra obtained by equipping $T\R^{n}(0)\simeq \R^n$ with the Lie bracket given by
  \begin{equation*}
     [\partial_{i}, \partial_{j}]=\left\{
     \begin{array}{cl}
   {\displaystyle \sum_{w_{k}=w_{i}+w_{j}}L_{ij}^{k}(a)\partial_{k}} & \text{if $w_{i}+w_{j}\leq  r$},   \\
         0 & \text{otherwise}  \\
     \end{array}\right.
 \end{equation*}
It follows from~(\ref{eq:Nilpotent.structure-constants-Xj(a)}) that the Lie algebra $TG^{(a)}(0)$ is precisely $\fg(a)$. Conversely, we have the following result. 

\begin{proposition}[\cite{CP:Privileged}]\label{eq:Carnot-coord.charact-nilpotent-approx} 
 Let $G$ be a nilpotent Lie group built out of $\R^n$. Then $G$ provides us with the nilpotent approximation of $(M,H)$ in some privileged coordinates at $a$ adapted to $(X_1,\ldots, X_n)$ if and only if the following two conditions are satisfied: 
 \begin{enumerate}
\item[(i)] The dilations~(\ref{eq:Nilpotent.dilations2}) are group automorphisms of $G$. 

\item[(ii)] The Lie algebra $TG(0)$ of $G$ agrees with $\fg(a)$. 
\end{enumerate}
\end{proposition}

We refer to~\cite{CP:Privileged} for the explicit constructions of examples of groups satisfying the properties (i)--(ii). When $r=2$ these examples provide us with all the nilpotent approximations at a given point. In general, this shows that the nilpotent approximation encompasses a very large class of groups, and so this process is very sensitive to the choice of the privileged coordinates.

\subsection{Coordinate description of $GM(a)$} 
Let $U_0$ be the domain of the $H$-frame $(X_1, \ldots, X_n)$. By Remark~\ref{rmk:Carnot-mfld.brackets-H-frame} there are functions $L_{ij}^k(x)\in C^\infty(U_0)$, $w_k\leq w_i+w_j$, given by the commutator relations~(\ref{eq:Carnot-mfld.brackets-H-frame}). 
As mentioned in Remark~\ref{rmk:Tangent-group.frame-fgM},  the $H$-frame $(X_1, \ldots, X_n)$ gives rise to a graded basis $(\xi_1(a), \ldots, \xi_n(a))$ of $\fg M(a)$, where $\xi_j(a)$ be the class of $X_j(a)$ in $\fg_{w_j}M(a)$. This graded basis defines a global system of coordinates on $\fg M(a)=GM(a)$. In these coordinates the dilations~(\ref{eq:Carnot.dilations}) are given by~(\ref{eq:Nilpotent.dilations2}). Moreover, as explained in, e.g.,~\cite{CP:Privileged}, the product law of $GM(a)$ can be described as follows. 

Let $\xi=\sum x_{i}\xi_{i}(a)$, $x_j\in \R$, be an element of ${\fg}M(a)$. The matrix $A_{a}(x)=\left(A_{a}(x)_{kj}\right)_{1\leq k,j\leq n}$ of the adjoint endomorphism $\ad_{\xi}$ with respect to the basis $(\xi_{1}(a),\ldots,\xi_{n}(a))$ is given by
\begin{equation*}
    A_{a}(x)_{kj} =\left\{\begin{array}{cl} {\displaystyle \sum_{w_i +w_j= w_k} L_{ij}^{k}(a) x_i }&\textup{if $w_j<w_k$},\\
0 &\textup{otherwise.}\end{array}\right.
\end{equation*}
In particular, this matrix  is nilpotent and lower-triangular.  Combining this with~(\ref{eq:Carnot.Dynkin-product}) shows that, in the coordinates
defined by the basis $(\xi_{1}(a),\ldots,\xi_{n}(a))$,  the product of $GM(a)$ is given by
\begin{equation}
  x\cdot y  =  \sum_{n\geq 1} \frac{(-1)^{n+1}}{n} \sum_{\substack{\alpha, \beta \in \N_0^n\\ \alpha_j+\beta_j\geq 1}} 
 \frac{(|\alpha|+|\beta|)^{-1}}{\alpha!\beta!} A_a(x)^{\alpha_1} A_a(y)^{\beta_1} \cdots A_a(x)^{\alpha_n} A_a(y)^{\beta_n-1}y. 
 \label{eq:GM.group-law}
\end{equation}
Therefore, for $k=1,\ldots, n$, we have
\begin{equation}
    (x\cdot y)_{k}=x_{k}+y_{k}+ \frac{1}{2} \sum_{w_i+w_j=w_k} L_{ij}^k(a)x_iy_j + \sum_{\substack{\brak\alpha+\brak\beta=w_{k}\\ |\alpha|+|\beta|\geq 3}}
    B_{\alpha\beta}^k\left(L(a)\right)x^{\alpha}y^{\beta},
    \label{eq:GM.group-law2}
\end{equation}where the coefficients $B_{\alpha\beta}^k(L(a))$ are universal polynomials in the structure constants $L_{pq}^{l}(a)$, $w_p+w_q=w_l$, (see~\cite{CP:Carnot}). The quadratic terms (resp., super-quadratic terms) appear only when $w_k\geq 2$ (resp., $w_k\geq 3$). 

\begin{definition}
 The graded nilpotent Lie group $G(a)$ is $\R^n$ equipped with the group law~(\ref{eq:GM.group-law}). 
\end{definition}

\begin{remark}\label{rmk:Carnot-coord.independence-G(a)}
 As~(\ref{eq:GM.group-law2}) shows, the group law of $G(a)$ only depends on the structure constants $L_{ij}^k(a)$, $w_i+w_j=w_k$. In particular,   $G(a)$ only depends on the jets of the $H$-frame $(X_1,\ldots,X_n)$ at $x=a$. 
\end{remark}

We summarize the previous discussion in the following. 

\begin{proposition}\label{prop:Carnot-coord.GM(a)-G(a)}
 Let $(\xi_1(a), \ldots, \xi_n(a))$ be the graded basis of $\fg M(a)$ defined by the $H$-frame $(X_1,\ldots, X_n)$. Then it defines a global system of coordinates that identifies the tangent group $GM(a)$ with the graded nilpotent Lie group $G(a)$. 
\end{proposition}

It is immediate from~(\ref{eq:GM.group-law2}) that the dilations~(\ref{eq:Nilpotent.dilations2}) are group automorphisms of $G(a)$. For $j=1,\ldots,n$, let $X_j^a$ be the left-invariant vector field  on $G(a)$ that agrees with $\partial_j$ at $x=0$. In addition, 
let $(\epsilon_1, \ldots, \epsilon_n)$ be the canonical basis of $\R^n$. We shall regard the vectors $\epsilon_j$ as elements of $\N_0^n$. Using~(\ref{eq:GM.group-law}) it then can be shown that
\begin{equation}
 X_j^a(x)= \partial_j +  \frac{1}{2} \sum_{w_i+w_j=w_k} L_{ij}^k(a)x_i \partial_k + \sum_{\substack{\brak\alpha+w_j=w_{k}\\ |\alpha|\geq 2}}
    B_{\alpha \epsilon_j}^k\left(L(a)\right)x^{\alpha}\partial_k.  
    \label{eq:GM-formula-Xja}
\end{equation}
It follows from this that the vector fields $X_1^a, \ldots, X_n^a$ satisfy the commutator relations~(\ref{eq:Nilpotent.structure-constants-Xj(a)}). This implies that the Lie algebra of $G(a)$ is $\fg(a)$. 

\subsection{Carnot coordinates}
As mentioned above,  the dilations~(\ref{eq:Nilpotent.dilations2}) are group automorphisms of $G(a)$. Moreover, its Lie algebra is $\fg(a)$. Therefore, by Proposition~\ref{eq:Carnot-coord.charact-nilpotent-approx}   the graded nilpotent Lie group $G(a)$ provides us with the nilpotent approximation of $(M,H)$  at $a$ in suitable systems of privileged coordinates. These systems of coordinates were the main object of study of~\cite{CP:Carnot}.  They can be described as follows.

\begin{definition}[\cite{CP:Carnot}]
 Let $(x_{1},\ldots,x_{n})$ be local coordinates centered at $a$ with range $U\subset \R^n$. We say that $(x_{1},\ldots,x_{n})$ are \emph{Carnot coordinates} at $a$ adapted to the $H$-frame $(X_1, \ldots, X_n)$ when 
 \begin{enumerate}
\item[(i)] They are linearly adapted at $a$ to $(X_1, \ldots, X_n)$. 

\item[(ii)] For $j=1, \ldots, n$ and as $t\rightarrow 0$, we have 
               \begin{equation}
            t^{w_{j}}\delta_{t}^{*}X_{j}=X^{a}_{j}+\op{O}(t) \qquad \text{in $\cX(U)$}.
             \label{eq:Carnot.Xj(a)=Xja}
        \end{equation}    
where $X_j^a$ is given by~(\ref{eq:GM-formula-Xja}).
\end{enumerate}
\end{definition}

\begin{remark}
 As $X_j^{a}$ is homogeneous of degree $-w_j$, we know by Proposition~\ref{prop:nilp-approx.vector-fields} that the asymptotics~(\ref{eq:Carnot.Xj(a)=Xja}) implies that $X_j$ has weight~$-w_j$. Therefore, Carnot coordinates are just privileged coordinates such that, for $j=1, \ldots, n$, the model vector field of $X_j$ is $X_j^a$. 
\end{remark}

\begin{proposition}[\cite{CP:Carnot}]
Let $(x_{1},\ldots,x_{n})$ be privileged coordinates at $a$ adapted to $(X_{1},\ldots,X_{n})$. Then $(x_1,\ldots, x_n)$ are Carnot coordinates if and only if in these coordinates the nilpotent approximation of $(M,H)$ at $a$ is given by the graded nilpotent group $G(a)$.  
\end{proposition}

In other words, the Carnot coordinates are precisely the privileged coordinates at $a$ adapted to $(X_{1},\ldots,X_{n})$ in which we have a natural identification between the nilpotent approximation and the tangent group $GM(a)$. 
In addition, the group $G(a)$ only depends on the structure constants $L_{ij}^k$, $w_i+w_j=w_k$ (\emph{cf.}~Remark~\ref{rmk:Carnot-coord.independence-G(a)}). Therefore, the Carnot coordinates provides us with a class of systems of privileged coordinates in which the nilpotent approximation is independent of the choice of the coordinate system in that class.  

\begin{example}
 When $r=1$ any system of linearly adapted coordinates is a system of Carnot coordinates (see~\cite{CP:Carnot}). 
\end{example}

\begin{example}
 If $(M,H)$ is a contact manifold, then any system of Darboux coordinates centered at a given point $a\in M$ is a system of Carnot coordinates at $a$ (see~\cite{CP:Carnot}). \end{example}

\begin{example}
 On any Carnot manifold the canonical coordinates of the 1st kind of Goodman~\cite{Go:LNM76} and Rothschild-Stein~\cite{RS:ActaMath76} are given by the inverse of the diffeomorphism,
 \begin{equation}
V\ni x\longrightarrow \exp(x_1X_1+\ldots +x_nX_n)(a) \in V_0,
\label{eq:Carnot-coord.can-coord1st}
\end{equation}
where $V$ is a sufficiently small open neighborhood of the origin and $V_0$ is an open neighborhood of $a$. It is shown in~\cite{CP:Carnot} that  these coordinates are Carnot coordinates at $a$ adapted to $(X_{1},\ldots,X_{n})$.
\end{example}

\begin{remark}
 In general the polynomial privileged coordinates of Bella\"iche~\cite{Be:Tangent} are not Carnot coordinates. Moreover, in step~$r\geq 2$ the canonical coordinates of the 2nd kind of Bianchini-Stefani~\cite{BS:SIAMJCO90} and Hermes~\cite{He:SIAMR} are never Carnot coordinates (see~\cite{CP:Carnot}). 
\end{remark}

\begin{remark}\label{rmk:Carnot-coord.priv-to-Carnot} 
 Let $(x_1, \ldots, x_n)$ be privileged coordinates at $a$ adapted to $(X_1, \ldots, X_n)$. They always can be converted into Carnot coordinates (see~\cite{CP:Carnot}). More precisely, for $j=1, \ldots, n$, let $X_j^{(a)}$ be the model vector field of $X_j$. We then obtain Carnot coordinates by means of the change of variables  provided by the diffeomorphism,  
\begin{equation*}
 \R^n\ni x \longrightarrow \exp\left(x_1X_1^{(a)}+\cdots + x_nX_n^{(a)}\right) \in \R^n, 
\end{equation*}
where $\exp(x_1X_1^{(a)}+\cdots + x_nX_n^{(a)}):=  \exp[t(x_1X_1^{(a)}+\cdots + x_nX_n^{(a)})](0)|_{t=1}$. Thanks to the homogeneity of the vector fields $X_j$ this diffeomorphism is defined on all $\R^n$ and is $w$-homogeneous (see, e.g., \cite{CP:Privileged}). 
\end{remark}

\subsection{The $\varepsilon$-Carnot coordinates} 
In what follows we will make use of a special type of polynomial Carnot coordinates.  

\begin{proposition}[\cite{CP:Carnot}] \label{prop:Carnot-coord.vareps-carnot-cor}
Let $(x_{1},\ldots,x_{n})$ be local coordinates near $a$. There is a unique polynomial change of coordinates $x\rightarrow \varepsilon_a(x)$ such that
\begin{enumerate}
 \item[(i)] This provides us with Carnot coordinates at $a$ adapted to the $H$-frame $(X_1,\ldots, X_n)$. 
 
\item[(ii)] The map $\varepsilon_a(x)$ is of the form  $\hat{\varepsilon}\circ T(x)$,  where $T(x)$ is an invertible affine map such that $T(x(a))=0$, and $\hat{\varepsilon}(x)$ is a polynomial diffeomorphism whose components $\hat{\varepsilon}_{k}(x)$, $k=1,\ldots,n$, are of the form, 
        \begin{equation}
         \hat{\varepsilon}_{k}(x)=x_{k}+\sum_{\substack{\brak\alpha \leq w_{k}\\|\alpha|\geq 2}}d_{k\alpha}x^{\alpha}, \qquad d_{k\alpha} \in \R. 
                              \label{eq:Carnot-coord.phi-vareps}
     \end{equation}
\end{enumerate}
\end{proposition}

\begin{definition}
    The Carnot coordinates provided by the change of coordinates $x\rightarrow \varepsilon_{a}(x)$ are called \emph{$\varepsilon$-Carnot coordinates}. 
    The map $\varepsilon_{a}:\R^n\rightarrow \R^n$ is called the \emph{$\varepsilon$-Carnot coordinate map}. 
\end{definition}

The $\varepsilon$-Carnot coordinates are obtained by converting to Carnot coordinates the polynomial privileged coordinates of Bella\"iche~\cite{Be:Tangent} (and their extension to Carnot manifolds in~\cite{CP:Privileged}) by means of the $w$-homogeneous change of variables described in Remark~\ref{rmk:Carnot-coord.priv-to-Carnot}. As a result we have an effective construction of Carnot coordinates. The coefficients of the matrix of the affine map $T$ and the coefficients $d_{k\alpha}$ in~(\ref{eq:Carnot-coord.phi-vareps}) are universal polynomial in the partial derivatives of the coefficients of the vector fields $X_j$, $j=1, \ldots, n$, in the original local coordinates (see~\cite{CP:Carnot}). In particular, we have the following smoothness result. 

 \begin{proposition}[\cite{CP:Carnot}]\label{prop-carnot-sm}
    The maps $(x,y)\rightarrow 
    \varepsilon_{x}(y)$ and $(x,y)\rightarrow 
    \varepsilon_{x}^{-1}(y)$ are smooth maps from $U\times \R^{n}$ to $\R^{n}$. 
\end{proposition}
 
\begin{remark}[\cite{CP:Carnot}]\label{rmk:Carnot-coord.characterization-Carnot-charts}
 The $\varepsilon$-Carnot coordinates form the building block of all systems of Carnot coordinates. More precisely,  
 the systems of Carnot coordinates at $a$ that are adapted to $(X_1,\ldots, X_n)$ are exactly the coordinates systems arising from local charts of the form $(\op{id}+\Theta)\circ \varepsilon_{\kappa(a)} \circ \kappa$, where $\kappa$ is an arbitrarily local chart near $a$ and $\Theta(x)$ is $\Ow(\|x\|^{w+1})$ near $x=0$. Incidentally, given any system of local Carnot coordinates $(x_1,\ldots, x_n)$ at $a$ that are adapted  to $(X_1,\ldots, X_n)$, a change of coordinates $x\rightarrow \phi(x)$ produces the same kind of Carnot coordinates if and only if $\phi(x)=x+ \Ow(\|x\|^{w+1})$ near $x=0$. 
\end{remark}

It is worth looking at the construction of the $\varepsilon$-Carnot coordinates in the special case of a graded nilpotent group $G$ with unit $e$.  
The Lie algebra $\fg:=TG(e)$ then has a grading $\fg=\fg_1\oplus \cdots \oplus \fg_r$ which is compatible with its Lie bracket. This grading defines a left-invariant Carnot filtration $H=(H_1,\ldots, H_r)$. Let  $(\xi_1, \ldots, \xi_n)$ be a graded basis of $\fg$, in the sense that  $\xi_j\in \fg_{w_j}$ for $j=1,\ldots, n$. For $j=1,\ldots, n$, we let $X_j$ be left-invariant vector field on $G$ such that $X_j(e)=\xi_j$. This yields a (global) left-invariant $H$-frame $(X_1,\ldots, X_n)$ of $TG$.  

As $G$ is a connected simply connected nilpotent Lie group, the exponential map $\exp:\fg \rightarrow G$ is a global diffeomorphism. Thus, we define a global diffeomorphism $\exp_X:\R^n\rightarrow G$ by letting 
\begin{align}
 \exp_X(x)&= \exp(x_1\xi_1+\cdots + x_n\xi_n) \nonumber \\
 & = \exp\left[t \left(x_1X_1+\cdots +x_nX_n\right)\right](e)|_{t=1}, \qquad x\in \R^n.
 \label{eq:Carnot-coord.canonical-coord}
\end{align}
This defines a global system of coordinates on $G$. These coordinates are the standard canonical coordinates of the first kind on $G$. They identify $G$ with $\R^n$ equipped with the Dynkin product~(\ref{eq:GM.group-law}) associated with the structure constants of $\fg$ with respect to the graded basis $(\xi_1, \ldots, \xi_n)$. 

\begin{proposition}\label{prop:Carnot-coord.graded-nilpotent-group}
In the canonical coordinates given by~(\ref{eq:Carnot-coord.canonical-coord}) we have
\begin{equation*}
 \varepsilon_y(x)=(-y)\cdot x=y^{-1}\cdot x \qquad \text{for all $x,y\in \R^n$}. 
\end{equation*}
\end{proposition}

\begin{remark}\label{rmk:Carnot-coord.graded-nilpotent-group}
Given $a\in G$, let $\lambda_a:G\rightarrow G$ be the left-multiplication by $a$. Proposition~\ref{prop:Carnot-coord.graded-nilpotent-group} 
 implies that the global chart $(\lambda_a \circ \exp_X )^{-1}:G\rightarrow \R^n$ provides us with $\varepsilon$-Carnot coordinates at $a$ adapted to $(X_1,\ldots, X_n)$ (see~\cite{CP:Carnot}).  
\end{remark}

For general Carnot manifolds we have an asymptotic version of Proposition~\ref{prop:Carnot-coord.graded-nilpotent-group} in the sense that, in Carnot coordinates at $a$, the  $\varepsilon$-Carnot coordinate map is osculated by the group law of the tangent group $GM(a)$. To make this precise we endow $\R^n\times \R^n$ with the dilations, 
\begin{equation*}
 t \cdot (x,y)=(t\cdot x, t\cdot y), \qquad x,y\in \R^n,\ t\in \R. 
\end{equation*}
We also assume that we are given a pseudo-norm $\| \cdot \|:\R^n\times \R^n \rightarrow [0,\infty)$ that satisfies~(\ref{eq:anisotropic.homogeneity-pseudo-norm}) with respect to these dilations. For instance, we may take 
 \begin{equation*}
\|(x,y)\|= |x_1|^{\frac{1}{w_1}} + |y_1|^{\frac{1}{w_1}} + \cdots + |x_n|^{\frac{1}{w_n}}+ |y_n|^{\frac{1}{w_n}},
 \qquad x,y\in \R^n.
\end{equation*}
 The pseudo-norm  $\| \cdot \|$ on $\R^n\times \R^n$ enables us to speak about $\Ow(\|(x,y)\|^{w+m})$-maps in the same way as in Definition~\ref{def:anisotropic.Thetaw}. In particular, if $\Theta(x,y)=\left(\Theta_{1}(x,y),\ldots, \Theta_{n}(x,y)\right)$ is a smooth map from 
    $U$ to $\R^{n}$, where $U$ is an open neighborhood 
    of $(0,0)\in \R^{n}\times \R^{n}$, then $\Theta(x,y)=\Ow( \|(x,y)\|^{w+m})$ near $(x,y)=(0,0)$ if and only if 
    $ \Theta_{j}(x,y)=\op{O}(\|(x,y)\|^{w_{j}+m})$ for $j=1,\ldots, n$. We then have the following approximation result. 

\begin{proposition}[\cite{CP:Carnot}]\label{prop-com}
    Let $(x_{1},\ldots,x_{n})$ be Carnot coordinates at $a$ that are adapted to the $H$-frame $(X_{1},\ldots,X_{n})$. Then, near 
    $(x,y)=(0,0)$, we have
    \begin{gather}
        \varepsilon_{y}(x)=(-y)\cdot x +\Ow\left(\|(x,y)\|^{w+1}\right),
        \label{eq:smoothness.vareps-yx1} \\ 
       \varepsilon_{y}^{-1}(x)=y\cdot x +\Ow\left(\|(x,y)\|^{w+1}\right), 
               \label{eq:smoothness.vareps-yx2}
    \end{gather}where $\cdot$ is the group law of $G(a)$ (i.e., the group law of $GM(a)$ under the identification described above).
\end{proposition}

\begin{remark}
 As we shall see, Proposition~\ref{prop-com} will be an important ingredient in the construction of tangent groupoid of a Carnot manifold in Section~\ref{sec:tangent-groupoid}. It is also an important ingredient in the construction of a full symbolic calculus for hypoelliptic pseudodifferential operators on Carnot manifolds in~\cite{CP:PsiDOs}. 
 \end{remark}
 
 \section{The Carnot Differential of a Carnot Manifold Map}\label{sec:Carnot-Differential}
 In this section, we show that the differential of any Carnot manifold map induces a smooth group bundle map between the tangent group bundles. We will see later that this map is the natural generalization to the setup of general Carnot manifolds of the Pansu derivative~\cite{Pa:AM89} for maps between nilpotent graded groups.  
 
 Throughout this section we let $(M,H)$ and $(M',H')$ be step $r$ Carnot manifolds with respective Carnot filtrations $H=(H_{1},\ldots,H_{r})$ and  
 $H'=(H_{1}',\ldots,H_{r}')$. 
 We set $n=\dim M$ and $n'=\dim M'$, and let $(w_{1},\ldots,w_{n})$ and $(w_{1}',\ldots,w_{n'}')$ be the respective weight sequences of $(M,H)$ and $(M',H')$. 
 In addition, we let $\phi:M \rightarrow M'$ be a Carnot manifold map in the sense of Definition~\ref{def:Carnot.Carnot-mfld-map}. That is, $\phi$ is a smooth map  such that $\phi'(x)\left( H_w(x)\right) \subset H'_w(\phi(x))$ for all $x\in M$ and $w=1,\ldots, r$. 

\begin{lemma}
For $w=1,\ldots, r$, the differential $\phi':TM\rightarrow TM'$ induces a smooth vector bundle map $\hat{\phi}_{[w]}':\fg_w M\rightarrow \fg_w M'$ that covers $\phi$.  
\end{lemma}
\begin{proof}
 Let $a\in M$ and set $a'=\phi(a)$. As $\phi$ is a Carnot manifold map, for $w=1,\ldots, r$,  its differential $\phi'(a)$ maps the subspace $H_w(a)$ to $H'_w(a')$, and so it descends to a  linear map $\hat{\phi}_{[w]}'(a): \fg_w M(a)\rightarrow \fg_w M'(a')$. We get a bundle map $\hat{\phi}_{[w]}':\fg_w M\rightarrow \fg_w M'$ which covers $\phi$. It  remains to show that this map is smooth. 
 
 Let $(X_1,\ldots, X_n)$ be an $H$-frame over an open neighborhood $U$ of $a$ and $(X_1',\ldots, X_{n'}')$ an $H'$-frame over an open neighborhood $U'$ of $a'$ with $\phi(U) \subset  U'$. Then $\{X_j; w_j \leq w\}$ and $\{X_k'; w_k'\leq w\}$ are smooth local frames of $H_w$ and $H_{w'}$, respectively. Here $\phi':TM\rightarrow TM'$ is a smooth vector bundle map  that sends $H_w(x)$ to $H_{w}'(\phi(x))$ for every $x\in M$. Therefore, there are functions $c_{jk}(x)\in C^\infty(U)$, $w_k'\leq w_j$, such that, for $j=1,\ldots, n$, we have  
\begin{equation}
 \phi'(x) \left( X_j(x)\right) = \sum_{w'_k \leq w_j} c_{jk}(x) X'_k\left(\phi(x)\right)\qquad \text{for all $x\in U$}.
 \label{eq:Carnot-map.hphiw-Xj} 
\end{equation}

For $j=1,\ldots, n$ and $x\in U$, let $\xi_j(x)$ be the class of $X_j(x)$ in $\fg_{w_j}  M(x)$. Then $\{\xi_j; w_j=w\}$ is a smooth frame of $\fg_w M$ over $U$. Likewise, we have a smooth frame $\{\xi_k'; w'_k=w\}$ of $\fg_{w}M'$ over $U'$, where $\xi_k'(x')$ is the class of  $X'_k(x)$ in $\fg_{w} M'$ for all $x'\in U'$. By construction $\hat\phi'_{[w]}(x)(\xi_j(x))$ is the class of $\phi'(x)(X_j(x))$ in $\fg_{w_j}M'(x)$. Combining this with~(\ref{eq:Carnot-map.hphiw-Xj}) gives
\begin{equation}
 \hat\phi'_{[w]}(x)(\xi_j(x))=  \sum_{w'_k = w_j} c_{jk}(x) \xi'_k\left(\phi(x)\right)\qquad \text{for all $x\in U$}. 
 \label{eq:Carnot-diff.smoothness-hatphiw}
\end{equation}
As the coefficients $c_{jk}(x)$ are smooth, this shows that the map $\hat\phi'_{[w]}:\fg M \rightarrow \fg M'$ is smooth. The proof is complete. 
\end{proof}

Using the gradings $\fg M = \fg_1 M\oplus \cdots \oplus \fg_r M$ and $\fg M' = \fg_1 M'\oplus \cdots \oplus \fg_r M'$, we can form the direct sum of the bundle maps $\hat{\phi}'_{[w]}$, $w=1,\ldots,r$, to get a smooth bundle map $ \hat{\phi}':\fg M\rightarrow \fg M'$ which covers $\phi$, so that we have
\begin{equation*}
   \hat{\phi}'(x)=  \hat{\phi}'_{[1]}(x)\oplus \cdots \oplus  \hat{\phi}'_{[r]}(x) \qquad \text{for all $x\in M$}. 
\end{equation*}
Since at the manifold level $GM=\fg M$ and $GM'=\fg M'$, we may also regard $\hat{\phi}'$ as a smooth bundle map from $GM$ to $GM'$. Moreover, as the construction of $\hat{\phi}'$ is compatible with the gradings  we see that, with respect to the dilations~(\ref{eq:Carnot.dilations}), we have
\begin{equation}
 \hat{\phi}'(a)(t\cdot \xi) = t \cdot  \left(\hat{\phi}'(a)\xi\right) \qquad \text{for all $(a,\xi)\in GM$ and $t\in \R$}. 
 \label{eq:Carnot-map.hatphi-homogeneity}
\end{equation}
Recall that the dilations~(\ref{eq:Carnot.dilations}) are group automorphisms (\emph{cf.}\ Remark~\ref{rmk:tangent.dilations-group-automorphisms}).

\begin{definition}
 The map $\hat{\phi}':GM\rightarrow GM'$ is called the \emph{Carnot differential of $\phi$}. For every $a\in M$, the map $\hat{\phi}'(a):GM(a)\rightarrow GM'(\phi(a))$ is called the  \emph{Carnot differential of $\phi$ at $a$} . 
\end{definition}

We have the following chain rule for Carnot differentials. 

\begin{proposition}\label{prop:Carnot.product-tangent-map}
  Let $\psi:M' \rightarrow M''$ be a Carnot manifold map, where $(M'',H'')$ is a Carnot manifold of step $r$.  Then, for all $a\in M$, we have
  \begin{equation*}
      \widehat{(\psi \circ \phi)}'(a)=\hat{\psi}'\left( \phi(a)\right) \circ \hat{\phi}'(a). 
  \end{equation*}
\end{proposition}
\begin{proof}
Let $a\in M$ and set $a'=\phi(a)$ and $a''=\psi\circ \phi(a)$. Given  $\xi\in \fg_w M(a)$, $w=1,\ldots, r$, let $X\in H_w(a)$ represent $\xi$ in $ \fg_w M(a)$. By construction 
$\phi'(a)X$ represents $\hat{\phi}'(a)\xi$ in $\fg_w M'(a')$, and so $\psi'(a')(\phi'(a)(X))$ represents  $\hat{\psi}'(a')(\hat{\phi}'(a)\xi)$ in $\fg_w M''(a'')$. Likewise, 
$(\psi \circ \phi)'(a)X$ represents  $\widehat{(\psi \circ \phi)}'(a)\xi$  in $\fg_w M''(a'')$. As $(\psi \circ \phi)'(a)= \psi'(a')\circ \phi'(a)$, we see that $ \widehat{(\psi \circ \phi)}'(a)\xi=\hat{\psi}'(a')( \hat{\phi}'(a)\xi)$. This gives the result. 
\end{proof}

\begin{proposition}\label{prop:Carnot-map.group-map}
For every $a\in M$, the Carnot differential $\hat{\phi}'(a)$ is a Lie algebra map from $\fg M(a)$ to $\fg M'(\phi(a))$ and a group map from $GM(a)$ to $GM'(\phi(a))$. 
\end{proposition}
\begin{proof}
 Let $a \in M$ and set $a'=\phi(a)$. As $GM(a)$ (resp., $GM'(a')$) is just $\fg M(a)$ (resp., $\fg M'(a')$) equipped with the Dynkin product~(\ref{eq:Carnot.Dynkin-product}), we only have to prove the compatibility of $\hat{\phi}'(a)$ with the Lie brackets of $\fg M(a)$ and $\fg M'(a')$. We shall first establish the result when $\phi$ has a smooth left-inverse $\psi:M'\rightarrow M$, i.e., $\psi \circ \phi=\op{id}$. In particular, the map $\phi$ is a smooth embedding.

\begin{claim}
For $w=1,\ldots,r$, there are an open neighborhood $U$ of $a$ in $M$ and an open neighborhood $U'$ of $\phi(U)$ in $M'$, so that, for every  section $X$ of $H_w$ over $U$, there is a section $X'$ of $H_w'$ over $U'$ such that
\begin{equation}
 \phi'(x) \big[ X(x)\big] = X'\big(\phi(x)\big) \qquad \text{for all $x\in U$}. 
 \label{eq:Carnot-map.phi-related}
\end{equation}
\end{claim}
\begin{proof}[Proof of the claim] Let $(X_1',\ldots, X'_{m'})$ be a frame of $H_w'$ over an open neighborhood $U_0'$ of $a'$ in $M'$ and $(X_1,\ldots, X_m)$ a frame of $H_w$ over an open neighborhood $U_0$ of $a$ in $M$. Set $U=U_0\cap \phi^{-1}(U_0')$ and $U'=U_0'\cap \psi^{-1}(U)$. Then $U$ and $U'$ are open neighborhoods of 
$a$ and $a'$, respectively. By definition $\phi(U)\subset U_0'$ and $\phi(U')\subset U$. As $\psi \circ \psi =\op{id}$ we also have $\phi(U) \subset \psi^{-1}(U)$, and so $\phi$ is contained in $U_0\cap \psi^{-1}(U)=U'$. Furthermore, as $\phi$ is a Carnot manifold map, for $j=1,\ldots, n$ the vector $\phi'(x)[X_j(x)]$ is in $H_w'(\phi(x))$ for all $x\in U$. Thus, as in~(\ref{eq:Carnot-map.hphiw-Xj}), there are functions $c_{jk}(x)\in C^\infty(U)$, $k=1,\ldots, m'$, such that 
\begin{equation}
 \phi'(x)\big[ X_j(x)\big] = \sum_{1\leq k \leq m'} c_{jk}(x) X_k'\left(\phi(x)\right) \qquad \text{for all $x\in U$}.
 \label{eq:Carnot-map.hphiw-Xj-Hw} 
\end{equation}

Let $X$ be a section of $H_w$ over $U$. Set $X(x)= \sum_{j=1}^m a_j(x)X_j(x)$, $a_j(x)\in C^\infty(U)$, and let $X'$ be the smooth section of  $H_w'$ over $U'$ defined by
\begin{equation*}
 X'(x')= \sum_{1\leq j \leq m} \sum_{1\leq k \leq m'} a_{j}\left(\psi(x')\right) c_{jk}\left(\psi(x')\right) X_k'(x'), \qquad x'\in U'. 
\end{equation*}
Let $x\in U$. As $\psi\circ \phi(x)=x$, we have
\begin{equation*}
 X'\big(\phi(x)\big)= \sum_{1\leq j \leq m} \sum_{1\leq k \leq m'} a_{j}\left(x\right) c_{jk}\left(x\right) X_k'(x'). 
\end{equation*}
Combining this with~(\ref{eq:Carnot-map.hphiw-Xj-Hw}) then gives
\begin{equation*}
  X'\big(\phi(x)\big)= \sum_{1\leq j \leq m} a_{j}\left(x\right)  \phi'(x)\big[ X_j(x)\big] =  \phi'(x)\big[ X(x)\big]. 
\end{equation*}
This proves the claim. 
\end{proof}
 
 In what follows, we shall say that a vector field $X$ on an open neighborhood of $a$ and a vector field $X'$ on an open neighborhood of $a'$ are $\phi$-related near $a$ when they satisfy~(\ref{eq:Carnot-map.phi-related}) near $a$. It follows from the above claim that, given any $\xi\in \fg_w M(a)$, $w\leq r$, we can find a smooth section $X$ of $H_w$ near $a$ and a smooth section $X'$ of $H_w'$ near $a'$ such that $X(a)$ represents $\xi$ in $\fg_w M(a)$ and $X$ and $X'$ are $\phi$-related near $a$. By construction $\hat\phi'(a) (\xi(a))$ is the class of $\phi'(a)[X(a)]$ in $\fg_{w} M(a')$. As $\phi'(a)[X(a)]=X'(\phi(a))=X'(a')$, we then get
\begin{equation}
 \hat\phi'(a) (\xi(a)) = \text{class of $X'(a')$ in $\fg_w M'(a')$}. 
 \label{eq:Carnot-map.hphi'-tildeX}
\end{equation}

Given $\eta \in \fg_{w'}M(a)$, $w'\leq r-w$, let $Y$ be a smooth section  $H_{w'}$ near $a$ and $Y'$  a smooth section  $H_{w'}'$ near $a'$ such that $Y(a)$ represents $\xi$ in $\fg_{w'} M(a)$ and $Y$ and $Y'$ are $\phi$-related near $a$. As in~(\ref{eq:Carnot-map.hphi'-tildeX}) $\hat{\phi}'(a)\eta$ is the class of $Y'(a')$ in $\fg_{w'} M'(a')$. The definition of the Lie bracket of $\fg'M"(a')$ then implies that
\begin{equation}
 \big[\hat{\phi}'(a)\xi,\hat{\phi}'(a)\eta\big] = \text{class of $[X',Y'](a')$ in $\fg_{w+w'} M'(a')$}.
 \label{eq:Carnot-map.hphi'-bracket}
\end{equation}

The definition of the Lie bracket of $\fg M(a)$ ensures that the Lie bracket $[\xi,\eta]$ is the class of $[X,Y](a)$ in $\fg_w M(a)$. As $X$ (resp., $Y$) is $\phi$-related to $X'$ (resp., $Y'$) near $a$, it follows from~\cite[Lemma~I.3.10]{Mi:AMS08} that $[X,Y]$ and $[X',Y']$ are $\phi$-related near $a$. Note that $[X,Y]$ and $[X',Y']$ are local sections 
of $H_{w+w'}$ and $H_{w+w'}'$, respectively. Therefore, as in~(\ref{eq:Carnot-map.hphi'-tildeX}) we have
\begin{equation*}
\hat{\phi}'(a)\big([\xi,\eta]\big) = \text{class of $[X',Y'](a')$ in $\fg_{w+w'} M'(a')$}.
\end{equation*}
Combining this with~(\ref{eq:Carnot-map.hphi'-bracket}) then shows that $\hat{\phi}'(a)\big([\xi,\eta]\big)=\big[\hat{\phi}'(a)\xi,\hat{\phi}'(a)\eta\big]$. This proves the compatibility of $\hat{\phi}'(a)$ with the Lie bracket when $\phi$ is a left-invertible Carnot manifold map.

Suppose now that $\phi:M\rightarrow M'$ is an arbitrary Carnot manifold map. We reduce to the left-invertible case as follows. Set $\tilde{M}=M\times M'$ and $\tilde{a}=(a,a')$. Let $\pi_1:\tilde{M} \rightarrow  M$ and $\pi_2:\tilde{M} \rightarrow  M'$ be the first factor and second factor projections. The differential $\pi_1'$ (resp., $\pi_2'$) induces a vector bundle isomorphism from the normal bundle $\ker \pi_2'$ (resp., $\ker \pi_1'$) onto $\pi_1^*TM$ (resp., $\pi_2^*TM'$), and so we get a natural vector bundle identification between $TM=\ker \pi_2' \oplus \ker \pi_1'$ and $\pi_1^*TM \oplus \pi_2^*TM'$. Define 
\begin{equation*}
 \tilde{H}_w=\pi_1^*H_w \oplus  \pi_2^*H_w', \qquad w=1, \ldots, r,
\end{equation*}
where $\pi_1^*H_w$ (resp., $\pi_2 ^*H_w'$) is seen as a sub-bundle of $\ker \pi_2'$ (resp., $\ker \pi_1'$). If $w+w'\leq r$, then
\begin{equation}
 \big[ \tilde{H}_w, \tilde{H}_{w'}\big]= \pi_1^*\left[H_w,H_{w'}\right]\oplus \pi_2^*[H_{w}',H_{w'}'] 
 \subset \pi_1^*H_{w+w'} \oplus \pi_2^*H_{w+w'}'= \tilde{H}_{w+w'}. 
 \label{eq:Carnot-map.bracket-tH}
\end{equation}
This shows that $\tilde{H}:=(\tilde{H}_1, \ldots, \tilde{H}_r)$ is a Carnot filtration, and so $(\tilde{M},\tilde{H})$ is a Carnot manifold. 

The projections $\pi_1:\tilde{M}\rightarrow M$ and $\pi_2:\tilde{M} \rightarrow M'$ are both Carnot manifold maps. Furthermore, thanks to~(\ref{eq:Carnot-map.bracket-tH}) we have a Lie algebra isomorphism,
\begin{equation*}
 \hat{\pi}_1'(\tilde{a}) \oplus \hat{\pi}_2'(\tilde{a}): \fg \tilde{M}(\tilde{a}) \stackrel{\sim}{\longrightarrow} \fg M(a) \oplus \fg M'(a').
\end{equation*}
Incidentally, the Carnot differentials $\hat{\pi}_1'(\tilde{a}):\fg \tilde{M}(\tilde{a})  \rightarrow \fg M(a)$ and $\hat{\pi}_2'(\tilde{a}): \fg \tilde{M}(\tilde{a})  \rightarrow \fg M'(a')$ are both Lie algebra maps. 

Let $\Phi: M\rightarrow \tilde{M}$ be the smooth map defined by 
\begin{equation*}
 \Phi(x)=\left(x,\phi(x)\right), \qquad x\in M. 
\end{equation*}
Given any $x\in M$, for $w=1,\ldots, r$, we have 
\begin{equation*}
 \Phi'(x)\big(H_w(x)\big)= H_w(x) \oplus \phi'(x)\big(H_w(x)\big)\subset H_w(x) \oplus H_{w}(\phi(x)) =\tilde{H}_w(\Phi(x)). 
\end{equation*}
This shows that $\Phi$ is a Carnot manifold map. As $\pi_1\circ \Phi =\op{id}$, we also see that $\Phi$ is has smooth left-inverse. Therefore, by the first part of the proof  the Carnot differential $\widehat{\Phi}'(\tilde{a}): \fg M (a)\rightarrow \fg \tilde{M}(\tilde{a})$ is a Lie algebra map. 

Note that $\phi = \pi_2 \circ \Phi$, and so by Proposition~\ref{prop:Carnot.product-tangent-map} we have $\hat{\phi}'(a) = \hat{\pi}_2'(\tilde{a}) \circ \hat{\Phi}'(a)$. As $\hat{\Phi}'(a): \fg M (a)\rightarrow \fg \tilde{M}(\tilde{a})$  and $ \hat{\pi}_2'(\tilde{a}): \fg \tilde{M}(\tilde{a}) \rightarrow \fg M'(a')$ are both Lie algebra maps, we then deduce that $\hat{\phi}'(a):\fg M(a)\rightarrow \fg M'(a')$ is a Lie algebra map. As mentioned above, this  implies that $\hat{\phi}'(a)$ is compatible with the product laws of $GM(a)$ and $GM'(a')$. The proof is complete. 
\end{proof}


Combining Proposition~\ref{prop:Carnot.product-tangent-map} and Proposition~\ref{prop:Carnot-map.group-map} leads us to the following functoriality result. 

\begin{proposition}\label{prop:Carnot-diff.functor}
 The assignment $(M,H)\rightarrow GM$ is a functor from the category of (step $r$) Carnot manifolds  to the category of smooth bundles of (step $r$) graded nilpotent Lie groups. 
\end{proposition}

In particular, in the case of Carnot diffeomorphisms we have the following statement.   

\begin{proposition}\label{prop:Carnot.inverse-tangent-map}
 Suppose that $\phi:M\rightarrow M'$ is a Carnot diffeomorphism. Then, for every $a\in M$, the Carnot differential $\hat{\phi}'(a)$ is a group isomorphism from $GM(a)$ onto $GM'(\phi(a))$ such that
        \begin{equation*}
           \hat{\phi}'(a)^{-1}= \widehat{\left(\phi^{-1}\right)}'\left( \phi(a)\right).
        \end{equation*}
\end{proposition}
 
\begin{remark}
 It is not necessary to require the map $\phi$ to be smooth to be able to define its Carnot differential at a given point $a\in M$. We only need $\phi$ to be differentiable at $a$ in such a way its differential $\phi'(a)$ maps $H_j(a)$ to $H'_j(\phi(a))$ for $j=1,\ldots, n$. The proof of Proposition~\ref{prop:Carnot-map.group-map} requires some extra regularity, but its arguments still go through when $\phi$ is $C^2$ near the point $a$. Thus, at least under this regularity assumption, the Carnot differential $\hat{\phi}'(a)$ is a group map. It would be interesting to obtain this property under a weaker regularity assumption. 
\end{remark}

 \section{Carnot Differential and Tangent Approximation}\label{sec:Tangent-approximation}
 In this section, we show that, in Carnot coordinates, any Carnot manifold map is approximated in a very precise way by its Carnot differential.  
 This result recasts in a differentiable context and extends to Carnot manifold maps between general Carnot manifolds the approximation results of Pansu~\cite{Pa:AM89} for locally Lipschitz continuous maps and quasi-conformal maps between Carnot groups, and their generalizations to quasi-conformal maps between ECC manifolds 
 by Margulis-Mostow~\cite{MM:GAFA95}. 
 
 Throughout this section we let $\phi:M\rightarrow M'$ be a Carnot manifold map, where $(M,H)$ and $(M',H')$ are step $r$ Carnot manifolds. We shall keep using the notation of the previous section. We also let $a\in M$ and set $a'=\phi(a)$. In addition, we let  $(X_{1},\ldots,X_{n})$ be an 
 $H$-frame on an open neighborhood $U_0$ of $a$ and $(X_{1}',\ldots,X_{n}')$ an 
 $H'$-frame  on an open neighborhood $U_0'$  of $a'$ containing $\phi(U_0)$. As in~(\ref{eq:Carnot-map.hphiw-Xj}) the fact that $\phi$ is a Carnot manifold map ensures us there are functions $c_{jk}(x)\in C^\infty(U_0)$, $w_k'\leq w_j$, such that on $U_0$ we have
\begin{equation}
     \phi'(x)\left( X_{j}(x)\right)=\sum_{w_{k}'\leq w_{j}}c_{jk}(x)X_{k}'\left(\phi(x)\right), \qquad j=1,\ldots,n. 
     \label{eq-phistar}
 \end{equation} 

As in Remark~\ref{rmk:Tangent-group.frame-fgM}, the $H$-frame $(X_1, \ldots, X_n)$ gives rise to a graded basis $(\xi_1(a), \ldots, \xi_n(a))$ of $\fg M(a)$, where $\xi_j(a)$ is the class of $X_j(a)$ in $\fg_{w_j}M(a)$. This graded basis defines a global system of coordinates that identifies the tangent group $GM(a)$ with the graded nilpotent Lie group $G(a)$.  This group is obtained by endowing $\R^n$ with the Dynkin product~(\ref{eq:GM.group-law}) associated with the structure constants $L_{ij}^k(a)$, $w_i+w_j=w_k$,  defined by the commutator relations~(\ref{rmk:Carnot-mfld.brackets-H-frame}). 

Likewise, the commutator relations~(\ref{rmk:Carnot-mfld.brackets-H-frame}) for $X_1', \ldots,X_{n'}'$ uniquely defines coefficients $\tilde{L}_{ij}^k(x')$ in $C^\infty(U_0')$ with $w_k'\leq w_i'+w_j'$. Let  $G'(a')$ be the graded nilpotent Lie group obtained by equipping $\R^{n'}$ with the Dynkin product~(\ref{eq:GM.group-law}) associated with the coefficients $\tilde{L}_{ij}^k(a')$, $w_i'+w_j'=w_k'$. As above, the $H'$-frame $(X_1',\ldots, X_{n'}')$ gives rise to a graded basis $(\xi_1'(a'), \ldots, \xi_{n'}'(a'))$ of $\fg M'(a')$ which defines a system of coordinates that identifies $GM'(a')$ with $G'(a')$.  

The Carnot differential $\hat{\phi}'(a)$ is a Lie group map from $GM(a)$ to $GM'(a')$. Using the identifications described above we may regard it as a Lie group map, 
\begin{equation}
 \hat{\phi}'(a):G(a) \longrightarrow G'(a'). 
 \label{eq:Tangent-approx.hphi'(a)-G(a)-G'(a')}
\end{equation}
 Note that~(\ref{eq:Carnot-map.hatphi-homogeneity}) implies that the above map is $w$-homogeneous. Recall also that $\hat{\phi}'(a)$ was originally defined as a linear map from $\fg M(a)=GM(a)$ to $\fg M'(a')=GM'(a')$. Therefore, under the identification~(\ref{eq:Tangent-approx.hphi'(a)-G(a)-G'(a')}) above we obtain a $w$-homogeneous linear map from $\R^n=G(a)$ to $\R^{n'}=G'(a')$. In particular, the map~(\ref{eq:Tangent-approx.hphi'(a)-G(a)-G'(a')}) is given by an $n'\times n$-matrix $(\hat{\phi}_{kj}'(a))$. In fact, it follows from~(\ref{eq:Carnot-diff.smoothness-hatphiw}) that we have
\begin{equation*}\label{eq-mat-3}
 \hat{\phi}_{kj}'(a)= \left\{ 
\begin{array}{ll}
 c_{jk}(a) & \text{if $w_j=w_k'$},\\
 0 & \text{otherwise}. 
\end{array}\right. 
\end{equation*}

 In what follows, we let $(x_{1},\ldots,x_{n})$ be privileged coordinates at $a$ adapted to 
 $(X_{1},\ldots,X_{n})$ with range $U$. We also let $(x_{1}',\ldots,x_{n'}')$ be privileged coordinates at $a'$ adapted to
 $(X_{1},\ldots,X_{n'})$ with range $U'\supset \phi(U)$. In these coordinates the map $\phi$ then appears as a smooth map from $U$ to $U'$. Note that $U$ and $U'$ are open neighborhoods of the origins of $\R^n$ and $\R^{n'}$, respectively. 
 
In the coordinates $(x_{1},\ldots,x_{n})$, for $j=1,\ldots, n$, we may write
  \begin{equation*}
        X_{j}(x)=\sum_{1\leq l \leq n}b_{jp}^{X}(x)\partial_{x_{p}}, \qquad b_{jp}^{X}(x)\in C^\infty(U). 
     \end{equation*}
  Likewise, in the coordinates $(x_{1}',\ldots,x_{n'}')$, for $k=1, \ldots, n'$, we have
  \begin{equation*}
       X_{k}'(x')=\sum_{1\leq l\leq  n'}b_{kq}^{X'}(x')\partial_{x_{q}'}, \qquad b_{kq}^{X'}(x')\in C^\infty(U'). 
     \end{equation*}
 If we regard the coefficients $c_{jk}(x)$, $w_k'\leq w_j$, in~(\ref{eq-phistar}) as smooth functions on $U$, then in terms of the coefficients $b_{jp}^{X}(x)$ and $b_{kq}^{X'}(x')$ above the relations~(\ref{eq-phistar}) imply that, for $j=1, \ldots, n$ and $k=1, \ldots, n'$, we have 
     \begin{equation}
       \sum_{1\leq p \leq n}b_{jp}^{X}(x)\partial_{x_{p}}\phi_{q}(x) = \sum_{w_{k}'\leq w_{j}} c_{jk}(x) b_{kq}^{X'}\left(\phi(x)\right), \qquad x\in U. 
       \label{eq:Carnot-prop.phi'XX'-coefficients}
     \end{equation}
  
 \begin{lemma}\label{lem:Carnot-prop.phi-hatphi-Ow}
 There is a $w$-homogeneous polynomial map $\tilde{\phi}:\R^n \rightarrow \R^{n'}$ such that 
 \begin{equation}
     \phi(x)=\tilde{\phi}(x)+\Ow(\|x\|^{w+1}) \qquad \text{near $x=0$}, 
     \label{eq:Carnot-prop.phi-hatphi-Ow}
 \end{equation}
  \end{lemma}
 \begin{proof}
Proposition~\ref{lem:multi.Thetat} and Remark~\ref{rmk:anisotropic.weighted-asymptotic} it is enough to show that $\phi(x)=\Ow(\|x\|^{w})$ near $x=0$. 
We also know by Lemma~\ref{lem-eq-we} that $\phi(x)=\Ow(\|x\|^{w})$ near $x=0$ if and only if, for every $q=1,\ldots, n'$, the component $\phi_q(x)$ has weight~$\geq w_q'$. That is, we have
     \begin{equation}
     \partial_{x}^{\alpha}\phi_{q}(0)=0 \qquad \text{whenever $w_{q}'-\brak \alpha>0$}. 
     \label{eq:Carnot-prop.weight-phi}
  \end{equation}
  Therefore, we only need to establish~(\ref{eq:Carnot-prop.weight-phi}) to prove the lemma. We shall prove~(\ref{eq:Carnot-prop.weight-phi}) by induction on $|\alpha|$. Let us say that~(\ref{eq:Carnot-prop.weight-phi}) holds up to order $m$ when it holds 
     for every multi-order $\alpha$ such that $w_{q}'-\brak \alpha>0$ and $|\alpha|\leq m$. We note that it holds up to 
     order $0$ since $\phi(0)=0$. 
     
     It remains to show that if~(\ref{eq:Carnot-prop.weight-phi}) holds up to order~$m$, then it holds up to 
     order~$m+1$. To reach this end we observe that, as the coordinates $(x_{1},\ldots,x_{n})$ are privileged coordinates at 
     $a$ adapted to the $H$-frame $(X_{1},\ldots, X_{n})$, we know that the vector field $X_{j}$ has weight $-w_{j}$ 
     and agrees with $\partial_{x_{j}}$ at $x=0$. The former property means that each coefficient $b_{jp}^{X}(x)$ has 
     weight~$\geq w_{p}-w_{j}$. Therefore, we see that, for $j=1,\ldots,n$ and $p=1,\ldots,n$, we have
     \begin{equation}
         b_{jp}^{X}(0)=\delta_{jp} \qquad \text{and} \qquad \partial^{\alpha}_{x}b_{jp}^{X}(0)=0  \quad \text{whenever 
         $w_{p}-w_{j}-\brak \alpha >0$}.
         \label{eq:Carnot-prop.bjlX-order-wj}
     \end{equation}Likewise, for $k=1,\ldots,n'$ and $q=1,\ldots,n'$, we have 
         \begin{equation}
         b_{kq}^{X'}(0)=\delta_{kq} \qquad \text{and} \qquad \partial^{\alpha}_{x'}b_{kq}^{X'}(0)=0  \quad \text{whenever 
         $w_{q}'-w_{k}'-\brak \alpha >0$}.
          \label{eq:Carnot-prop.bjlX'-order-w'l}
     \end{equation}
     
     Suppose now that~(\ref{eq:Carnot-prop.weight-phi}) holds up to order $m$. Given $j\in \{1,\ldots,n\}$ and $q\in \{1,\ldots,n'\}$, 
     let $\alpha \in \N_{0}^{n}$ be such that $w_{q}'-w_{j}-\brak \alpha>0$ and $|\alpha|\leq 
     m$. Using the equality $b_{jp}^{X}(0)=\delta_{jp}$, we see that the partial derivative of order $\alpha$ at $x=0$ of the l.h.s.~of~(\ref{eq:Carnot-prop.phi'XX'-coefficients}) is equal to
     \begin{equation*}
        \partial_{x}^{\alpha}\partial_{x_{j}}\phi_{q}(0)+ 
         \sum_{\substack{\beta+\gamma=\alpha\\ \gamma \neq 
         \alpha}}\binom{\alpha}{\beta}\partial_{x}^{\beta}b_{jp}^{X}(0)\partial^{\gamma}_{x}\partial_{x_{p}}\phi_{q}(0).
     \end{equation*}
     We claim that in the summation above each term 
     $\partial_{x}^{\beta}b_{jp}^{X}(0)\partial^{\gamma}_{x}\partial_{x_{p}}\phi_{q}(0)$ is zero. To see this we observe 
     that
     \begin{equation*}
         (w_{p}-w_{j}-\brak \beta)+(w_{q}'-w_{p}-\brak \gamma)=w_{q}'-w_{j}-\brak \alpha>0.
     \end{equation*}Therefore, at least one the integers $w_{p}-w_{j}-\brak \beta$ or $w_{q}'-w_{p}-\brak \gamma$ must be 
     positive. If the former is positive, then~(\ref{eq:Carnot-prop.bjlX-order-wj}) ensures us that $\partial_{x}^{\beta}b_{jp}^{X}(0)=0$. If 
     $w_{q}'-w_{p}-\brak \gamma>0$, then, as $\gamma \neq \alpha$, we have 
     $|\gamma|+1\leq (|\alpha|-1)+1\leq m$. As~(\ref{eq:Carnot-prop.weight-phi}) holds up to order $m$, we 
     see that 
     $\partial^{\gamma}_{x}\partial_{x_{p}}\phi_{q}(0)=0$. In any case  
     $\partial_{x}^{\beta}b_{jq}^{X}(0)\partial^{\gamma}_{x}\partial_{x_{p}}\phi_{q}(0)$ must be zero. We then deduce 
     that the partial derivative of order $\alpha$ at $x=0$ of the l.h.s.~of~(\ref{eq:Carnot-prop.phi'XX'-coefficients}) is equal to 
     $\partial_{x}^{\alpha}\partial_{x_{j}}\phi_{q}(0)$. 
       
     Bearing this in mind, the partial derivative of order $\alpha$ at $x=0$ of the r.h.s.~of~(\ref{eq:Carnot-prop.phi'XX'-coefficients}) is equal to 
     \begin{equation}
        \sum_{\substack{\beta+\gamma=\alpha\\ w_{k}'\leq 
         w_{j}}}\binom{\alpha}{\beta}\partial_{x}^{\beta}c_{jk}(0)\partial^{\gamma}_{x}(b_{kq}^{X'}\circ \phi)(0).
         \label{eq:Carnot-prop.phi'XX'-coefficients-darhs}
     \end{equation}
     
     \begin{claim}
         Let $\gamma \in \N_{0}^{n}$ and $b\in C^{\infty}(U')$. Then $\partial^{\gamma}_{x}(b\circ \phi)(x)$ is a linear 
         combination of terms of the form
         \begin{equation}
           \partial^{\gamma_{1}}_{x}\phi_{q_{1}}(x) \cdots   
           \partial^{\gamma_{\ell}}_{x}\phi_{q_{\ell}}(x)(\partial_{x_{q_{1}}'}\cdots  \partial_{x_{q_{\ell}}'}b)(\phi(x)),
           \label{eq:Carnot-prop.higher-order-chain-rule}
         \end{equation}where $\ell$ ranges over $\{0,\ldots,|\gamma|\}$, the integers $q_{1},\ldots,q_{\ell}$ range over 
         $\{1,\ldots,n'\}$ and $(\gamma_{1},\ldots,\gamma_{\ell})$ ranges over all $p$-tuples in $(\N_{0}^{n})^{p}$ such that 
         $\gamma_{1}+\cdots +\gamma_{\ell}=\gamma$. 
     \end{claim}
     \begin{proof}
         This claim can be obtained as a by-product of the multivariate higher-order chain rule (a.k.a.\ multivariate Fa\`a di 
         Bruno formula). For our purpose we don't need the precise expressions of the coefficients of the various terms~(\ref{eq:Carnot-prop.higher-order-chain-rule}). 
         We thus can proceed to prove the claim by induction on $|\gamma|$ as follows. Let us say that a term of the 
         form~(\ref{eq:Carnot-prop.higher-order-chain-rule}) has order $m$, with $m=|\gamma|$.  
         We note that the claim at order $1$ is an immediate consequence of the multivariate chain rule. 
         Furthermore, this rule also implies that if  we apply a partial derivative $\partial_{x_{i}}$ to a term of the 
         form~(\ref{eq:Carnot-prop.higher-order-chain-rule}), then we get a sum of terms of the 
         form~(\ref{eq:Carnot-prop.higher-order-chain-rule}) of the next order. Thus, if the claim is true at a given 
         order $m$, then it is true at order $m+1$. It then follows that it holds at any order. This proves the claim. 
      \end{proof}
     
     Thanks to the claim above we know that in~(\ref{eq:Carnot-prop.phi'XX'-coefficients-darhs}) each partial derivative 
     $\partial^{\gamma}_{x}(b_{kq}^{X'}\circ\phi)(0)$ is a linear combination of terms of the form,
     \begin{equation}
          \partial^{\gamma_{1}}_{x}\phi_{q_{1}}(0) \cdots   
           \partial^{\gamma_{\ell}}_{x}\phi_{q_{\ell}}(0)(\partial_{x_{q_{1}}'}\cdots  \partial_{x_{q_{\ell}}'}b_{kq}^{X'})(0),
           \label{eq:Carnot-prop.higher-order-chain-rule-bX'}
         \end{equation}where $\ell,q_{1},\ldots,q_{\ell},\gamma_{1},\ldots,\gamma_{\ell}$ are as in~(\ref{eq:Carnot-prop.higher-order-chain-rule}). We observe that
         \begin{multline*}
             (w_{q_{1}}'-\brak{\gamma_{1}})+\cdots +  (w_{q_{\ell}}'-\brak{\gamma_{\ell}})+(w_{q}'-w_{k}'-(w_{q_{1}}'+\cdots 
             +w_{q_{\ell}}')) \\
             = w_{q}'-w_{k}'-\brak \gamma\geq w_{q}'-w_{j}-\brak \alpha>0.
         \end{multline*}
         Therefore, at least one of the numbers $w_{q_{s}}'-\brak{\gamma_{s}}$, $s=1,\ldots,\ell$, or 
         $w_{q}'-w_{k}'-(w_{q_{1}}'+\cdots 
             +w_{q_{\ell}}')$ must be positive. If $w_{q_{s}}'-\brak{\gamma_{s}}>0$, then, as $|\gamma_{s}|\leq 
             |\gamma|\leq |\alpha|\leq m$, the induction assumption implies that
             $\partial_{x}^{\gamma_{s}}\phi_{q_{s}}(0)=0$. 
             If $w_{q}'-w_{k}'-(w_{q_{1}}'+\cdots 
             +w_{q_{\ell}}')>0$, then~(\ref{eq:Carnot-prop.bjlX'-order-w'l}) implies that $(\partial_{x_{q_{1}}'}\cdots  
             \partial_{x_{q_{\ell}}'}b_{kq}^{X'})(0)=0$. In any case we see that all the terms~(\ref{eq:Carnot-prop.higher-order-chain-rule-bX'}) vanish, 
             and hence $\partial^{\gamma}_{x}(b_{kq}^{X'}\circ 
     \phi)(0)=0$. It then follows that the partial derivative of order $\alpha$ at $x=0$ of the r.h.s.~of~(\ref{eq:Carnot-prop.phi'XX'-coefficients}) is 
     zero. As the partial derivative of order $\alpha$ at $x=0$ of the l.h.s.~agrees with 
     $\partial_{x}^{\alpha}\partial_{x_{j}}\phi_{q}(0)$, we deduce that 
     \begin{equation*}
      \partial_{x}^{\alpha}\partial_{x_{j}}\phi_{q}(0)= 0 \qquad \text{whenever $w_{q}-\brak \alpha-w_{j}>0$ and 
      $|\alpha|\leq m$}.
     \end{equation*}
     This proves that~(\ref{eq:Carnot-prop.weight-phi}) holds up to order~$m+1$. The proof is complete. 
 \end{proof}
 
We shall now identify $\tilde{\phi}(x)$. In what follows, we denote by $G^{(a)}$ the nilpotent approximation of $(M,H)$ at $a$ in the privileged coordinates $(x_1,\ldots, x_n)$ and, for $j=1,\ldots,n$, we let $X_j^{(a)}$ be the model vector field of $X_j$ at $a$ in these coordinates.  Likewise, we denote by $G^{'(a')}$ the nilpotent approximation of  $(M',H')$ at $a'$ in the privileged coordinates $(x_1',\ldots, x_{n'}')$ and, for $k=1,\ldots, n'$, we let $X^{'(a')}_k$ be the model vector field of $X_k'$ at $a'$ in these  coordinates. As in~(\ref{eq:Carnot-coord.canonical-coord}) we define the exponential maps $\exp_{X^{(a)}}: \R^n \rightarrow \R^n$ and  $\exp_{X^{'(a')}}: \R^{n'} \rightarrow \R^{n'}$ by 
\begin{gather*}
 \exp_{X^{(a)}}(x)= \exp\left( x_1 X_1^{(a)} +\cdots + x_n X_n^{(a)}\right), \qquad x=(x_1,\ldots, x_n)\in\R^n,\\
 \exp_{X^{'(a')}}(x')= 
\exp\left( x_1' X_1^{'(a')} +\cdots + x_{n'}' X_{n'}^{'(a')}\right), 
 \qquad x'=(x_1',\ldots, x_{n'}')\in\R^{n'}. 
\end{gather*}
Note that  $\exp_{X^{(a)}}$ is a $w$-homogeneous polynomial diffeomorphism. This is also is a Lie group isomorphism from $G(a)$ onto $G^{(a)}$, since it identifies $G^{(a)}$ with $\R^n$ equipped with the Dynkin product~(\ref{eq:GM.group-law}). Similarly, $\exp_{X^{'(a')}}$ is a $w$-homogeneous polynomial Lie group isomorphism from $G'(a')$ onto $G^{'(a')}$. 
 
 \begin{lemma}\label{lem:Carnot-prop.hatphi=phi'H}
   We have
\begin{equation}
 \tilde{\phi} = \exp_{X^{'(a')}} \circ \hat{\phi}'(a) \circ \exp_{X^{(a)}}^{-1}. 
 \label{eq:Carnot-prop.tphi-hatphi}
\end{equation}
In addition, $\tilde{\phi}$ is a Lie group map from $G^{(a)}$ to $G^{'(a')}$. 
 \end{lemma}
 \begin{proof}
The proof of~(\ref{eq:Carnot-prop.tphi-hatphi}) is based on the following claim. 
 
 \begin{claim}
 For $ j=1,\ldots, n$,  we have
\begin{equation}
 \tilde{\phi}'(x)\big[ X_j^{(a)}(x)\big] = \sum_{w'_k=w_j} c_{jk}(a) X^{'(a')}_k\big(\tilde{\phi}(x)\big) \qquad  \text{for all $x\in \R^n$}. 
 \label{eq:tangent-approx.tphiXj(a)-Xk(a')tphi} 
\end{equation}
 \end{claim}
 \begin{proof}[Proof of the Claim]
 For $j=1,\ldots,n$ and $k=1,\ldots,n'$,  set 
 \begin{equation}
     X_{j}^{(a)}(x)=\sum_{1\leq p \leq n}b_{jp}^{(a)}(x)\partial_{x_{p}} \qquad \text{and} \qquad 
     X_{k}^{'(a')}(x)=\sum_{1 \leq q \leq n'}b_{kq}^{(a')}(x')\partial_{x_{q}'},
     \label{eq:tangent-map.Xj(a)-X'k(a')}
 \end{equation}
 As $X_j$ is a homogenous polynomial vector field of weight $-w_j$, each coefficient $b_{jp}^{(a)}(x)$ is a (weighted) homogeneous polynomial of degree $w_p-w_j$, and so  $b_{jp}^{(a)}(x)=0$ when $w_p<w_j$. 
 Similarly, each coefficient $b_{kq}^{(a')}(x')$ is  a (weighted) homogeneous polynomial of degree $w_q'-w_j$, and so $b_{kq}^{(a')}(x')=0$ when $w_q'<w_k'$.
 
As $(x_{1},\ldots,x_{n})$ are privileged coordinates at $a$ adapted to $(X_{1},\ldots,X_{n})$, for $j=1,\ldots,n$ and as $t\rightarrow 0$, we have
 \begin{equation*}
     t^{w_{j}}\delta_{t}^{*}X_{j}=X_{j}^{(a)}+\op{O}(t) \qquad \text{in $\cX(U)$}. 
 \end{equation*}
 In terms of the coefficients $b_{jp}^{X}(x)$, $p=1, \ldots, n$,  this implies that, as $t\rightarrow 0$, we have
 \begin{equation}
    t^{w_{j}-w_{p}}b_{jp}^{X}(t\cdot x)=b_{jp}^{(a)}(x)+\op{O}(t) \qquad \text{in $C^{\infty}(U)$}. 
       \label{eq:tangent-map.rescaling-Xj}
 \end{equation}
 Likewise, for $k=1,\ldots,n'$ and $q=1,\ldots,n'$, and as $t\rightarrow 0$, we have
 \begin{equation*}
     t^{w_{k}'-w_{q}'}b_{kq}^{X'}(x')=b_{kq}^{(a')}(x')+\op{O}(t) \qquad \text{in $C^{\infty}(U')$}. 
 \end{equation*}
By Lemma~\ref{lem-eq-we} this implies that, if we set $\cU'=\{(x',t)\in U' \times \R; \ t\cdot x' \in U'\}$, then there are functions $\Theta_{kq}(x',t)\in C^\infty(\cU')$, $k,q=1,\ldots, n'$,  such that 
 \begin{equation}
 t^{w_{k}'-w_{q}'}b_{kq}^{X'}(t\cdot x')=b_{kq}^{(a')}(x')+t \Theta_{kq}(x',t) \qquad \text{for all $(x',t)\in \cU'$}.
 \label{eq:tangent-map.rescaling-Xk'} 
\end{equation}
 
 We also observe that~(\ref{eq:Carnot-prop.phi-hatphi-Ow}) and Lemma~\ref{lem-eq-we} imply that, as $t\rightarrow 0$, we have
 \begin{equation}
     t^{-1}\cdot \phi(t\cdot x)=\tilde{\phi}(x)+\op{O}(t) \qquad \text{in $C^{\infty}(U,\R^{n'})$}.
     \label{eq:tangent-map.rescaling-phi}
 \end{equation}In particular, we may termwise differentiate the asymptotics and see that, given any $x\in \R^{n}$, for $p=1,\ldots,n$ and 
 $q=1,\ldots,n'$,  we have
 \begin{equation}
     t^{w_{p}-w_{q}'}\partial_{x_{p}}\phi_{q}(t\cdot x)=\partial_{x_{p}}\tilde{\phi}_{q}(x) +\op{O}(t) \qquad \text{as 
     $t\rightarrow 0$}.
     \label{eq:tangent-map.rescaling-dphi}
 \end{equation}
 Moreover, by combining~(\ref{eq:tangent-map.rescaling-Xk'}) and~(\ref{eq:tangent-map.rescaling-phi}) we see that, as $t\rightarrow 0$, we have 
\begin{align}
  t^{w_{k}'-w_{q}'}b_{kq}^{X'}\big(\phi(t\cdot x)\big) & = t^{w_{k}'-w_{q}'}b_{kq}^{X'}\left[t \cdot \big(t^{-1}\cdot \phi(t\cdot x)\big)\right] \nonumber \\ 
  & = b_{kq}^{(a')}\big(t^{-1}\cdot \phi(t\cdot x)\big)+t \Theta_{kq}\big(t^{-1}\cdot \phi(t\cdot x), t\big)
  \label{eq:tangent-map.rescaling-Xk'-tphi} \\
  &  = b_{kq}^{(a')}\big(\tilde{\phi}(x)\big)+ \op{O}(t). \nonumber
\end{align}
 
 Given $x\in \R^n$ and $t\in \R\setminus 0$ such that $t\cdot x\in U$, substituting $t\cdot x$ for $x$ in~(\ref{eq:Carnot-prop.phi'XX'-coefficients}) and multiplying both sides by 
 $t^{w_{j}-w_{q}'}$ gives
     \begin{equation}
       \sum_{1\leq p \leq n}t^{w_{j}-w_{p}}b_{jp}^{X}(t\cdot x)t^{w_{p}-w_{q}'}\partial_{x_{p}}\phi_{q}(t\cdot x) \\  = 
       \sum_{w_{k}'\leq w_{j}} t^{w_{j}-w_{k}'}c_{jk}(t\cdot x) 
       t^{w_{k}'-w_{q}'}b_{kq}^{X'}\big[\phi(t\cdot x))\big]. 
       \label{eq:tangent-map.rescaling-equation-Xj-phi}
     \end{equation}
Letting $t\rightarrow 0$ and using~(\ref{eq:tangent-map.rescaling-Xj}), (\ref{eq:tangent-map.rescaling-dphi}) and~(\ref{eq:tangent-map.rescaling-Xk'-tphi}) gives
     \begin{equation}
       \sum_{1\leq p \leq n} b_{jp}^{(a)}(x)\partial_{x_{p}}\tilde{\phi}_{q}(x) = 
       \sum_{w_{k}'=w_{j}} c_{jk}(a) b_{kq}^{(a')}(\tilde{\phi}(x)). 
       \label{eq:tangent-map.tphiXj(a)-X_k(a')tphi}
     \end{equation}
 Note that we have
\begin{equation*}
 \tilde{\phi}'(x)\big[ X_j^{(a)}(x)\big] = \sum_{1\leq p \leq n} b_{jp}^{(a)}(x)\tilde{\phi}'(x) [\partial_{x_p}] =  
 \sum_{1\leq q \leq n'} \sum_{1\leq p \leq n} b_{jp}^{(a)}(x)   \partial_{x_p} \tilde{\phi}_q(x) \partial_{x'_q}. 
\end{equation*}
 Combining this with~(\ref{eq:tangent-map.Xj(a)-X'k(a')}) and~(\ref{eq:tangent-map.tphiXj(a)-X_k(a')tphi}) we then obtain
\begin{equation*}
  \tilde{\phi}'(x)\big[ X_j^{(a)}(x)\big] = \sum_{1\leq q \leq n'}  \sum_{w_{k}'=w_{j}} c_{jk}(a) b_{kq}^{(a)'}(\tilde{\phi}(x))\partial_{x'_q} =   \sum_{w_{k}'=w_{j}} c_{jk}(a) X_{k}^{'(a')}\big ( \tilde{\phi}(x)\big). 
\end{equation*}
 The claim is thus proved.
 \end{proof}
 
Let us go back to the proof of the lemma. Let $x\in \R^n$, and set $x'=\hat{\phi}'(a)x$. Note that in view of~(\ref{eq-mat-3}) we have
\begin{equation*}
 x_k'= \sum_{0\leq j \leq n}\hat{\phi}'(a)_{kj}x_j = \sum_{w_j=w_k'} c_{jk}(a)x_j.  
\end{equation*}
For $t\geq 0$, we also set
\begin{equation*}
 x(t)= \exp\big[t\big(x_1X_1^{(a)}+\cdots +x_nX_n^{(a)}\big)\big](0), \quad y(t)= \exp\big[t\big(x_1'X_1^{'(a')}+\cdots +x_{n'}X_{n'}^{'(a')}\big)\big](0). 
\end{equation*}
By definition $x(t)$ and $y(t)$ are the solutions of the initial-value problems, 
\begin{gather}
 x(0)=0, \qquad \dot{x}(t)=  \sum_{1\leq j \leq n} x_j X_j^{(a)}\big( x(t)\big),
 \label{eq:tangent-approx.IVP-x} \\
 y(0)=0, \qquad  \dot{y}(t)=  \sum_{1\leq k \leq n'} x_k' X_k^{'(a')}\big( y(t)\big) = \sum_{w_j=w_k'} c_{jk}(a) x_j X_k^{'(a')}\big( y(t)\big). 
  \label{eq:tangent-approx.IVP-y}
\end{gather}
 Note also that $x(1)=\exp_{X^{(a)}}(x)$ and $y(1)=\exp_{X^{'(a')}}(x')$. 
 
We have $\tilde{\phi}(x(0))=\tilde{\phi}(0)=0$. Moreover, by using~(\ref{eq:tangent-approx.tphiXj(a)-Xk(a')tphi}) and~(\ref{eq:tangent-approx.IVP-x}) we see that $\frac{d}{dt}\tilde{\phi}(x(t))$ is equal to
 \begin{equation*}
\tilde{\phi}'(x)\big[ \dot{x}(t)\big] = \sum_{1\leq j \leq n} x_j  \tilde{\phi}'(x)\big[ X_j^{(a)}\big(x(t)\big)\big] =  
\sum_{w_j=w_k'} x_j  c_{jk}(a) X_k^{'(a')}\big[ \tilde{\phi}\big ( x(t)\big) \big]. 
\end{equation*}
 This shows that $\tilde{\phi}(x(t))$ is a solution of the initial-value problem~(\ref{eq:tangent-approx.IVP-y}), and so it must agree with $y(t)$ for all $t\geq 0$. As $\tilde{\phi}(x(1))=\tilde{\phi}(\exp_{X^{(a)}}(x))$ and $y(1)=\exp_{X^{'(a')}}(x')= \exp_{X^{'(a')}}(\hat{\phi}'(a)x)$ we deduce that 
 \begin{equation*}
\tilde{\phi}\big(\exp_{X^{(a)}}(x)\big)= \exp_{X^{'(a')}}\big(\hat{\phi}'(a)x\big) \qquad \text{for all $x\in \R^n$}. 
\end{equation*}
This proves~(\ref{eq:Carnot-prop.tphi-hatphi}). 
 
 Finally, recall that $\hat{\phi}'(a)$ is a Lie group map from $G(a)$ to $G'(a')$. As mentioned above, $ \exp_{X^{(a)}}$  is a Lie group map from $G(a)$ to $G^{(a)}$ and  $ \exp_{X^{'(a')}}$  is a Lie group map from $G'(a')$ to $G^{'(a')}$. Thus, the formula~(\ref{eq:Carnot-prop.tphi-hatphi}) immediately implies that $\tilde{\phi}$ is a Lie group map from $G^{(a)}$ to $G^{'(a')}$. The proof is complete. 
 \end{proof}
 
Combining Lemma~\ref{lem:Carnot-prop.phi-hatphi-Ow} and Lemma~\ref{lem:Carnot-prop.hatphi=phi'H} provides us with the following approximation result.
 
 \begin{proposition}\label{prop:Carnot-prop.tangent-map-approx-privileged}
     Let $(x_{1},\ldots,x_{n})$ be privileged coordinates at $a$ adapted to the $H$-frame 
 $(X_{1},\ldots,X_{n})$  and $(x_{1},\ldots,x_{n'})$ privileged coordinates at $a'$ adapted to the $H'$-frame 
 $(X_{1},\ldots,X_{n'})$. Then, in these coordinates, we have
 \begin{equation*}
      \phi(x)= \tilde{\phi}(x)+\Ow(\|x\|^{w+1}) \qquad \text{near $x=0$}, 
 \end{equation*}
 where $\tilde{\phi}(x)$ is the Lie group map from $G^{(a)}$ to $G^{'(a')}$ given by~(\ref{eq:Carnot-prop.tphi-hatphi}).  
 \end{proposition}

When $(x_1,\ldots,x_n)$ and $(x_1, \ldots, x_{n'}')$ are Carnot coordinates the exponential maps $\exp_{X^{(a)}}$ and $\exp_{X^{'(a')}}$ are the identity maps on $\R^n$ and $\R^{n'}$, respectively. Thus, in this case the Lie group map $\tilde{\phi}$ agrees with $\hat{\phi}'(a)$. Therefore, we arrive at the following result. 

\begin{theorem}\label{thm:Carnot-prop.tangent-map-approx}
     Let $(x_{1},\ldots,x_{n})$ be Carnot coordinates at $a$ adapted to the $H$-frame 
 $(X_{1},\ldots,X_{n})$  and $(x_{1},\ldots,x_{n'})$ Carnot coordinates at $a'$ adapted to the $H'$-frame 
 $(X_{1},\ldots,X_{n'})$. Then, in these coordinates, we have
 \begin{equation*}
      \phi(x)= \hat{\phi}'(a)x+\Ow(\|x\|^{w+1}) \qquad \text{near $x=0$}, 
 \end{equation*}
 where $\hat{\phi}'(a)$ is regarded as a $w$-homogeneous group map from $G(a)$ to $G'(a')$. 
 \end{theorem}
 
\begin{remark}\label{rmk:Tangent-map.need-Carnot-coord}
   It is only by working in Carnot coordinates that we obtain an approximation by the Carnot tangent map $\hat{\phi}'(a)$ (compare~\cite[Prop.~5.20]{Be:Tangent}). 
\end{remark}

  \begin{remark}
     In the case of Heisenberg manifolds a version of Theorem~\ref{thm:Carnot-prop.tangent-map-approx} was obtained for Heisenberg diffeomorphisms 
     in~\cite{Po:PJM06}. 
 \end{remark}

Let us mention a couple of corollaries of Theorem~\ref{thm:Carnot-prop.tangent-map-approx}. First, by specializing Theorem~\ref{thm:Carnot-prop.tangent-map-approx} to $\varepsilon$-Carnot coordinates leads us to the following corollary.

\begin{corollary}\label{cor:Carnot-map.approx-vareps-Carnot-coord}
 Let $\kappa$ be a local chart near $a$ and $\tilde{\kappa}$ a local chart near $a'=\phi(a)$. Denote by $\varepsilon^{\kappa}$ the $\varepsilon$-Carnot coordinate map in the local coordinates defined by $\kappa$ (resp., $\tilde{\kappa}$) which is associated with the $H$-frame $(X_1,\ldots, X_n)$ (resp., the $H'$-frame $(X_1',\ldots, X_n')$).   Then, near $x=0$, we have 
\begin{equation*}
\left( \varepsilon_{\tilde{\kappa}(a')}^{\tilde{\kappa}}\circ \tilde{\kappa}\right) \circ \phi \circ \left( \varepsilon^\kappa_{\kappa(a)} \circ \kappa \right)^{-1}(x)= \hat{\phi}'(a)x +\Ow\left( \|x\|^{w+1}\right).
\end{equation*}
\end{corollary}
\begin{proof}
 The map $( \varepsilon_{\tilde{\kappa}(a')}^{\tilde{\kappa}}\circ \tilde{\kappa}) \circ \phi \circ( \varepsilon^\kappa_{\kappa(a)} \circ \kappa)^{-1}$ is the map $\phi$ in the Carnot coordinates defined by the local charts $\varepsilon^\kappa_{\kappa(a)} \circ \kappa$ and $\varepsilon_{\tilde{\kappa}(a')}^{\tilde{\kappa}}\circ \tilde{\kappa}$. The result then follows by applying Theorem~\ref{thm:Carnot-prop.tangent-map-approx}. 
\end{proof}
 
In addition, combining Theorem~\ref{thm:Carnot-prop.tangent-map-approx} with Lemma~\ref{lem-eq-we} immediately gives the following result. 

\begin{corollary}\label{cor:Carnot-prop.tangent-map-approx}
     Let $(x_{1},\ldots,x_{n})$ be Carnot coordinates at $a$ adapted to the $H$-frame 
 $(X_{1},\ldots,X_{n})$  and $(x_{1},\ldots,x_{n'})$ Carnot coordinates at $a'$ adapted to the $H'$-frame 
 $(X_{1},\ldots,X_{n'})$. Denote by $U$ the range of the coordinates $(x_{1},\ldots,x_{n})$. Then, in these coordinates and as $t\rightarrow 0$, we have
 \begin{equation}
     t^{-1}\cdot \phi(t\cdot x)= \hat{\phi}'(a)x+\op{O}(t) \qquad \text{in $C^\infty(U,\R^{n'})$}.
     \label{eq:Tangent-approx.phitx-hphi'(a)x}  
 \end{equation}
In particular, for all $x\in \R^n$, we have 
\begin{equation}
\lim_{t\rightarrow 0} t^{-1}\cdot \phi(t\cdot x)= \hat{\phi}'(a)x.
\label{eq:Tangent-approx.phitx-hphi'(a)x-lim} 
\end{equation}
\end{corollary}

\begin{remark}
 We will see in Section~\ref{sec:Pansu-derivative} that the asymptotics~(\ref{eq:Tangent-approx.phitx-hphi'(a)x}) leads us to the identification of the Carnot differential with Pansu derivative in the case of maps between nilpotent graded groups. More generally, in the setup of ECC manifolds we can identify the Carnot differential with the differential of 
Margulis-Mostow~\cite{MM:GAFA95}. This uses the identification of the tangent group $GM(a)$ as described in Section~\ref{sec:Carnot-Manifolds} with the tangent group of  Margulis-Mostow~\cite{MM:GAFA95, MM:JAM00}, which is defined in terms of equivalence classes of paths. 
\end{remark}

\begin{remark}
 The equality~(\ref{eq:Tangent-approx.phitx-hphi'(a)x-lim}) provides us with an alternative definition of the Carnot differential. It actually allows us to extend the definition of Carnot differential to a larger class of maps between Carnot manifolds by requiring the existence of the limit in the left-hand side of~(\ref{eq:Tangent-approx.phitx-hphi'(a)x-lim}). Using this definition it would be interesting to have a version of  the theorem of Rademacher-Stepanov~\cite{St:MS25} on the differentiability of locally Lipschitz continuous  functions. In the setup of Carnot groups and ECC manifolds results of Pansu~\cite{Pa:AM89} and Margulis-Mostow~\cite{MM:GAFA95} formulate this in terms of locally Lipschitz continuity with respect to Carnot-Carath\'eodory metrics. For general Carnot manifolds we don't have such metrics. Nevertheless,  for  map $\phi:M\rightarrow M'$ between general Carnot manifolds it seems natural to replace local Lipschitz continuity by requiring that, in any Carnot coordinates at $a$ and at $a'=\phi(a)$, we have 
\begin{equation*}
 \phi(x)= \Ow\left(\|x\|^w\right)\qquad \text{near $x=0$}, 
\end{equation*}
where the bounds of the asymptotics are locally uniform with respect to $a$. When $r=1$ this is equivalent to local Lipschitz continuity. 
\end{remark}

Finally, we mention the following result which describes the action of Carnot diffeomorphisms on Carnot coordinates 

 \begin{proposition}\label{prop:Carnot-approx.action-Carnot-coordinates}
 Suppose that $\phi$ is a Carnot diffeomorphism and $\kappa$ is a local chart near $a$ that gives rise to Carnot coordinates at $a$ adapted to the $H$-frame $(X_{1},\ldots,X_{n})$. Then the local chart $\hat{\phi}'(a)\circ \kappa\circ \phi^{-1}$ provides us with  Carnot coordinates at $\phi(a)$ that are adapted to the $H'$-frame 
    $(X_{1}',\ldots,X_{n}')$.  
 \end{proposition}
 \begin{proof}
     Set $a'=\phi(a)$ and let $\tilde{\kappa}$ be a local chart near $a'$. Let $\varepsilon^{\tilde{\kappa}}$ be the $\varepsilon$-Carnot coordinate map in the local coordinates defined by $\tilde{\kappa}$ which is associated with the $H'$-frame  $(X_{1}',\ldots,X_{n}')$, and set $\psi=\kappa \circ \phi^{-1} \circ (\varepsilon^{\tilde{\kappa}}_{\tilde{\kappa}(a')} \circ \tilde{\kappa})^{-1}$. This is the Carnot diffeomorphism $\phi^{-1}$ in the local coordinates provided by the local charts $\kappa$ and $\varepsilon^{\tilde{\kappa}}_{\tilde{\kappa}(a')} \circ \tilde{\kappa}$. By  assumption the chart $\kappa$ gives rise to Carnot coordinates at $a$ adapted to $(X_{1},\ldots,X_{n})$. We also know by Proposition~\ref{prop:Carnot-coord.vareps-carnot-cor} that  $\varepsilon^{\tilde{\kappa}}_{\tilde{\kappa}(a')} \circ \tilde{\kappa}$ gives rise to Carnot coordinates at $a'$ that are adapted to $(X_{1}',\ldots,X_{n}')$.  Therefore, by Theorem~\ref{thm:Carnot-prop.tangent-map-approx} and Proposition~\ref{prop:Carnot.inverse-tangent-map}, near $x=0$, we have 
     \begin{equation*}
         \psi(x)= \widehat{\left(\phi^{-1}\right)}'(a') + \Ow\left(\|x\|^{w+1}\right)= \hat{\phi}(a)^{-1}x + \Ow\left(\|x\|^{w+1}\right). 
     \end{equation*}
     By Lemma~\ref{lem-eq-we} this implies that, for all $x \in \R^{n}$, we have 
     $t^{-1}\cdot \psi(t\cdot x)= \hat{\phi}'(a)^{-1}x+\op{O}(t)$ as $t\rightarrow 0$. Substituting $\hat{\phi}'(a)x$ and using  the $w$-homogeneity of $ \hat{\phi}'(a)$ shows that, 
     for all $x \in \R^{n}$ and as $t\rightarrow 0$, we have 
     \begin{equation*}
        t^{-1}\cdot \hat{\phi}'(a) \circ \psi (t\cdot x) = \hat{\phi}'(a) \left[  t^{-1}\cdot \psi(t\cdot x)\right] = x +\op{O}(t).   
     \end{equation*}
      Using Lemma~\ref{lem-eq-we} once again, we deduce that $\hat{\phi}'(a) \circ \psi(x)=x +\Ow(\|x\|^{w+1})$ near $x=0$. That is, the map $\Theta(x):= \hat{\phi}'(a) \circ \psi(x)$ is a $\Ow(\|x\|^{w+1})$-perturbation of the identity map. To conclude we observe that
$ \hat{\phi}'(a)\circ \kappa\circ \phi^{-1}=   \hat{\phi}'(a)\circ \psi \circ \varepsilon^{\tilde{\kappa}}_{\tilde{\kappa}(a')} \circ \tilde{\kappa}= \Theta \circ \varepsilon^{\tilde{\kappa}}_{\tilde{\kappa}(a')} \circ \tilde{\kappa}$. Thus, the local chart $ \hat{\phi}'(a)\circ \kappa\circ \phi^{-1}$ is of the same type as the charts considered in 
Remark~\ref{rmk:Carnot-coord.characterization-Carnot-charts}, and so it provides us with
 Carnot coordinates at $a'$ adapted to $(X_{1}',\ldots,X_{n}')$. The proof is complete. 
 \end{proof}
 
 \section{Carnot Differential and Pansu Derivative}\label{sec:Pansu-derivative} 
In this section, we establish the precise relationship between the Carnot differential and the Pansu derivative in the case of maps between nilpotent graded groups. This ultimately shows that the Carnot differential is the natural generalization of the Pansu derivative to maps between general Carnot manifolds. 

Throughout this section we let $G$ and $G'$ be graded nilpotent Lie groups of step~$r$. We denote by $e$ their respective units. 
The Lie algebras $\fg=TG(e)$ and $\fg'=TG'(e)$ have gradings $\fg =\fg_1 \oplus \cdots \oplus \fg_r$ and  $\fg' =\fg_1' \oplus \cdots \oplus \fg_r'$ that are compatible with their respective Lie brackets. 
These gradings give rise to dilations $\delta_t$, $t\in \R\setminus 0$, that are group automorphisms of $G$ and $G'$, respectively. 
They also give rise to vector bundle gradings $TG=E_1\oplus \cdots \oplus E_r$ and $TG'=E_1'\oplus \cdots \oplus E_r'$, where $E_w$ (resp., $E_{w}'$) is obtained by left-translating $\fg_w$ (resp.,  $\fg_w'$) over $G$ (resp., $G'$). 
These vector bundle gradings yield left-invariant Carnot filtrations $H_1\subset \cdots \subset H_r=TG$ and $H_1'\subset \cdots \subset H_r'=TG'$. 
They also provide us with natural vector bundle identifications $TG\simeq \fg G$ and $TG'\simeq \fg G'$, where $\fg G$ and $\fg G'$ are the respective tangent Lie algebra bundles of $G$ and $G'$. This identifies the $w$-th components $\fg_w G$ and $\fg_{w} G'$ with the bundles $E_w$ and $E_w'$, respectively.

In addition, as mentioned in~\cite{CP:Carnot} and in Example~\ref{ex:Tangent-group.group}, the left-regular action of $G$ on itself gives rise to a natural global trivialization $\lambda: \fg G \stackrel{\sim}{\rightarrow} G\times \fg$.  If for any given $a\in G$ we let $\lambda_a:G\rightarrow G$ be the left-multiplication by $a$, then we have 
\begin{equation*}
 \lambda(a,\xi)= \left( a, \lambda_a'(e)^{-1} \xi\right) \qquad \text{for all $(a,\xi)\in \fg G$}. 
\end{equation*}
Here $ \lambda_a'(e)^{-1} $ is regarded as a linear map from $\fg G(a) \simeq TG(a)$ onto $\fg$. Let us denote by $GG$ the tangent Lie group bundle of $G$. Composing the trivialization $\lambda$ with the exponential map $\exp:\fg \rightarrow G$ provides us with a global trivialization $\Lambda: GG \stackrel{\sim}{\rightarrow} G\times G$. More precisely, if for any $a\in G$, we set $\Lambda_a=\exp \circ \lambda_a'(e)^{-1}: \fg G(a)=GG(a)\rightarrow G$, then we have 
\begin{equation*}
\Lambda(a,\xi)= \left( a, \Lambda_a \xi\right)   \qquad \text{for all $(a,\xi)\in G G$}.
\end{equation*}
 Likewise, we have global trivializations $\lambda: \fg G' \stackrel{\sim}{\rightarrow} G\times \fg'$ and  $\Lambda: GG' \stackrel{\sim}{\rightarrow} G'\times G'$, where $GG'$ is the tangent group bundle of $G'$. 
  
\begin{definition}[Pansu~\cite{Pa:AM89}] Given a map $\phi:G\rightarrow G'$ and a point $a\in G$, its \emph{Pansu derivative} at $a$ is the map $D\phi(a):G\rightarrow G'$ defined by 
\begin{equation*}
 D\phi(a)x = \lim_{t\rightarrow 0} \delta_t^{-1} \left[ \phi(a)^{-1} \cdot \phi\left( a\cdot \delta_t(x)\right) \right], \qquad x\in G, 
\end{equation*}
 provided the limit exists. 
\end{definition}
When it exists the Pansu derivative is compatible with the dilations of $G$ and $G'$. When $G$ and $G'$ are Carnot groups we can equip them with Carnot-Carath\'eodory metrics associated with left-invariant sub-Riemannian metrics on $H_1$ and $H_1'$. A result of Pansu~\cite{Pa:AM89} asserts that the Pansu derivative $D\phi(a)$ exists and is a group map for almost every point  $a\in G$ when $\phi$ is locally Lipschitz continuous with respect to these metrics.  

Suppose that $\phi:G\rightarrow G'$ is a (smooth) Carnot manifold map. For every $a\in G$, the Carnot differential $\hat{\phi}'(a)$ is a group map from $GG(a)$ to $GG'(a')$, with $a'=\phi(a)$. Under the identifications $GG(a)\simeq G$ and $GG'(a')\simeq G'$ described above this is the group map, 
\begin{equation}
 \Lambda_{a'} \circ \hat{\phi}'(a) \circ \Lambda_a^{-1}: G \longrightarrow G'. 
 \label{eq:Carnot-diff.identification-hphi'-groups}
\end{equation}
 
\begin{theorem}\label{thm:Pansu.Carnot-Pansu}
 Let $\phi: G \rightarrow G'$ be a (smooth) Carnot manifold map. Then its Pansu derivative exists everywhere, and  we have  
\begin{equation}
 D\phi(a) = \Lambda_{\phi(a)} \circ \hat{\phi}'(a) \circ \Lambda_a^{-1} \qquad \forall a \in G.
 \label{eq:Carnot-diff.Carnot-diff-Pansu-derivative} 
\end{equation}
Thus, under the identification~(\ref{eq:Carnot-diff.identification-hphi'-groups}), the Pansu derivative agrees with the Carnot differential $\hat{\phi}'$ at every point. 
\end{theorem}
\begin{proof}
The result is merely a special case of Corollary~\ref{cor:Carnot-prop.tangent-map-approx}. The main task is to unwind the various identifications involved. 

Let $(\xi_1, \ldots, \xi_n)$ be a graded basis of $\fg$. For $j=1, \ldots, n$, let $X_j$ be the left-invariant vector field on $G$ such that $X_j(e)=\xi_j$. Then $(X_1,\ldots, X_n)$ is a (global) $H$-frame of $TG$. Let $\exp_X:\R^n \rightarrow G$ be the $w$-homogeneous diffeomorphism~(\ref{eq:Carnot-coord.canonical-coord}).  In addition, as above the 
 $H$-frame $(X_1, \ldots, X_n)$ gives rise to a graded basis $(\xi_1(a),\ldots, \xi_n(a))$ of $\fg G(a)$, where $\xi_j(a)$ is the class of $X_j(a)$ in $\fg_{w_j} G(a)$. In fact, under the identification $\fg G(a)\simeq TG(a)$ this is just $X_j(a)$. As $X_j$ is a left-invariant vector field, we have $\xi_j(a)= \lambda_a'(e)[X_j(e)]=\lambda_a'(e)\xi_j$. Let 
 $\chi_{a}$ be the isomorphism from $\R^n$ onto $\fg G(a)=GG(a)$ defined by the basis $(\xi_1(a),\ldots, \xi_n(a))$. Then, for all $x\in \R^n$, we have 
\begin{align*}
 \Lambda_a \circ\chi_{a} (x)& = \Lambda_a \left( x_1 \xi_1(a) + \cdots +  x_n \xi_n(a)\right) \\  
 & = \exp \circ \lambda_a'(e)^{-1}\left( x_1 \lambda_a'(e)\xi_1 + \cdots +  x_n\lambda_a'(e)\xi_n\right) \\ \
 & = \exp(x_1\xi_1 + \cdots + x_n \xi_n).  
\end{align*}
As $\exp(x_1\xi_1 + \cdots + x_n \xi_n)=\exp(x_1X_1 + \cdots + x_n X_n)=\exp_X(x)$, we see that
\begin{equation}
 \exp_X(x)=  \Lambda_a \circ\chi_{a} (x)  \qquad \text{for all $x\in \R^n$}. 
  \label{eq:Carnot-diff.expX}
\end{equation}

Similarly, let $(\xi_1',\ldots, \xi_{n'}')$ be a graded basis of $\fg'$ and $(X_1', \ldots, X_{n'}')$ the associated left-invariant $H'$-frame of $TG'$. Let $\exp_{X'}:\R^{n'}\rightarrow G'$ be the $w$-homogeneous diffeomorphism~(\ref{eq:Carnot-coord.canonical-coord}) associated with $(X_1', \ldots, X_{n'}')$. 
The $H'$-frame $(X_1', \ldots, X_{n'}')$ also gives rise to a graded basis $(\xi_1'(a'), \ldots, \xi_{n'}'(a'))$ of $\fg G'(a')$ and a corresponding linear isomorphism $\chi_{a'}:\R^{n'}\rightarrow G'$. In the same way as in~(\ref{eq:Carnot-diff.expX}) we have
\begin{equation}
  \Lambda_{a'} \circ \chi_{a'} (x') = \exp_{X'}(x') \qquad \text{for all $x'\in \R^{n'}$}.
  \label{eq:Carnot-diff.expX'} 
\end{equation}
In addition, as above the graded bases $(\xi_1(a),\ldots, \xi_n(a))$ and $(\xi_1'(a'), \ldots, \xi_{n'}'(a'))$ enables us to identify the Carnot differential $\hat{\phi}'(a)$ with the map, 
\begin{equation}
 \bar{\phi}'(a):= \chi_{a'}^{-1}\circ \hat{\phi}'(a) \circ \chi_{a} : \R^n \longrightarrow \R^{n'}.
 \label{eq:Carnot-diff.barphi'a} 
\end{equation}

Bearing all this in mind,  we know by Proposition~\ref{prop:Carnot-coord.graded-nilpotent-group} and Remark~\ref{rmk:Carnot-coord.graded-nilpotent-group} 
that the global chart $ (\lambda_a \circ \exp_X)^{-1}: G \rightarrow \R^n$ provides us with Carnot coordinates at $a$ adapted to $(X_1,\ldots, X_n)$. Likewise, the chart $( \lambda_{a'} \circ \exp_{X'})^{-1}: G' \rightarrow \R^{n'}$ gives rise to Carnot coordinates at $a'$ adapted to $(X_1',\ldots, X_n')$. Therefore, by Corollary~\ref{cor:Carnot-prop.tangent-map-approx}, for all $x\in \R^n$ and as $t\rightarrow 0$, we have 
\begin{equation}
 t^{-1} \cdot \left[ \left(\lambda_{a'} \circ \exp_{X'}\right)^{-1} \circ \phi \circ \left(\lambda_a \circ \exp_X \right)(t\cdot x)\right]=\bar{\phi}'(a)x +\op{O}(t). 
\end{equation}
Given any $y\in G$,  substituting $\exp_X^{-1}(y)$ for $x$ and using the $w$-homogeneity of $\exp_{X}$ and $\exp_{X'}$ we deduce that, for all $x\in G$ and as $t\rightarrow 0$, we have 
\begin{align*}
 \delta_t^{-1} \left[ \phi(a)^{-1} \cdot \phi\left( a\cdot \delta_t(y)\right) \right] & = \left(\exp_{X'} \circ \delta_t^{-1} \circ \exp_{X'}^{-1}\right) \circ \left(\lambda_{a'}^{-1}\circ \phi \circ 
 \lambda_a \right) \circ \left(\exp_X\circ \delta_t\circ \exp_X^{-1}\right)\!(y)\\ 
 & = \exp_{X'}\left\{  t^{-1} \cdot \left[ \left(\lambda_{a'} \circ \exp_{X'}\right)^{-1} \circ \phi \circ \left(\lambda_a \circ \exp_X\right)\left(t\cdot \exp_X^{-1}(y)\right)\right]\right\}\\
 & = \exp_{X'}\left\{  \bar{\phi}'(a) \exp_X^{-1}(y)  +\op{O}(t) \right\} \\
 &= \exp_{X'}\circ \bar{\phi}'(a) \circ  \exp_X^{-1}(y) +\op{O}(t). 
\end{align*}
This shows that $\phi$ admits a Pansu derivative at $a$. Moreover, by using~(\ref{eq:Carnot-diff.expX}), (\ref{eq:Carnot-diff.expX'}) and~(\ref{eq:Carnot-diff.barphi'a})  we see that the Pansu derivative is equal to
\begin{align*}
 \exp_{X'}\circ \bar{\phi}'(a) \circ  \exp_X^{-1} & = \left(  \Lambda_{a'} \circ \chi_{a'}\right) \circ \left(  \chi_{a'}^{-1}\circ \hat{\phi}'(a) \circ \chi_{a}\right) \circ 
  \left(  \Lambda_a \circ\chi_{a}\right)^{-1}\\
  &  =  \Lambda_{a'} \circ \hat{\phi}'(a) \circ \Lambda_a^{-1}. 
\end{align*}
This proves the result.
\end{proof}
  
 \section{Lie Groupoids and Connes' Tangent Groupoid}\label{sec:Lie-groupoids}
 In this section, we recall the main definitions and examples of Lie groupoids and review the construction of the tangent groupoid of a manifold by Connes~\cite{Co:NCG}. 
 
 The notion of groupoid interpolates between groups and spaces. Therefore, they provide us with a natural framework for the encoding of deformations of spaces to groups. 
 
\begin{definition} A groupoid structure on a set $\cG$ is given by
    \begin{enumerate}
        \item [(i)] A set $\cG^{(0)}$ (called base) and a one-to-one map $\epsilon:\cG^{(0)}\rightarrow \cG$ (called unit map). 
    
        \item[(ii)] Maps $s: \cG\rightarrow \cG^{(0)}$ and $r: \cG\rightarrow \cG^{(0)}$ (called range and 
        source maps, respectively) such that
        \begin{equation*}
            s\left[\epsilon(\gamma)\right]=r\left[\epsilon(\gamma)\right]=\gamma  \quad \text{for all  
            $\gamma \in \cG^{(0)}$}.
        \end{equation*}
    
        \item[(iii)] A multiplication $ \mu: \cG^{(2)}\ni (\gamma_{1},\gamma_{2}) \rightarrow  \gamma_{1}\cdot \gamma_{2}\in \cG$, where 
        $\cG^{(2)}=\{(\gamma_{1},\gamma_{2})\in \cG^{2};\ s(\gamma_{1})=r(\gamma_{2})\}$ and the following properties are satisfied:
        \begin{gather*}
            \gamma\cdot \epsilon\left[s (\gamma)\right]=  \epsilon\left[r(\gamma)\right]\cdot 
        \gamma=   \gamma   \quad \text{for all  $\gamma \in \cG$},\\
    s (\gamma_1 \cdot \gamma_2) =s (\gamma_2) \quad \text{and} \quad r(\gamma_1 \cdot \gamma_2) = r (\gamma_1) 
        \quad  \text{for all $(\gamma_1, \gamma_2 ) \in \mathcal{G}^{(2)}$}, \\    
         (\gamma_1 \cdot \gamma_2) \cdot \gamma_3 = \gamma_1 \cdot (\gamma_2 \cdot \gamma_3) \quad \text{for all 
        $(\gamma_{1},\gamma_{2},\gamma_{3})\in \cG^{(3)}$}.   
        \end{gather*} Here we have set $\cG^{(3)}=\left\{ (\gamma_{1},\gamma_{2},\gamma_{3})\in \cG^{3};\ 
        s(\gamma_{1})=r(\gamma_{2}) \ \text{and} \ s(\gamma_{2})=r(\gamma_{3})\right\}$.
        
        \item[(iv)] An inverse map $\iota:\cG\ni \gamma\rightarrow \gamma^{-1}\in \cG$ such that
        \begin{equation}\label{eq-gp-1}
            \gamma \cdot \gamma^{-1}= \epsilon\left[r (\gamma)\right] \quad  \text{and} \quad 
           \gamma^{-1}\cdot \gamma = \epsilon\left[s (\gamma) \right]
            \quad \text{for all  $\gamma \in \cG$}.
        \end{equation}
    \end{enumerate}
\end{definition}

\begin{remark}
    The equations~(\ref{eq-gp-1}) uniquely determine the inverse $\gamma^{-1}$. This implies that $(\gamma^{-1})^{-1}=\gamma$ 
    for all $\gamma \in \cG$, i.e., the inverse map is an involution. 
\end{remark}

\begin{remark}\label{rmk:Groupoid.morphism}
    Given groupoids $\cG$ and $\cG'$, a groupoid map from $\cG$ to $\cG'$ is given by maps $\Phi:\cG\rightarrow 
    \cG'$ and $\Phi^{0}:\cG^{(0)}\rightarrow (\cG')^{(0)}$ such that
    \begin{gather*}
        \Phi\circ  \epsilon=  \epsilon\circ \Phi^{0}, \qquad \Phi^{0}\circ s=s\circ \Phi, \qquad \Phi^{0}\circ r=r\circ \Phi,\\
        (\Phi\otimes \Phi)\circ \mu=\mu \circ \Phi, \qquad \Phi\circ \iota =\iota \circ \Phi.
    \end{gather*}
\end{remark}

\begin{example}[Pair Groupoid]\label{ex:Groupoid.pair-groupoid}
  Let $X$ be any set. Then $X\times X$ is a groupoid with $\cG^{0}=X$, and the unit, source and range maps defined by
  \begin{equation*}
     \epsilon(x)=(x,x), \qquad s(x,y)=y, \qquad  r(x,y)=x.
  \end{equation*}The multiplication and inverse maps are given by
  \begin{equation*}
      (x,y)\cdot (y,z)=(x,z) \qquad \text{and} \qquad (x,y)^{-1}=(y,x).
  \end{equation*}
\end{example}

\begin{example}[Group Bundle]\label{ex:Groupoid.group-bundle}
Let $G\stackrel{\pi}{\longrightarrow}M$ be a group bundle over a space $X$, so that each fiber $G(x)=\pi^{-1}(x)$, $x\in X$, is a group. Then $G$ defines a groupoid with  $\cG=G$ and 
$\cG^{(0)}=M$. The unit map is given by 
\begin{equation*}
    \epsilon(x)=e_{x} \qquad \text{for all $x\in X$},
\end{equation*}
where $e_{x}$ is the unit element of $G(x)$. The source and range maps 
are equal to $\pi$. Thus,
\begin{equation*}
    \cG^{(2)}=\{(g_{1},g_{2})\in G^{2}; \ \pi(g_{1})=\pi(g_{2})\}.
\end{equation*}That is, a pair $(g_{1},g_{2})\in G^{2}$ is in $\cG^{(2)}$ if and only if $g_{1}$ and $g_{2}$ lie in the 
same fiber. The multiplication and inverse maps of $\cG$ are given by the fiberwise multiplication and inverse maps of 
$G$. 
\end{example}

\begin{example}[Action Groupoid]\label{ex:Groupoid.group-action}
 Let $X$ be a set equipped with the left-action $G\times X\ni (g,x)\rightarrow gx\in X$ of some group $G$. The groupoid associated with this action is defined as follows. We 
 take $\cG=G\times X$ and $\cG^{(0)}=X$. The unit, source and range maps are defined by
 \begin{equation*}
     \epsilon(x)=(e,x), \qquad \sigma(g,x)=x, \qquad r(g,x)=gx.
 \end{equation*}The multiplication and inverse maps are given by
\begin{equation*}
     (h,gx)\cdot (g,x)=(hg,x) \qquad \text{and}\qquad (g,x)^{-1}=(g^{-1},gx).
\end{equation*}
\end{example}

In what follows, we will be interested in groupoids carrying differentiable structures in the following sense. 

\begin{definition} We shall say that a groupoid $\cG$ is a \emph{Lie groupoid} when
    \begin{enumerate}
        \item[(i)] $\cG$ and $\cG^{(0)}$ are smooth manifolds.  
    
        \item[(ii)] The unit map $\epsilon:\cG^{(0)}\rightarrow \cG$ is a smooth embedding.  
    
        \item[(iii)] The source and range maps $s:\cG\rightarrow \cG^{(0)}$ and   $r:\cG\rightarrow \cG^{(0)}$  are 
        smooth submersions (so that $\cG^{(2)}$ is a manifold). 
    
        \item[(iv)] The multiplication map $\mu:\cG^{(2)}\rightarrow \cG$  and the inverse map $\iota:\cG\rightarrow 
        \cG$ are smooth maps.  
    \end{enumerate}
\end{definition}

\begin{remark} 
    The following are examples of Lie groupoids:
    \begin{itemize}
        \item  The pair groupoid of any smooth manifold (\emph{cf.}~Example~\ref{ex:Groupoid.pair-groupoid}).
    
        \item  The groupoid associated with a smooth bundle of Lie groups 
        (\emph{cf.}~Example~\ref{ex:Groupoid.group-bundle}).
    
        \item  The action groupoid of the (smooth) action of a Lie group on a manifold (\emph{cf.}~Example~\ref{ex:Groupoid.group-action}).
    \end{itemize}
\end{remark}

\subsection{Connes' tangent groupoid}
As mentioned above, groupoids interpolate between spaces and groups. This aspect especially pertains in the construction of the tangent groupoid of a manifold of Connes~\cite{Co:NCG}. More precisely, given any smooth manifold $M$, the tangent groupoid $\cG=\cG M$ is a Lie groupoid that encodes the smooth deformation of the space $M\times M$ to the tangent bundle $TM$. This groupoid is obtained as follows.

Set $\R^*=\R\setminus \{0\}$.  At the set-theoretic level $\cG$ and the base $\cG^{(0)}$ are defined by
\begin{equation}
\cG = TM \sqcup (M \times M \times \R^*) \quad \text{and} \quad \cG^{(0)} = M \times \R. 
\label{eq:Groupoid.Connes-groupoid}
\end{equation}
The unit map $\epsilon:\cG^{(0)}\rightarrow \cG$ is  given by
\begin{equation*}
\epsilon(x,t) = \left\{ 
\begin{array}{ll}
(x,x,t) & \text{for $t\neq 0$ and $x \in M$},\smallskip \\
(x,0) \in TM(x) & \text{for $t=0$ and $x \in M$}.
\end{array} \right.
\end{equation*}
The range and source maps are defined by
\begin{equation*}
    \begin{array}{cl}
        r (x,y,t) = (x, t) \quad \text{and} \quad s (x,y,t) = (y,t) &\qquad \text{for $t\neq 0$ and $x,y \in M$},\smallskip\\
    r (x, v) = s (x, v) = (x, 0) &\qquad \textrm{for $x\in M$ and $v \in TM(x)$}.
    \end{array}
\end{equation*}

We have  $\cG^{(2)}=\cG^{(2)*} \sqcup TM^{(2)}$, where 
\begin{gather}
  \cG^{(2)*}:=\left\{ \left( (x,y,t),(y,z,t)\right); x,y,z\in M, \ t\in \R^*\right\}, 
  \label{eq:Groupoid.cG(2)}\\
 TM^{(2)}:= \left\{\left((x,v),(x,w)\right)\in TM\times TM; x\in M, \ v,w\in TM(x)\right\}.  
\end{gather}
In particular, $TM^{(2)}$ is just the fiber product over $M$ of $TM$ with itself. We then define the multiplication map $\cG^{(2)} \rightarrow \cG$ by 
\begin{align*}
(x,y,t) \cdot (y,z,t) = (x,z,t) & \qquad \text{for $t\neq 0$ and $x, y, z\in M$},\\
(x,v) \cdot(x,w) = (x,v+w) & \qquad \text{for $x\in M$ and $v,w\in TM(x)$}.
\end{align*}
In addition, the inverse map $\cG \rightarrow \cG$ is given by
\begin{align*}
         (x,y,t)^{-1}=(y,x,t) &\qquad \text{for $t\neq 0$ and $x,y\in M$},\\
    (x,v)^{-1} = (x,-v)&  \qquad \text{for $x\in M$ and $v\in TM(x)$}.
\end{align*}

Note that $\cG^{(0)}=M\times \R$ is a manifold and $\cG=\cG M$ is the disjoint union of the manifolds $TM$ 
and $\cG^*:=M\times M\times \R^*$. We glue their topologies and differentiable structures as follows. 

The topology of $\cG$ is the weakest topology such that
    \begin{itemize}
        \item The inclusion of $TM$ into  $\cG$ is a topological embedding. 
    
        \item The inclusion of $\cG^*$ into $\cG$ is an open continuous map.
    
        \item A sequence $\left( (x_{\ell},y_{\ell},t_{\ell})\right)_{\ell\geq 0}\subset \cG^*$ converges to a point $(x,v)\in TM$ if 
        and only if, as $\ell \rightarrow \infty$, we have 
        \begin{gather}
            x_{\ell}\longrightarrow x, \qquad y_{\ell}\longrightarrow x, \qquad t_{\ell}\longrightarrow 0,  \qquad \text{and}\\
             \frac{1}{t_{\ell}}\left( \kappa(y_{\ell})-\kappa(x_{\ell})\right) \longrightarrow \kappa'(x)v \qquad \text{for any local 
             chart $\kappa$ near $x$}.\label{eq-gp-2}
        \end{gather}
    \end{itemize}
It can be checked that the condition~(\ref{eq-gp-2}) does not depend on the choice of the chart $\kappa$ near $x$.     

The differentiable structure of $\cG=\cG M$ is such that
   \begin{itemize}
    \item The inclusion of $\cG^*$ into $\cG$ is a smooth embedding.  

    \item For any point $x_{0}\in M$ and any local chart $\kappa$ near $x_{0}$, a local chart from a neighborhood of 
    $TM(x_0)$ to $\R^{n}\times \R^{n}\times \R$ is given by 
     \begin{equation}
         \begin{array}{cl}
            \gamma_{\kappa}(x,y,t)=\left(\kappa(x),t^{-1}\left(\kappa(y)-\kappa(x)\right),t\right) &\quad \text{for $t\neq 0$ and 
         $x,y\in \op{dom}(\kappa)$},\smallskip \\
          \gamma_{\kappa}(x,v)=\left(\kappa(x),\kappa'(x)v, 0\right)&\quad \text{for $x\in \op{dom}(\kappa)$ and $v\in 
          TM(x)$}.  
         \end{array}
         \label{eq:Lie-Groupoid.chart-gammak}
     \end{equation}
\end{itemize}
If $\kappa$ and $\kappa_{1}$ are local charts near $x_{0}$ and $\phi=\kappa_{1}\circ \kappa^{-1}$ is the 
    corresponding transition map, then, for all $(x,v,t)\in \R^{n}\times \R^{n}\times \R$ such that $x$ and $x+tv$ 
    are in $\op{dom}(\phi)$, we have
    \begin{equation*}
        \gamma_{\kappa_{1}}\circ \gamma_{\kappa}^{-1}(x,v,t)=\left\{ 
        \begin{array}{ll}
            \left( \phi(x),t^{-1}\left(\phi(x+tv)-\phi(x)\right),t\right) & \text{if $t\neq 0$},\smallskip\\
            \left(\phi(x),\phi'(x)v,0\right) & \text{if $t=0$}.
        \end{array}\right.
    \end{equation*}
    Therefore, the maps $\gamma_{\kappa}$ define a system of local charts near $TM$. We then have the following result.

\begin{proposition}[Connes~\cite{Co:NCG}]\label{prop:Groupoid.Connes-groupoid-bdiff} 
    With respect to the differentiable structure above, the tangent groupoid 
    $\cG=\cG M$ is a Lie groupoid. 
\end{proposition}

The original motivation of Connes~\cite{Co:NCG} for the construction of the tangent groupoid was to produce a new proof of the index theorem of Atiyah-Singer~\cite{AS:AM68}. This construction also give some nice insight on the notion of tangent space of a manifold. More precisely, the fact that $\cG M$ is a differentiable manifold implies that, for every $x_0\in M$, the submanifolds $\{x_0\}\times M\times \{t\}$ look more and more like the tangent space $TM(x_0)$ as $t\rightarrow 0$. To understand these manifolds in local coordinates, let $\kappa:V\rightarrow U$ be a local chart whose domain is an open neighborhood of $x_0$. Up to translation the chart $\gamma_\kappa$ in~(\ref{eq:Lie-Groupoid.chart-gammak}) gives rise to a local chart $\tilde{\gamma}_\kappa^t: \{x_0\}\times V\times \{t\}\rightarrow W_t$, where $W_t$ is an open neighborhood of $\kappa(x_0)$ and we have 
\begin{equation*}
 \tilde{\gamma}_\kappa^t(x_0,x,t)=\kappa(x_0)+t^{-1}\left(\kappa(x)-\kappa(x_0)\right)\qquad \forall x\in V. 
\end{equation*}
The transition map $ \tilde{\gamma}_\kappa^1\circ ( \tilde{\gamma}_\kappa^t)^{-1}$ is just the rescaling map $x\rightarrow \kappa(x_0)+t^{-1}(x-\kappa(x_0))$. Thus, in local coordinates, the submanifold $\{x_0\}\times M\times \{t\}$ is merely a zoomed in version of $\{x_0\}\times M\times \{1\}\simeq M$, where the zooming performed by rescaling by $t^{-1}$ along rays out of $x_0$. Thus Connes' construction recasts in the language of groupoids the well-known idea that, as we zoom in more and more around $x_0$, the manifold $M$ looks more and more like the tangent space $TM(x_0)$. In addition, the very fact that we obtain a groupoid, and not just any manifold, further accounts for the additive group structure of $TM(x_0)$.

\section{The Tangent Groupoid of a Carnot Manifold}\label{sec:tangent-groupoid}
In this section, we construct the analogue for Carnot manifolds of Connnes' tangent groupoid. Bella\"iche~\cite{Be:Tangent} conjectured the existence of such a groupoid for and this would explain why the tangent space of a Carnot manifold should be a group.  


 Throughout this section we let $(M,H)$ be a Carnot manifold. 
 
\subsection{Construction of $\cG_H M$ as an abstract groupoid}
Connes tangent groupoid is obtained as the Lie groupoid that encodes the smooth deformation of the pair manifold $M\times M$ to the tangent bundle $TM$. Likewise, we shall associated with the Carnot manifold a Lie groupoid $\cG_H M$ that encodes the smooth deformation of $M\times M$ to the tangent group bundle $GM$. 

At the set-theoretic level we define the groupoid $\cG=\cG_H M$ and its base $\cG^{(0)}$ by
\begin{equation*}
   \cG= GM \sqcup ( M \times M \times \R^* ) \qquad \text{and} \qquad \cG^{(0)}=M \times \R. 
\end{equation*}
The unit map $\epsilon:\cG^{(0)}\rightarrow \cG$ is  given by
\begin{equation*}
\epsilon(x,t) = \left\{ \begin{array}{ll}
(x,x,t) & \text{for $t\neq 0$ and $x \in M$},\smallskip \\
(x,0) \in GM & \text{for $t=0$ and $x \in M$}.
\end{array} \right.
\end{equation*}
The range and source maps $r,s:\mathcal{G}\rightarrow \cG^{(0)}$ are defined by
\begin{equation*}
    \begin{array}{cl}
        r (x,y,t) = (x, t) \quad \text{and} \quad s (x,y,t) = (y,t) &\qquad \text{for $t\neq 0$ and $x,y \in M$},\smallskip\\
    r (x, \xi) = s (x, \xi) = (x, 0) &\qquad \textrm{for $x\in M$ and $\xi \in GM(x)$}.
    \end{array}
\end{equation*}

We have $\cG^{(2)}= \cG^{(2)*}\sqcup GM^{(2)}$, where $\cG^{(2)*}$ is defined as in~(\ref{eq:Groupoid.cG(2)}) and $GM^{(2)}$ is the fiber product over $M$ of $GM$ with itself. That is, 
\begin{equation}
 GM^{(2)}= \left\{\left((x,\xi),(x,\eta)\right)\in GM\times GM; x\in M, \ \xi,\eta\in GM(x)\right\}.
 \label{eq:Groupoid.GM(2)}
\end{equation}
 We define the multiplication map $\cG^{(2)}\rightarrow \cG$ by
\begin{align*}
(x,y,t) \cdot (y,z,t) = (x,z,t) & \qquad \text{for $t\neq 0$ and $x, y, z\in M$},\\
(x,\xi) \cdot(x,\eta) = (x,\xi\cdot \eta) & \qquad \text{for $x\in M$ and $\xi,\eta\in GM(x)$},
\end{align*}where $\cdot$ is the product law of $GM(x)$. 

Finally, the inverse map $\cG\rightarrow \cG$ is given by
\begin{align*}
         (x,y,t)^{-1}=(y,x,t) &\qquad \text{for $t\neq 0$ and $x,y\in M$},\\
    (x,\xi)^{-1} = (x,\xi^{-1})=(x,-\xi)&  \qquad \text{for $x\in M$ and $\xi\in GM(x)$}.
\end{align*}
It is immediate to check that this defines an abstract groupoid.  

\begin{definition}
   The groupoid $\cG_{H}M$ defined above is called the \emph{tangent groupoid} of the Carnot manifold $(M,H)$. 
\end{definition}

\subsection{Topology and differentiable structure on $\cG_H M$} 
As in the case of Connes' tangent groupoid, $\cG^{(0)}=M\times \R$ is a manifold and $\cG=\cG_H M$ 
is the disjoint union of the manifolds $GM$ and $\cG^*=M\times M\times \R^*$. We glue together the topologies and differentiable structures of these manifolds as follows. 

\begin{definition}
 An \emph{$H$-chart} is given by the data of a local chart $\kappa:M\supset U\rightarrow V\subset \R^n$ for $M$ and an $H$-frame over $U$. 
\end{definition}

Given any  local $H$-chart $\kappa: U\rightarrow V$ with $H$-frame $(X_{1},\ldots,X_{n})$, we let $\varepsilon^\kappa: V\times \R^n \rightarrow  \R^n$, 
$(x,y)\rightarrow \varepsilon^{\kappa}_{x}(y)$, be the  $\varepsilon$-Carnot coordinate map in the local coordinates defined by 
$\kappa$ which is associated with the frame $(\kappa_{*}X_{1},\ldots,\kappa_{*}X_{n})$. As we have implicitly done before, we  will often identify the $H$-frame $(X_{1},\ldots,X_{n})$ with its pushforward by $\kappa$. Incidentally, we regard $V$ as a Carnot manifold so that $\kappa$ becomes a Carnot diffeomorphism. 

As mentioned in Remark~\ref{rmk:Tangent-group.frame-fgM}, the $H$-frame $(X_{1},\ldots,X_{n})$ gives rise to a smooth graded frame $(\xi_1(x), \ldots, \xi_n(x))$ of $\fg M$ over $U$, where $\xi_j(x)$ is the class of $X_j(x)$ in $\fg_{w_j}(x)$. This provides us with a local trivialization of $\fg M$ and $GM$ over $U$.  At every point $a\in U$ the tangent group $GM(a)$ is identified with the graded nilpotent $G(a)$, which is obtained by equipping $\R^n$ with the Dynkin product associated with the structure constants $L_{ij}^k(a)$, $w_i+w_j=w_k$, arising from the commutator relations~(\ref{eq:Carnot-mfld.brackets-H-frame}). 
We have a similar identification of $GV(\kappa(a))$ with $G(a)$ by using the $H$-frame $(\kappa_{*}X_{1},\ldots,\kappa_{*}X_{n})$, since the coefficients in the commutator relations~(\ref{eq:Carnot-mfld.brackets-H-frame}) are just $L_{ij}^k(\kappa^{-1}(x))$. Note also that, the graded frame $(\xi_1^\kappa(x), \ldots, \xi_n^\kappa(x))$ of $\fg V$ defined by $(\kappa_{*}X_{1},\ldots,\kappa_{*}X_{n})$ is such that $\xi_j^\kappa(\kappa(x))=\hat\kappa'(x)\xi_j(x)$ for all $x\in V$ and $j=1, \ldots, n$. 

As $V$ is an open subset of $\R^n$ and $(\kappa_{*}X_{1},\ldots,\kappa_{*}X_{n})$ is a global $H$-frame over $V$, at the manifold level it is only natural to identify $\fg V$ and $GV$ with the trivial bundle $V\times \R^n$. We obtain these identifications by using the graded frame $(\xi_1^\kappa(x), \ldots, \xi_n^\kappa(x))$. At the algebraic level this identifies $GV(\kappa(a))$ with $G(a)$. Under this identification the Carnot differential $\hat\kappa'(a)$ becomes a (homogeneous) Lie group isomorphism $\hat\kappa'(a):GM(a)\rightarrow G(a)$. As $\xi_j^\kappa(\kappa(a))=\hat\kappa'(a)\xi_j(a)$ what we obtain is just the identification of $GM(a)$ with $G(a)$ provided by the graded basis $(\xi_1(a), \ldots, \xi_n(a))$. \emph{In what follows this is how we are going to regard the Carnot differential of an $H$-chart.} 

\begin{definition}
    The topology of $\cG=\cG_H M$ is the weakest topology such that
    \begin{itemize}
        \item The inclusion of $GM$ into $\cG$ is a topological embedding.  
    
        \item The inclusion of $\cG^*= M\times M\times \R^*$ into $\cG$ is an open continuous map.
    
        \item A sequence $\left( (x_{\ell },y_{\ell },t_{\ell })\right)_{\ell \geq 1}\subset \cG^*$ converges to a point $(x,\xi)\in GM$ if 
        and only if, as $\ell \rightarrow \infty$, we have 
        \begin{gather}
            x_{\ell }\longrightarrow x, \qquad y_{\ell }\longrightarrow x, \qquad t_{\ell }\longrightarrow 0,  \qquad \text{and}\\
            t_{\ell }^{-1}\cdot \left(\varepsilon_{\kappa(x_{\ell })}^{\kappa}\circ \kappa\right)(y_{\ell }) \longrightarrow 
           \hat{\kappa}'(x)\xi \qquad \text{for any local 
             $H$-chart $\kappa$ near $x$}.\label{eq-gp-3}
        \end{gather}
    \end{itemize}
\end{definition}

The consistency of the condition~(\ref{eq-gp-3}) is the purpose of the following lemma.

\begin{lemma}
   The condition~(\ref{eq-gp-3}) does not depend on the choice of the local $H$-chart $\kappa$. 
\end{lemma}
\begin{proof}
 Let $\left( (x_{\ell },y_{\ell },t_{\ell })\right)_{\ell \geq 1}\subset \cG^*$ and $(x_{0},\xi_{0})\in GM$ be such that as $\ell \rightarrow \infty$ we have 
        \begin{gather*}
            x_{\ell }\longrightarrow x_{0}, \qquad y_{\ell }\longrightarrow x_{0}, \qquad t_{\ell }\longrightarrow 0,  \qquad \text{and}\\
            t_{\ell }^{-1}\cdot \left(\varepsilon_{\kappa(x_{\ell })}^{\kappa}\circ \kappa\right)(y_{\ell }) \longrightarrow 
            \hat{\kappa}'(x_{0})\xi_{0} \qquad \text{for some local 
             $H$-chart $\kappa$ near $x_{0}$}.
        \end{gather*}
Let $\kappa_{1}$ be another local $H$-chart near $x_{0}$. Let $\phi=\kappa _{1}\circ \kappa^{-1}$ be the corresponding 
transition map with domain $U:=\ran (\kappa) \cap \kappa(\dom (\kappa_{1}))$.  
For 
$\ell$ large enough $\kappa(x_{\ell})$ and $\kappa(y_{\ell})$ are in $U$. Set $\hat{x}_{\ell}=\kappa(x_{\ell})$ and $\hat{y}_{\ell}=t_{\ell }^{-1}\cdot 
\left(\varepsilon_{\kappa(x_{\ell })}^{\kappa}\circ \kappa\right)(y_{\ell })$. Then
\begin{align}\label{eq-gp-4}
    \left(\varepsilon_{\kappa_{1}(x_{\ell})}^{\kappa_{1}}\circ \kappa_{1}\right)(y_{\ell}) & = \varepsilon_{\phi\circ \kappa (x_\ell)}^{\kappa_1}\circ (\kappa_1 \circ \kappa^{-1})\circ 
    \kappa(y_{\ell}) \nonumber\\
    &  = \varepsilon_{\phi(\hat{x}_\ell)}^{\kappa_1}\circ \phi \circ \left(\varepsilon_{\kappa(x_\ell)}^\kappa\right)^{-1} \circ \left( \varepsilon_{\kappa(x_\ell)}^\kappa\circ \kappa\right)(y_\ell) \\
&  =\varepsilon_{\phi(\hat{x}_{\ell})}^{\kappa_1}\circ \phi\circ \left(\varepsilon^{\kappa}_{\hat{x_{\ell}}}\right)^{-1}(t_{\ell}\cdot \hat{y}_{\ell}).\nonumber
\end{align}    

Set $W=\{(x,y)\in U\times \R^n;(\varepsilon^{\kappa}_{x})^{-1}(y)\in U\}$ and $\cU= \{(x,y,t)\in U\times \R^{n}\times \R;(x,t\cdot y)\in W\}$. Let $\Phi:W\rightarrow \R^n$ be the smooth map defined by $\Phi(x,y)= \varepsilon_{\phi(x)}^{\kappa_1}\circ \phi\circ \left(\varepsilon^{\kappa}_{x}\right)^{-1}(y)$ for all $(x,y)\in W$. It follows from Corollary~\ref{cor:Carnot-map.approx-vareps-Carnot-coord} and Lemma~\ref{lem:anisotropic.Theta-parameter} that there is a smooth map $\Theta:\cU\rightarrow \R^n$ such that 
\begin{equation*}
 t^{-1}\cdot \Phi(x,t\cdot y)= \hat\phi'(x)y+ t\Theta(x,y,t) \qquad \text{for all $(x,y,t)\in \cU$, $t\neq 0$}. 
\end{equation*}
Combining this with~(\ref{eq-gp-4}) we see that, as soon as $\ell$ is large enough, we have
\begin{equation*}
 t_\ell^{-1}\cdot \left(\varepsilon_{\kappa_{1}(x_{\ell})}^{\kappa_{1}}\circ \kappa_{1}\right)(y_{\ell})=  t_\ell^{-1}\cdot \Phi(\hat{x}_\ell, t_\ell \cdot \hat{y}_\ell) = 
 \hat\phi'(\hat{x}_\ell)\hat{y}_\ell+ t_\ell\Theta(\hat{x}_\ell,\hat{y}_\ell, t_\ell). 
\end{equation*}
We know that $\hat{x}_{\ell}\rightarrow \kappa(x_{0})$ and $\hat{y}_{\ell}\rightarrow \hat{\kappa}'(x_{0})\xi_{0}$ as 
 $\ell \rightarrow \infty$. Therefore, as $\ell \rightarrow \infty$ we have
 \begin{equation*}
      t_{\ell}^{-1}\cdot \left[ (\varepsilon_{\kappa_{1}(x_{\ell})}^{\kappa_{1}}\circ 
    \kappa_{1})(y_{\ell})\right] \longrightarrow  \hat\phi'\left(\kappa(x_0)\right)\hat{\kappa}'(x_0)\xi_0. 
 \end{equation*}
As $\phi=\kappa_1\circ \kappa^{-1}$ by Proposition~\ref{prop:Carnot.product-tangent-map} we have $\hat\phi'(\kappa(x_0))\hat{\kappa}'(x_0)=\widehat{(\phi\circ \kappa)}'(x_0)=\hat{\kappa}_1'(x_0)$. Thus, 
 \begin{equation*}
      t_{\ell}^{-1}\cdot \left[ (\varepsilon_{\kappa_{1}(x_{\ell})}^{\kappa_{1}}\circ 
    \kappa_{1})(y_{\ell})\right] \longrightarrow \hat{\kappa}_1'(x_0)\xi_0 \qquad \text{as $\ell\rightarrow \infty$}. 
 \end{equation*}
 This shows that if the condition~(\ref{eq-gp-3}) is satisfied by a local $H$-chart near $x_{0}$, then it is satisfied by any 
 other local $H$-chart near $x_{0}$. The lemma is thus proved. 
 \end{proof}

Let $\kappa:V\rightarrow U$ be a local $H$-chart. Note that $V$ (resp., $U$) 
is an open subset of $M$ (resp., $\R^{n}$). Set
 \begin{gather*}
     \cV=GM_{|V} \sqcup \left(V\times V\times \R^*\right), \\
     \cU=\left\{(x,y,t)\in U\times \R^{n}\times \R;\ (\varepsilon_{x}^{\kappa})^{-1}(t\cdot y)\in U\right\}.
 \end{gather*}
We define the map $\gamma_{\kappa}:\cV\rightarrow \cU$ by
\begin{equation}\label{eq-gp-7}
    \begin{array}{cl}
       \gamma_{\kappa}(x,y,t)=(\kappa(x), t^{-1}\cdot \varepsilon_{\kappa(x)}^{\kappa}\circ \kappa(y),t) & \quad 
       \text{for $x,y\in V$ and $t\neq 0$}, \smallskip\\
       \gamma_{\kappa}(x,\xi)=\left(\kappa(x),\hat{\kappa}'(x)\xi, 0\right) & \quad \text{for $x\in V$ and $\xi\in GM(x)$}.
    \end{array}
\end{equation}

\begin{lemma}\label{lem-gp-1} The set $\cV$ is an open subset of $\cG$ and $\gamma_{\kappa}$ is a homeomorphism from 
    $\cV$ onto $\cU$ whose inverse map is given by
        \begin{equation}\label{eq-gp-6}
         \gamma_{\kappa}^{-1}(x,y,t)= \left\{   \begin{array}{ll}
               (\kappa^{-1}(x),(\varepsilon_{x}^{\kappa}\circ \kappa)^{-1}(t\cdot y),t) &  \textup{for  $(x,y,t)\in 
	\cU$ 
	and $t\neq 0$},\smallskip\\
	(\kappa^{-1}(x),\widehat{(\kappa^{-1})}'\!(x)y) & \textup{for $(x,y)\in U\times \R^{n}$ and $t=0$}.
            \end{array}\right.
        \end{equation}    
\end{lemma}
\begin{proof}
    Let us show that $\cV$ is an open subset of $\cG$ by showing that its complement is a closed subset. We have
    \begin{equation*}
        \cG\setminus \cV = GM_{|M\setminus V} \sqcup \left(\cG\setminus \cV\right)^*, \quad \text{where}\  \left(\cG\setminus \cV\right)^*=\left[(M\times 
        M)\setminus(V\times V)\right] \times \R^*. 
    \end{equation*}
  We observe that  $GM_{|M\setminus V}$ is a closed subset of $GM$ and  $(\cG\setminus \cV)^*$ is a closed subset of $M\times 
  M\times\R^*$. In addition, let $\left((x_{\ell},y_{\ell},t_{\ell})\right)_{\ell\geq 0}$ be a sequence in $(\cG\setminus \cV)^*$ that 
  converges in $\cG$ to a point $(x_{0},\xi_{0})$ in $GM$. This implies that $x_{\ell}$ and $y_{\ell}$ both converge to $x_{0}$ in $M$. 
  We observe that, as $(x_{\ell},y_{\ell})$ is in $(M\times 
        M)\setminus(V\times V)$, for every $\ell=1,2, \ldots$, at least one of the points $x_{\ell}$ or $y_{\ell}$ must be contained in 
        $M\setminus V$. Therefore, the subset $M\setminus V$ contains an infinite subsequence of at least one of the 
        sequences $(x_{\ell})_{\ell \geq 0}$ and $(y_{\ell})_{\ell \geq 0}$. As both sequences converge to $x_{0}$ and $M\setminus V$ is closed, we then deduce that $x_{0}$ must be contained in 
        $M\setminus V$, and so $(x_{0},\xi_{0})$ is contained in $GM_{|M\setminus V}$. It follows from this that $\cG\setminus \cV$ is a closed subset of $\cG$, 
        and hence $\cV$ is an open subset of $\cG$. 
    
  It remains to show that $\gamma_{\kappa}$ is a homeomorphism. It is straightforward to check that $\gamma_{\kappa}$ is a 
  bijection whose inverse map is given by~(\ref{eq-gp-6}). Moreover, it is immediate from the definition of the topology of 
  $\cG$ that $\gamma_{\kappa}$ is a continuous map. Therefore, we only need to check that the inverse map 
  $\gamma_{\kappa}^{-1}$ is continuous. It is clear that it is continuous on the open subset $\cU^*:=\{(x,y,t)\in \cU; \ t\neq 0\}$ and on 
  the closed subset $U\times \R^{n}\times\{0\}$. In addition, let $\left((x_{\ell},y_{\ell},t_{\ell})\right)_{\ell\geq 0}$ be a sequence in $\cU^*$ 
  converging to a boundary point $(x_{0},y_{0},0)\in U\times \R^{n}\times \{0\}$. For $\ell=0,1,\ldots$, set $\hat{x}_{\ell}=\kappa^{-1}(x_{\ell})$ and 
  $\hat{y}_{\ell}= (\varepsilon_{x_{\ell}}^{\kappa}\circ \kappa)^{-1}(t_{\ell}\cdot y_{\ell})$, so that 
  $\gamma_{\ell}^{-1}(x_{\ell},y_{\ell},t_{\ell})=(\hat{x}_{\ell},\hat{y}_{\ell},t_{\ell})$. As $\ell \rightarrow \infty$ we have
  \begin{equation*}
      \hat{x}_{\ell}\longrightarrow \kappa^{-1}(x_{0}), \quad \hat{y}_{\ell}\longrightarrow (\varepsilon_{x_{0}}^{\kappa}\circ 
      \kappa)^{-1}(0)=\kappa^{-1}(x_{0}), \quad t_{\ell}\longrightarrow 0.
  \end{equation*}
  Note also that $\left[ (\varepsilon^{\kappa}_{\kappa(\hat{x}_{\ell})}\circ \kappa)(\hat{y}_{\ell})\right]= (\varepsilon^{\kappa}_{x_{\ell}}\circ \kappa)\circ (\varepsilon_{x_{\ell}}^{\kappa}\circ \kappa)^{-1}(t_{\ell}\cdot y_{\ell})=t_{\ell}\cdot y_{\ell}$. Thus, 
  \begin{equation*}
     t_{\ell}^{-1}\cdot \left[ (\varepsilon^{\kappa}_{\kappa(\hat{x}_{\ell})}\circ \kappa)(\hat{y}_{\ell})\right]=y_{\ell}\longrightarrow 
     y_{0}=\hat{\kappa}'(\kappa^{-1}(x_{0}))\left[\widehat{(\kappa^{-1})}'\!(x_{0})y_{0}\right].
  \end{equation*}Therefore, we see that in $\cG$ we have 
  \begin{equation*}
      \gamma_{\kappa}^{-1}(x_{\ell},y_{\ell},t_{\ell})\longrightarrow 
      (\kappa^{-1}(x_{0}),\widehat{(\kappa^{-1})}'(x_{0})y_{0})=\gamma_{\kappa}^{-1}(x_{0},y_{0},0).
  \end{equation*}
  This shows that $\gamma_{\kappa}^{-1}$ is continuous at any boundary point $(x_{0},y_{0},0)\in U\times \R^{n}\times \{0\}$. Therefore, this is a continuous map, and hence
  $\gamma_{\kappa}$ is a homeomorphism. The proof is  complete. 
\end{proof}

\begin{lemma}\label{lem-gp-2}
    Let $\kappa_{1}:V_1\rightarrow U_1$ be another local $H$-chart and
    $\gamma_{\kappa_{1}}:\cV_{1}\rightarrow \cU_{1}$ the homeomorphism~(\ref{eq-gp-7}) associated with 
    $\kappa_{1}$. Then the transition map $\gamma_{\kappa_{1}}\circ \gamma_{\kappa}^{-1}:  
    \gamma_{\kappa}(\cV\cap \cV_{1})\rightarrow \gamma_{\kappa_{1}}(\cV\cap \cV_{1})$ is a 
    smooth diffeomorphism.
\end{lemma}
\begin{proof}
    The inverse map of $\gamma_{\kappa_{1}}\circ \gamma_{\kappa}^{-1}$ is $\gamma_{\kappa}\circ 
    \gamma_{\kappa_{1}}^{-1}: 
    \gamma_{\kappa_{1}}(\cV\cap \cV_{1})\rightarrow \gamma_{\kappa}(\cV\cap \cV_{1})$. Therefore, it is enough to prove that 
    $\gamma_{\kappa_{1}}\circ \gamma_{\kappa}^{-1}$ is a smooth map, since swapping the roles of $\kappa$ and 
    $\kappa_{1}$ would then show that its inverse map is smooth as well. 
    
Let $\phi=\kappa _{1}\circ \kappa^{-1}$ be the corresponding transition map. Its domain is $U_{\phi}:=\kappa(V\cap V_{1})$. 
We also set $\cU_\phi= \gamma_{\kappa}(\cV\cap \cV_{1})$.  
A point $(x,y,t)\in U\times \R^{n}\times \R^*$  is contained in $\cU_\phi$ if and only if the following conditions are satisfied:
\begin{equation*}
    \kappa^{-1}(x)\in V_{1},  \qquad (\varepsilon_{x}^{\kappa})^{-1}(t\cdot y)\in U,  \qquad 
    \kappa^{-1}\circ (\varepsilon_{x}^{\kappa})^{-1}(t\cdot y)\in V_{1}.
\end{equation*}Thus,
\begin{equation*}
    \cU_{\phi}=\left\{ (x,y,t)\in U_{\phi}\times \R^{n}\times \R;\ (\varepsilon_{x}^{\kappa})^{-1}(t\cdot y)\in 
    U_{\phi}\right\}.
\end{equation*}

Let $(x,y,t)\in  \gamma_{\kappa}(\cV\cap\cV_{1})$. If $t\neq 0$, then 
\begin{equation*}
\begin{split}
    \gamma_{\kappa_{1}}\circ \gamma_{\kappa}^{-1}(x,y,t) & = \left( \kappa_{1}\left[\kappa^{-1}(x)\right], t^{-1}\cdot 
    \left[ (\varepsilon^{\kappa_{1}}_{\kappa_{1}\left[\kappa^{-1}(x)\right]}\circ \kappa_{1})\circ 
    (\varepsilon_{x}^{\kappa}\circ \kappa)^{-1}(t\cdot y)\right], t\right)\\
     & = \left( \phi(x),  t^{-1}\cdot \left[(\varepsilon^{\kappa_1}_{\phi(x)}\circ \phi \circ (\varepsilon_{x}^{\kappa})^{-1})(t\cdot 
    y)\right], t\right).
\end{split}
\end{equation*}
If $t=0$, then using Proposition~\ref{prop:Carnot.product-tangent-map} we obtain
\begin{equation*}
   \gamma_{\kappa_{1}}\circ \gamma_{\kappa}^{-1}(x,y,0)= \left(  \kappa_{1}\left[\kappa^{-1}(x)\right], 
   \hat{\kappa}_{1}'(\kappa(x))\widehat{(\kappa^{-1})}'(x)y,0\right)= \left( \phi(x), \hat{\phi}'(x)y,0\right).   
\end{equation*}
In addition, it follows from Corollary~\ref{cor:Carnot-map.approx-vareps-Carnot-coord} and Lemma~\ref{lem:anisotropic.Theta-parameter} that  there is a smooth map $\Theta:\cU_{\phi}\rightarrow \R^{n}$ such that, for all $(x,y,t)\in \cU_\phi$ with $t\neq 0$, we have
\begin{equation*}
    t^{-1}\cdot \left[(\varepsilon^{\kappa_1}_{\phi(x)}\circ \phi \circ (\varepsilon_{x}^{\kappa})^{-1})(t\cdot 
    y)\right]= \hat{\phi}'(x)y+t\Theta(x,y,t). 
\end{equation*}
It follows from all this that, for all $(x,y,t)\in \cU_\phi$, we have
\begin{equation*}
    \gamma_{\kappa_{1}}\circ \gamma_{\kappa}^{-1}(x,y,t)= \left( \phi(x), \hat{\phi}'(x)y+t\Theta(x,y,t),t\right).  
\end{equation*}
This shows that the transition map $ \gamma_{\kappa_{1}}\circ \gamma_{\kappa}^{-1}$ is smooth. The proof is 
complete. 
\end{proof}

Lemma~\ref{lem-gp-1} and Lemma~\ref{lem-gp-2} precisely show that, as $\kappa$ ranges over all the local $H$-charts of $M$, the maps 
$\gamma_{\kappa}$ form a system of local charts near $GM$.  This leads us to the following definition.

\begin{definition}
    The differentiable structure of $\cG=\cG_H M$ is the unique differentiable structure such that
    \begin{itemize}
        \item  The inclusion of $\cG^*=M\times M\times \R^*$ into $\cG$ is a smooth embedding.
    
        \item  A system of local charts near $GM$ is given by the maps $\gamma_{\kappa}$, where $\kappa$ ranges over 
        all local $H$-charts on $M$. 
    \end{itemize}
\end{definition}

\subsection{$\cG_H M$ as a Lie groupoid} 
The above differentiable structure turns $\cG=\cG_H M$ into a smooth manifold  where $\cG^*=M\times 
M\times \R^*$ is an open subset and $GM$ is a hypersurface. We shall now check that with respect to this differentiable structure $\cG_H M$ is a Lie groupoid. 

\begin{lemma}\label{lem:groupoid.unit-map}
    The unit map $\epsilon:\cG^{(0)}\rightarrow \cG$ is a smooth embedding.
\end{lemma}
\begin{proof}
    As the unit map $\epsilon$ is one-to-one we only have to show it is an immersion. Let $\kappa$ be a local $H$-chart with domain 
    $V$ and range $U$ and $\gamma_{\kappa}:\cV\rightarrow \cU$ the local chart~(\ref{eq-gp-7}) associated with 
    $\kappa$. Let $(x,t)\in U\times \R$. If $t\neq 0$, then 
    \begin{equation*}
       ( \gamma_{\kappa}\circ 
        \epsilon)(\kappa^{-1}(x),t)=\gamma_{\kappa}\left(\kappa^{-1}(x),\kappa^{-1}(x),t\right)=\left(x,t^{-1}\cdot 
        \varepsilon_{x}^{\kappa}(x),t\right)=(x,0,t).
    \end{equation*}
    If $t=0$, then $(\gamma_{\kappa}\circ \epsilon)(\kappa^{-1}(x),0)=\gamma_{\kappa}(\kappa^{-1}(x),0)=(x,0,0)$. Thus,
    \begin{equation*}
        (\gamma_{\kappa}\circ\epsilon)(\kappa^{-1}(x),t)=(x,0,t) \qquad \text{for all $(x,t)\in U\times \R$}.
    \end{equation*}
    It follows from this that the unit map $\epsilon$ is an immersion. The proof is complete.
\end{proof}

\begin{lemma}\label{lem:Carnot-groupoid.inverse-map}
    The inverse map $\iota:\cG\rightarrow \cG$ is a diffeomorphism.
\end{lemma}
\begin{proof}
    As the inverse map $\iota:\cG\rightarrow \cG$ is an involution and restricts to a smooth map on the open set $\cG^*=M\times M\times \R^*$, we 
    only have to check it is smooth near the hypersurface $GM$. Let $\kappa$ be a local $H$-chart with domain 
    $V$ and range $U$ and $\gamma_{\kappa}:\cV\rightarrow \cU$ the local chart~(\ref{eq-gp-7}) associated with 
    $\kappa$. In addition, to simplify notation we shall let $U\times \R^{n}\ni (x,y)\rightarrow \varepsilon_{x}(y)$ 
 be its $\varepsilon$-Carnot coordinate map. We observe that $\iota(\cV)=\cV$. Set $\iota_{\kappa}=\gamma_{\kappa}\circ \iota \circ 
    \gamma_{\kappa}^{-1}$; this is a map from $\cU$ to itself.  Let $(x,y,t)\in \cU$. If $t\neq 0$, then 
    \begin{equation}\label{eq-gp-10}
      \iota_{\kappa}(x,y,t)=\gamma_{\kappa}\left( 
        (\varepsilon_{x}\circ \kappa)^{-1}(t\cdot y), \kappa^{-1}(x),t\right)=\left(\varepsilon_{x}^{-1}(t\cdot y), 
        t^{-1}\cdot \varepsilon_{\varepsilon_{x}^{-1}(t\cdot y)}(x),t \right).
    \end{equation}
    If $t=0$, then by using Proposition~\ref{prop:Carnot.inverse-tangent-map} we get
    \begin{equation}\label{eq-gp-11}
      \iota_{\kappa}(x,y,t)=  \gamma_{\kappa}\left( \kappa^{-1}(x),-\widehat{(\kappa^{-1})}'(x)y \right)= \gamma_{\kappa}\left( 
        \kappa^{-1}(x),[\hat{\kappa}'(\kappa^{-1}(x))]^{-1}(-y)\right) = (x,-y).
    \end{equation}
    
   Let $x\in U$. We observe that in the $\varepsilon$-Carnot coordinates at $x$ the $\varepsilon$-Carnot coordinate map is the map $(y,z)\rightarrow \varepsilon_{\varepsilon_x^{-1}(y)} \circ \varepsilon_x^{-1}(z)$. Therefore, by Proposition~\ref{prop-com} near $(y,z)=(0,0)$ we have   
\begin{equation*}
 \varepsilon_{\varepsilon_x^{-1}(y)} \circ \varepsilon_x^{-1}(z)=(-y)\cdot z+ \Ow\left(\|(y,z)\|^{w+1}\right). 
\end{equation*}
 Setting $z=0$ and using the fact that $  \varepsilon_x^{-1}(0)=x$ shows that, for all $x\in U$, we have 
 \begin{equation*}
\varepsilon_{\varepsilon_x^{-1}(y)}(x)=-y+\Ow\left(\|y\|^{w+1}\right) \qquad \text{near $y=0$}.  
\end{equation*}
 Using Lemma~\ref{lem:anisotropic.Theta-parameter} we then deduce there is a smooth map $\Theta:\cU\rightarrow \R^n$ such that 
     \begin{equation*}
    t^{-1}\cdot \varepsilon_{\varepsilon_{x}^{-1}(t\cdot y)}(x)=-y + t\Theta(x,y,t) \quad \text{for all $(x,y,t)\in \cU$, $t\neq 0$}.
    \end{equation*}
    Combining this with~(\ref{eq-gp-10}) and (\ref{eq-gp-11}) we see that,  for all $(x,y,t)\in \cU$, we have
    \begin{equation*}
       \iota_{\kappa}(x,y,t)=  \left(\varepsilon_{x}^{-1}(t\cdot y),  
        -y+t\Theta(x,y,t),t\right).  
    \end{equation*}
   This shows that $ \iota_{\kappa}$ is a smooth map from $\cU$ to itself. This establishes the smoothness of the inverse map $ \iota$ near $GM$. The proof is complete.    
\end{proof}

\begin{lemma}\label{lem-sub}
    The range map $r:\cG\rightarrow \cG^{(0)}$ and the source map $s:\cG\rightarrow \cG^{(0)}$ are submersions.
\end{lemma}
\begin{proof}
    As $s=r\circ \iota$ and we know by Lemma~\ref{lem:Carnot-groupoid.inverse-map} that $\iota$ is a diffeomorphism, we only have to show that the range map $r$ is a submersion. Moreover, it is immediate from its definition that $r$ induce submersions on the open set $\cG^*$. Therefore, we 
    only need to check it is a submersion near the hypersurface $GM$. 
    
    Let $\kappa$ be a local $H$-chart with domain 
    $V$ and range $U$ and let $\gamma_{\kappa}:\cV\rightarrow \cU$ be the local chart~(\ref{eq-gp-7}) associated with 
    $\kappa$.     Let $(x,y,t)\in \cU$. If $t\neq 0$, then 
    \begin{equation*}
     (r\circ \gamma_{\kappa}^{-1})(x,y,t)=r\left(\kappa^{-1}(x), (\varepsilon_{x}^{\kappa}\circ \kappa)^{-1}(t\cdot 
     y),t\right)=(\kappa^{-1}(x),t).
    \end{equation*}
    If $t=0$, then $\gamma_{\kappa}^{-1}(x,y,0)=(\kappa^{-1}(x),\widehat{(\kappa^{-1})}'(x)y)$, and so we have
    \begin{equation*}
      (r\circ \gamma_{\kappa}^{-1})(x,y,0)=(s\circ \gamma_{\kappa}^{-1})(x,y,0)=\left( \kappa^{-1}(x),0\right).  
    \end{equation*}
    Note that $(\varepsilon_{x}^{\kappa}\circ \kappa)^{-1}(0)=\kappa^{-1}\circ 
    (\varepsilon_{x}^{\kappa})^{-1}(0)=\kappa^{-1}(x)$. Thus, 
    \begin{equation*}
      \left((\kappa\otimes \op{id})\circ r\circ \gamma_{\kappa}^{-1}\right)(x,y,t) = (x,t) \qquad \text{for all $(x,y,t)\in \cU$}. 
    \end{equation*}
    It immediately follows from this that the range map $r$ is a submersion near $GM$. The proof is complete. 
\end{proof}

It remains to look at the multiplication map $\mu: \cG^{(2)}\rightarrow \cG$. It follows from Lemma~\ref{lem-sub} that $\cG^{(2)}$ is a submanifold of $\cG\times \cG$. 

\begin{lemma}\label{lem-gp-3}
    The multiplication map $\mu:\cG^{(2)}\rightarrow \cG$ is smooth.
\end{lemma}
\begin{proof}
We know that $\cG^{(2)}=\cG^{(2)*} \sqcup GM^{(2)}$, where $\cG^{(2)*}$ and $GM^{(2)}$ are defined as in~(\ref{eq:Groupoid.cG(2)}) and~(\ref{eq:Groupoid.GM(2)}), respectively. 
    It will be convenient to introduce the following notation: 
    \begin{align*}
        [x,y,z,t]:=\left( (x,y,t), (y,z,t)\right) \in \cG^*\times \cG^* & \quad \text{for $x,y,z\in 
        M$ and $t\neq 0$},\\
        [x,\xi,\eta]:=\left( (x,\xi),(x,\eta)\right)\in GM\times GM  & \quad \text{for $x\in M$ and $\xi,\eta\in 
        GM(x)$}.
    \end{align*}
Using this notation, we have 
\begin{equation*}
    \cG^{(2)*}:=\left\{[x,y,z,t];\ x,y,z\in M, \ t\neq 0\right\}, \qquad GM^{(2)}:= \left\{[x,\xi,\eta]; \ x\in M, \ \xi,\eta\in 
        GM(x)\right\}.
\end{equation*}
Note that $ \cG^{(2)*}$ is an open subset of $\cG^{(2)}$ which is diffeomorphic to $M\times M\times M\times \R^*$, and $GM^{(2)}$ is a hypersurface in $\cG^{(2)}$. Moreover, we have 
\begin{equation*}
 \mu\left([x,y,z,t]\right)=(x,z,t) \qquad \text{for all $x,y,z\in M$ and $t\in \R^*$}. 
\end{equation*}
 This shows that the map $\mu$ is smooth on $\cG^{(2)*}$. Therefore, we only have to check it is smooth near the hypersurface
$GM^{(2)}$. 

Bearing this in mind, a system of local charts near $GM^{(2)}$ is obtained as follows. Let $\kappa:V\rightarrow U$ be a local $H$-chart and $\gamma_{\kappa}:\cV\rightarrow \cU$ the local 
 chart~(\ref{eq-gp-7}) of $\cG$ associated with $\kappa$. In addition, to simplify notation we shall let $U\times \R^{n}\ni (x,y)\rightarrow \varepsilon_{x}(y)$ 
 be its $\varepsilon$-Carnot coordinate map. Set 
 \begin{equation*}
     \cU^{(2)}=\left\{(x,y,z,t)\in U\times \R^{n}\times \R^{n}\times \R; \ \varepsilon_{x}^{-1}(t\cdot y)\in U 
     \ \text{and} \ (\varepsilon_{\varepsilon_{x}^{-1}(t\cdot y)})^{-1}(t\cdot z)\in U\right\}.
 \end{equation*}
We also set $\cV^{(2)}=GM^{(2)}_{|V}\sqcup \cV^{(2)*}$, with
\begin{equation*}
     \cV^{(2)*}:=\left\{[x,y,z,t];\ x,y,z\in V, \ t\neq 0\right\} \quad \text{and} \quad GM^{(2)}_{|V}= 
     \left\{[x,\xi,\eta]; \ x\in V, \ \xi,\eta\in 
        GM(x)\right\}.   
\end{equation*}
We then define the local chart $\gamma_{\kappa}^{(2)}:\cV^{(2)}\rightarrow \cU^{(2)}$ by letting
\begin{equation*}
    \begin{array}{cl}
        \gamma_{\kappa}^{(2)}\left([x,y,z,t]\right)=\left( \kappa(x), t^{-1}\cdot \left[(\varepsilon_{\kappa(x)}\circ 
        \kappa)(y)\right],t^{-1}\cdot \left[(\varepsilon_{\kappa(y)}\circ \kappa)(z)\right], t\right)& \text{for $x,y,z\in 
            V$ and $t\neq 0$},\medskip\\
           \gamma_{\kappa}^{(2)}\left([x,\xi,\eta]\right)=\left( \kappa(x), \hat{\kappa}'(x)\xi,    
           \hat{\kappa}'(x)\eta, 0\right) & \text{for $x\in V$ and $\xi,\eta\in 
            GM(x)$}.
    \end{array}
\end{equation*}
The inverse map $(\gamma_{\kappa}^{(2)})^{-1}:\cU^{(2)}\rightarrow \cV^{(2)}$ is such that, for all $(x,y,z,t)\in 
\cU^{(2)}$, we have
\begin{equation*}
 \left(\gamma_{\kappa}^{(2)}\right)^{-1}(x,y,z,t)=\left\{ 
 \begin{array}{ll}
     \left[\kappa^{-1}(x), (\varepsilon_{x}\circ \kappa)^{-1}(t\cdot y), (\varepsilon_{\varepsilon_{x}^{-1}(t\cdot 
     y)}\circ \kappa)^{-1}(t\cdot z), t\right] & \text{if $t\neq 0$},\smallskip\\
     \left[ \kappa^{-1}(x), \widehat{(\kappa^{-1})}'(x)y, \widehat{(\kappa^{-1})}'(x)z\right] & \text{if $t=0$}.
 \end{array}\right.
\end{equation*}
As $\kappa$ ranges over local $H$-charts, the maps $\gamma_{\kappa}^{(2)}$ form a system of local charts near $GM^{(2)}$. 

Keeping on using the notation above, set $\mu_{\kappa}=\gamma_{\kappa}\circ \mu \circ (\gamma_{\kappa}^{(2)})^{-1}: 
\cU^{(2)}\rightarrow \cU$. Let $(x,y,z,t)\in \cU^{(2)}$. If $t\neq 0$, then we have
\begin{align}\label{eq-gp-15}
    \mu_{\kappa}(x,y,z,t) 
    & = \gamma_{\kappa}\left( \kappa^{-1}(x), (\varepsilon_{\varepsilon_{x}^{-1}(t\cdot y)}\circ 
    \kappa)^{-1}(t\cdot z),t\right) \nonumber \\ 
    & = \left( x, t^{-1}\cdot \left[( \varepsilon_{x}\circ 
    \varepsilon_{\varepsilon_{x}^{-1}(t\cdot y)}^{-1})(t\cdot z)\right],t\right).
\end{align}
If $t=0$, then using Proposition~\ref{prop:Carnot.inverse-tangent-map} we get
\begin{equation}\label{eq-gp-16}
\begin{split}
\mu_{\kappa}(x,y,z,0) & =  \gamma_{\kappa}\left( \kappa^{-1}(x), \left[ \widehat{(\kappa^{-1})}'(x)y\right]\cdot  \left[ 
 \widehat{(\kappa^{-1})}'(x)z\right]\right)\\ 
 & = \gamma_{\kappa}\left( \kappa^{-1}(x),\hat{\kappa}'\left(\kappa^{-1}(x)\right)^{-1}(y\cdot z)\right)\\
 & = (x,y\cdot z,0).
\end{split}
\end{equation}

Let $x\in U$. As mentioned in the proof of Lemma~\ref{lem:Carnot-groupoid.inverse-map} above, in the $\varepsilon$-Carnot coordinates at $x$ the $\varepsilon$-Carnot coordinate map is the map $(y,z)\rightarrow \varepsilon_{\varepsilon_x^{-1}(y)} \circ \varepsilon_x^{-1}(z)$. Therefore, by using Proposition~\ref{prop-com}  we see that, for all $x\in U$, near $(y,z)=(0,0)$,  we have   
\begin{equation*}
\varepsilon_x\circ \varepsilon_{\varepsilon_x^{-1}(y)}^{-1}(z)= \left(\varepsilon_{\varepsilon_x^{-1}(y)} \circ \varepsilon_x^{-1}\right)^{-1}(z)=y\cdot z+ \Ow\left(\|(y,z)\|^{w+1}\right). 
\end{equation*}
Using Lemma~\ref{lem:anisotropic.Theta-parameter} we then deduce there is a smooth map $\Theta:\cU^{(2)} \rightarrow \R^n$ such that, for all $(x,y,z,t)\in \cU^{(2)}$ with $t\neq 0$, we have
\begin{equation*}
 t^{-1}\cdot \varepsilon_x\circ \varepsilon_{\varepsilon_x^{-1}(t\cdot y)}^{-1}(t\cdot z)= y\cdot z + t\Theta(x,y,z,t).  
\end{equation*}
Combining this with~(\ref{eq-gp-15})--(\ref{eq-gp-16}) we see that
\begin{equation*}
    \mu_{\kappa}(x,y,z,t)= \left( x, y\cdot z +t\Theta(x,y,z,t),t\right) \quad \text{for 
    all $(x,y,z,t)\in \cU^{(2)}$}.
\end{equation*}
This shows that $\mu_{\kappa}$ is a smooth map from $\cU^{(2)}$ to $\cU$. As the maps
$\gamma_{\kappa}^{(2)}$ form a system of local charts near $GM^{(2)}$, it then follows that the multiplication map is 
smooth near $GM^{(2)}$. This completes the proof. 
\end{proof}

Combining Lemma~\ref{lem:groupoid.unit-map}, Lemma~\ref{lem-sub}, Lemma~\ref{lem:Carnot-groupoid.inverse-map}, and Lemma~\ref{lem-gp-3} together we arrive at the main result of this section. 

\begin{theorem}\label{thm:Groupoid.main}
 The groupoid $\cG=\cG_{H}M$ is a Lie groupoid.  
\end{theorem}

\begin{remark}
    When $r=1$ the groupoid $\cG_{H}M$ agrees with Connes' tangent groupoid~(\ref{eq:Groupoid.Connes-groupoid}). 
\end{remark}

\begin{remark}
The Lie groupoid $\cG=\cG_H M$ contains the submanifold with boundary $\overline{\cG}^+=\{ t\geq 0\}=\cG^+\sqcup GM$, where $\cG^+=M\times M\times (0,\infty)$. Here the tangent group bundle $GM$ appears as the boundary of $\overline{\cG}^+$. Moreover, $\overline{\cG}^+$ is a groupoid in the category of manifolds with boundary (in the sense of~\cite{NWX:Pacific99}). In the case of Heisenberg manifolds we recover the tangent groupoid of~\cite{Po:PJM06}. We refer to~\cite{vEY:BLMS17} for an alternative construction of the tangent groupoid of a Carnot manifold as a groupoid in the category of manifolds of with boundary (see also~\cite{Jv:Preprint10, vE:AM10.Part1}). 
\end{remark}

In the same way, as with Connes' tangent groupoid, the fact that $\cG_H M$ is a differentiable manifold implies that, for every $x_0\in M$, the submanifolds $\{x_0\}\times M\times \{t\}$ looks more and more like the tangent group $GM(x_0)$ as $t\rightarrow 0$ and this convergence is uniform (and even smooth) with respect to the base point $x_0$. Let $\kappa:V\rightarrow U$ be a local $H$-chart, where $V$ is an open neighborhood of $x_0$. For each $t\in \R^*$, the chart $\gamma_\kappa$ in~(\ref{eq-gp-6}) gives rise to the chart $\tilde{\gamma}_\kappa^t:\{x_0\}\times V\times \{t\}\rightarrow W_t$, where $W_t$ is an open neighborhood of $\kappa(x_0)$ and we have 
\begin{equation*}
 \tilde{\gamma}_\kappa^t (x_0,x,t)= \varepsilon_{\kappa(x_0)}^{-1} \circ \delta_{t^{-1}} \circ (\varepsilon_{\kappa(x_0)} \circ \kappa)(x) \qquad \forall x \in V. 
\end{equation*}
The transition map $ \tilde{\gamma}_\kappa^t\circ ( \tilde{\gamma}_\kappa^1)^{-1}$ is the anisotropic rescaling map $x\rightarrow \varepsilon_{\kappa(x_0)}^{-1}\circ \delta_{t^{-1}} \circ \varepsilon_{\kappa(x_0)}(x)$. Thus, in local coordinates, the submanifold $\{x_0\}\times M\times \{t\}$ is a zoomed-in version's  of $\{x_0\}\times M\times \{1\}\simeq M$, where the zooming is performed by the anisotropic dilation $\delta_{t^{-1}}$ in $\varepsilon$-Carnot coordinates at $x_0$. Furthermore, as with Connes' tangent groupoid, the very fact that we obtain a Lie groupoid accounts for the group structure of the tangent group $GM(x_0)$. Therefore, going back to the conjecture of Bella\"iche alluded to at the beginning of the section, we see that we obtain an explanation why the tangent space of a Carnot manifold should be a group.  

\subsection{Functoriality}
Let us now look at the functoriality of the construction of the tangent groupoid of a Carnot manifold. Let $(M,H)$ 
and $(M',H')$ be Carnot manifolds of step $r$ with respective tangent groupoids $\cG=\cG_{H}M$ and $\cG'=\cG_{H'}M'$. (We do not assume that $M$ and $M'$ have same dimension.) 
Given a Carnot map $\phi:M\rightarrow M'$ we define maps $\Phi:\cG\rightarrow \cG'$ and $\Phi^{0}:\cG^{0}\rightarrow (\cG')^{0}$ by
\begin{align}
\Phi(x,y,t)= \left(\phi(x),\phi(y),t\right) & \qquad \text{for $t\neq 0$ and $x, y\in M$},\label{eq:Groupoid.Phi1}\\
\Phi(x,\xi) = \left(\phi(x), \hat{\phi}'(x)\xi\right) & \qquad \text{for $(x,\xi)\in GM$},\\
\Phi^{0}(x,t)=\left( \phi(x),t\right) & \qquad \text{for $t\in \R$ and $x\in M$}.
\label{eq:Groupoid.Phi3}
\end{align}
The pair $(\Phi,\Phi^{0})$ provides us with a morphism of groupoids in the sense of Remark~\ref{rmk:Groupoid.morphism}. 

\begin{lemma}\label{lem:groupoid.smoothness-Phi}
The maps $\Phi$ and $\Phi^{0}$ are smooth.
\end{lemma}
\begin{proof}
It is immediate that $\Phi^{0}$  is a smooth map between the manifolds $\cG^{0}=M\times \R$ and 
$(\cG')^{0}=M'\times \R$. It is also immediate that $\Phi$ induces a smooth map between the open sets $\cG^*=M\times M\times \R^*$ and 
$(\cG')^*=M\times M\times \R^*$. Therefore, we only have to check that $\Phi$ is smooth near the hypersurface $GM$. 

Given $a\in M$ and setting $a'=\phi(a)$, let $\kappa_{a}:U\rightarrow V$  be a local $H$-chart near $a$ 
and $\kappa_{a'}:U'\rightarrow V'$  be a local $H'$-chart near $a'$. Without any loss of generality we may assume that 
$\phi(V)$ is contained in $V'$. We also let $\gamma_{\kappa_{a}}:\cV\rightarrow \cU$ 
(resp., $\gamma_{\kappa_{a'}}:\cV'\rightarrow \cU'$) be the local chart~(\ref{eq-gp-7}) associated with $\kappa_{a}$ (resp., 
$\kappa_{a'}$). Note that $\cU$ and $\cU'$ are open subsets of $\R^{n}\times \R^{n} \times \R$ and $\R^{n'}\times 
\R^{n'} \times \R$, respectively, where $n$ and $n'$ are the respective dimensions of $M$ and $M'$. To simplify notation 
let $\varepsilon_{x}^{a}:U\times \R^{n}\ni (x,y)\rightarrow \varepsilon_{x}^{a}(y)$ and $\varepsilon_{x}^{a'}: U'\times \R^{n'}\ni (x,y)\rightarrow 
\varepsilon_{x}^{a'}(y)$ be the respective $\varepsilon$-Carnot coordinate map of $\kappa_{a}$ and $\kappa_{a'}$.     

Set $\phi_\kappa=\kappa_{a'}\circ \phi \circ \kappa_a^{-1}:U\rightarrow U'$ and 
$\Phi_{\kappa}=\gamma_{\kappa_{a'}}\circ \Phi \circ \gamma_{\kappa_{a}}^{-1}:\cU\rightarrow \cU'$. Let $(x,y,t)\in 
\cU$. If $t\neq 0$, then we have
 \begin{align}
 \Phi_{\kappa}(x,y,t)& = \gamma_{\kappa_{a'}}\circ \Phi\left( \kappa_{a}^{-1}(x), (\varepsilon_{x}^{a}\circ 
  \kappa_{a})^{-1}(t\cdot y),t\right)\nonumber \\
      & = \gamma_{\kappa_{a'}}\left(\phi\circ\kappa_{a}^{-1}(y), \phi\circ 
      \kappa_{a}^{-1}\circ(\varepsilon_{x}^{a})^{-1}(t\cdot y), t\right)  
      \label{eq:Groupoid.Phik1}\\
      & = \left( \kappa_{a'}\circ \phi\circ\kappa_{a}^{-1}(x),  t^{-1}\cdot [\varepsilon^{a'}_{ \kappa_{a'}\circ \phi\circ\kappa_{a}^{-1}(x)}\circ 
      \kappa_{a'}\circ \phi\circ\kappa_{a}^{-1} \circ (\varepsilon_{x}^{a})^{-1}(t\cdot y)],t\right) \nonumber \\
     & = \left(\phi_\kappa(x),  t^{-1}\cdot [\varepsilon^{a'}_{\phi_\kappa(x)}\circ \phi_\kappa \circ (\varepsilon_{x}^{a})^{-1}(t\cdot y)],t\right). \nonumber
 \end{align}
 If $t=0$, then we have
 \begin{align*}
  \Phi_{\kappa}(x,y,0)     & =\gamma_{\kappa_{a'}}\circ \Phi\left( \kappa_{a}^{-1}(x),\widehat{(\kappa_{a}^{-1})}'(x)y\right)\nonumber \\
      & = \gamma_{\kappa_{a'}}\left( \phi\circ \kappa_{a}^{-1}(x), \hat{\phi}'\left(\kappa_{a}^{-1}(x)\right)\circ  \widehat{(\kappa_{a}^{-1})}'(x)y \right)\\
      & = \left( \kappa_{a'}\circ \phi\circ\kappa_{a}^{-1}(x), 
       \hat{\kappa}_{a'}'\left(  \phi\circ \kappa_{a}^{-1}(x)\right) \circ  \hat{\phi}'\left(\kappa_{a}^{-1}(x)\right)\circ  \widehat{(\kappa_{a}^{-1})}'(x)y, 
      0\right). \nonumber 
 \end{align*}
Proposition~\ref{prop:Carnot.product-tangent-map} ensures us that $ 
\hat{\kappa}_{a'}'\left(  \phi\circ \kappa_{a}^{-1}(x)\right) \circ  \hat{\phi}'\left(\kappa_{a}^{-1}(x)\right)\circ  \widehat{(\kappa_{a}^{-1})}'(x)= \hat{\phi}_\kappa'(x)$. Thus, \begin{equation}
  \Phi_{\kappa}(x,y,0) =  \left( \phi_\kappa(x), \hat{\phi}_\kappa'(x)y, 0\right) \qquad \text{if $t=0$}. 
\label{eq:Groupoid.Phik2}
\end{equation}

In the same way as in the proof of Lemma~\ref{lem-gp-2}, it follows from Corollary~\ref{cor:Carnot-map.approx-vareps-Carnot-coord} and Lemma~\ref{lem:anisotropic.Theta-parameter} that there is a smooth map $\Theta:\cU\rightarrow \cU'$ such that we have 
     \begin{equation*}
        t^{-1}\cdot \left[\varepsilon^{a'}_{\phi_\kappa(x)}\circ \phi_\kappa\circ (\varepsilon_x^a)^{-1}(t\cdot x)\right] 
        = \hat{\phi}_\kappa'(x)y + t\Theta(x,y,t) \qquad \text{for all $(x,y,t)\in \cU$, $t\neq 0$}. 
     \end{equation*}
 Combining  this with~(\ref{eq:Groupoid.Phik1})--(\ref {eq:Groupoid.Phik2}) we see that, for all $(x,y,t)\in \cU$, we have
\begin{equation*}
 \Phi_{\kappa}(x,y,t)= \left( \phi_\kappa(x),  \hat{\phi}_\kappa'(x)y + t\Theta(x,y,t),t\right).   
\end{equation*}
This shows that $\Phi_{\kappa}$ is a smooth map from $\cU$ to $\cU'$. It then follows that $\Phi$ is smooth near  the 
boundary $GM$. The proof is complete.  
\end{proof}

\begin{lemma}\label{lem:groupoid.composition-Phi}
    Let $(M'',H'')$ be a step $r$ Carnot manifold with tangent groupoid $\cG''=\cG_{H''}M''$. Let $\psi:M'\rightarrow 
    M''$ be a (smooth) Carnot manifold map and set $\pi=\psi\circ \phi$. In addition, let $\Psi:\cG'\rightarrow \cG''$ (resp., 
    $\Pi:\cG\rightarrow \cG''$) and $\Psi^{0}:(\cG')^{0}\rightarrow (\cG'')^{0}$ (resp., $\Pi^{0}:\cG^{0}\rightarrow 
    (\cG'')^{0}$) be the maps~(\ref{eq:Groupoid.Phi1})--(\ref{eq:Groupoid.Phi3}) associated with $\Psi$ (resp., $\Pi$). Then we have
    \begin{equation*}
        \Pi=\Psi\circ \Phi \qquad \text{and} \qquad \Pi^{0}=\Psi^{0}\circ \Phi^{0}.
    \end{equation*}
\end{lemma}
\begin{proof}
    It is immediate that $\Pi^{0}=\Psi^{0}\circ \Phi^{0}$ and $ \Pi=\Psi\circ \Phi$ on $\cG^*=M\times M\times 
    \R^*$. In addition, it follows from Proposition~\ref{prop:Carnot.product-tangent-map} that  $ \Pi=\Psi\circ \Phi$ on $GM$. This proves the 
    result.
\end{proof}

Combining Lemma~\ref{lem:groupoid.smoothness-Phi} and Lemma~\ref{lem:groupoid.composition-Phi} we then arrive at the following statement. 
\begin{proposition}
 The pair $(\Phi,\Phi^{0})$ is a morphism of Lie groupoids. 
\end{proposition}

\begin{corollary}\label{cor:Groupoid.functoriality}
    The assignment $(M,H)\rightarrow \cG_{H}M$ is a functor from the category of Carnot manifolds of step $r$ to 
    the category of Lie groupoids. 
\end{corollary}

\begin{remark}
    Much like we can associate a Lie algebra with any Lie group, with any differentiable groupoid $\cG$ is associated 
    a Lie algebroid (see, e.g., \cite{Ma:LMS05}). In particular, this is important in the context of the pseudodifferential 
    calculus on groupoids (see, e.g., \cite{ALN:AnnMath07, DS:AdvM14, LMN:DocMath00, MP:CRAS97, NWX:Pacific99, Va:JFA06}). Recall that a Lie algebroid structure on a vector bundle $A$ over a 
    manifold $M$ is given by 
    \begin{enumerate}
        \item[(i)] A Lie bracket on the space of sections $C^{\infty}(M,A)$.   
    
        \item[(ii)]  An anchor map $\rho:A\rightarrow TM$, i.e., a vector bundle map satisfying Leibniz's rule.
    \end{enumerate}
    For a differentiable groupoid $\cG$ the associated Lie algebroid $A\cG$ is the vertical bundle $\ker (r')$ of the range map restricted to the unit space $M=\cG^{(0)}$, and 
    the anchor map is given by the differential of the source map. In the case of the tangent groupoid $\cG_{H}M$ of a 
    Carnot manifold $(M,H)$ we thus obtain a vector bundle $A \cG_{H}M\rightarrow M\times 
    \R$, where
    \begin{equation*}
        A \cG_{H}M=\fg M \bigsqcup \pi^{*}TM.
    \end{equation*}Here $\fg M$ is the tangent Lie algebra bundle of $(M,H)$ and $\pi^{*}TM$ is the pullback of $TM$ by the first factor projection $M\times \R^*\rightarrow M$. 
\end{remark}

\begin{remark}
   It would be very interesting to extend the construction of the tangent groupoid of a Carnot manifold to the setup of 
   non-equiregular Carnot-Carath\'eodory manifolds. In this case the tangent group at a singular point $a$ is the quotient of 
   the Carnot manifold tangent group $GM(a)$ by a finite group (see, e.g., \cite{Be:Tangent}). Therefore, we should  expect the tangent groupoid to be 
   a manifold with orbifold-type singularities.  
\end{remark}

\end{document}